\documentclass[12pt]{amsart}

\usepackage{amsthm,amsfonts,amssymb,amsmath,oldgerm}
\numberwithin{equation}{section}
\usepackage{amsmath}
\usepackage{setspace}
\usepackage{stmaryrd}
\usepackage{mathrsfs}
\usepackage{fancyhdr}
\usepackage{graphicx}
\usepackage{enumerate}
\usepackage{bm}
\usepackage{bbm}
\usepackage{dsfont}
\usepackage{multirow}
\usepackage{psfrag}
\usepackage[font=scriptsize]{caption}
\usepackage[font=scriptsize]{subcaption}
\usepackage{listings}
\usepackage{tikz}
\usepackage{empheq}
\usepackage{cases}
\usetikzlibrary{matrix} 
\usepackage[framemethod=TikZ]{mdframed}
\mdfdefinestyle{MyFrame}{backgroundcolor=gray!20!white}
\usepackage[makeroom]{cancel}
\usepackage[top=22mm,bottom=22mm,left=26mm,right=26mm,a4paper]{geometry}

\usepackage[english]{babel}
\usepackage[autostyle, english = american]{csquotes}
\MakeOuterQuote{"}

\usepackage{hyperref,bookmark}

\usepackage[normalem]{ulem}
\definecolor{violet}{rgb}{0.580,0.,0.827}

\usepackage[textwidth= \marginparwidth,textsize=tiny]{todonotes}

\hypersetup{
    colorlinks,          
    citecolor=blue,        
    filecolor=blue,      
    urlcolor=blue,           
    linkcolor=blue,
}

\newcommand\dD{\mathrm{d}}

\newcommand{\beq}{\begin{equation}}
\newcommand{\eeq}{\end{equation}}
\newcommand{\beqa}{\begin{eqnarray}}
\newcommand{\eeqa}{\end{eqnarray}}
\newcommand\br{\begin{remark}}
\newcommand\er{\end{remark}}
\newcommand\bp{\begin{pmatrix}}
\newcommand\ep{\end{pmatrix}}
\newcommand{\be}{\begin{equation}}
\newcommand{\ee}{\end{equation}}
\newcommand\ba{\begin{equation}\begin{aligned}}
\newcommand\ea{\end{aligned}\end{equation}}
\newcommand\ds{\displaystyle}

\newcommand{\beg}{\begin{example}}
\newcommand{\eeg}{\end{exaplem}}
\newcommand{\bpr}{\begin{proposition}}
\newcommand{\epr}{\end{proposition}}
\newcommand{\bt}{\begin{theorem}}
\newcommand{\et}{\end{theorem}}
\newcommand{\bc}{\begin{corollary}}
\newcommand{\ec}{\end{corollary}}
\newcommand{\bl}{\begin{lemma}}
\newcommand{\el}{\end{lemma}}
\newcommand{\bd}{\begin{definition}}
\newcommand{\ed}{\end{definition}}
\newcommand{\brs}{\begin{remarks}}
\newcommand{\ers}{\end{remarks}}

\newcommand{\T}{{\mathbb T}}

\newcommand\cA{{\mathcal A}}
\newcommand\cB{{\mathcal B}}
\newcommand\cC{{\mathcal C}}
\newcommand\cD{{\mathcal D}}

\newcommand\cH{{\mathcal H}}

\newcommand\cK{{\mathcal K}}
\newcommand\cL{{\mathcal L}}

\newcommand\cR{{\mathcal R}}
\newcommand\cS{{\mathcal S}}

\newcommand\scC{{\mathscr C}}

\usepackage[none]{hyphenat}
\newtheorem{theorem}{Theorem}[section]
\newtheorem{proposition}[theorem]{Proposition}
\newtheorem{corollary}[theorem]{Corollary}
\newtheorem{definition}[theorem]{Definition}

\newtheorem{remark}[theorem]{Remark}

\newtheorem{lemma}[theorem]{Lemma}
\newcommand{\R}{{\mathbb{R}}}

\newcommand{\N}{{\mathbb{N}}}
\newcommand{\Z}{{\mathbb{Z}}}

\newcommand{\eps}{\varepsilon}

\usepackage{xcolor}

\newcommand{\udiv}{\mbox{div}}

\title[Longtime and chaos in microscopic systems with singular interactions]{Longtime and chaotic dynamics in microscopic systems with singular interactions} 

\author[A. B\'ejar-L\'opez]{Alexis B\'ejar-L\'opez}
\address[Alexis B\'ejar-L\'opez]{\newline Departamento de Matem\'atica Aplicada and Research Unit ``Modeling Nature'' (MNat), Facultad de Ciencias, Universidad de Granada, 18071 Granada, Spain}
\email{alexisbejar@ugr.ess}

\author[A. Blaustein]{Alain Blaustein}
\address[Alain Blaustein]{\newline Dept. of Mathematics, Huck institutes, Pennsylvania State University, University Park, PA 16803, USA} \email{akb7016@psu.edu}
	
\author[P. E. Jabin]{Pierre-Emmanuel Jabin}
\address[Pierre-Emmanuel Jabin]{\newline Dept. of Mathematics, Huck institutes, and Excellence Research Unit ``Modeling Nature'', Pennsylvania State University, University Park, PA 16803, USA} \email{pejabin@psu.edu}

\author[J. Soler]{Juan Soler}
\address[Juan Soler]{\newline Departamento de Matem\'atica Aplicada and Research Unit ``Modeling Nature'' (MNat), Facultad de Ciencias, Universidad de Granada, 18071 Granada, Spain}
\email{jsoler@ugr.es}

\begin{document}



\thanks{\textbf{Acknowledgment.}  This work has been partially supported by the grants: DMS Grant  1908739, 2049020, 2205694 and 2219397 by the NSF (USA);  by the State Research Agency of the Spanish Ministry of Science and FEDER-EU, project PID2022-137228OB-I00 (MICIU/AEI /10.13039/501100011033); by Modeling Nature Research Unit, Grant QUAL21-011 funded by Consejería de Universidad, Investigaci\'on e Innovaci\'on (Junta de Andalucía); and by the Spanish Ministry of Science, Innovation and Universities FPU research grant FPU19/01702 (A. B-L)}

	\begin{abstract}
		This paper investigates the long time dynamics of interacting particle systems subject to singular interactions. We consider a microscopic system of $N$ interacting point particles, where the time evolution of the joint distribution $f_N(t)$ is governed by the Liouville equation. Our primary objective is to analyze the system's behavior over extended time intervals, focusing on the stability, the potential chaotic dynamics and the impact of singularities. In particular, we aim to derive reduced models in the regime $N \gg 1$, exploring both the mean-field approximation and configurations far from chaos, where the mean-field approximation no longer holds. These reduced models do not always emerge  but in these cases we prove that it is possible to derive uniform bounds in $ L^2 $, both over time and with respect to the number of particles, on the marginals $ \left(f_{k,N}\right)_{1\leq k \leq N}$,  irrespective of the initial state’s chaotic nature. Furthermore, we extend previous results by considering a wide range of singular interaction kernels $ K \in W^{\frac{-2}{d+2}, d+2} $ in dimension $d\geq 2$, surpassing the traditional $ L^d $ regularity barriers.  Finally, we address the highly singular case of $ K \in H^{-1} $ under a threshold temperature regime, offering new insights into the behavior of such systems.
	\end{abstract}
 
	\maketitle

	\tableofcontents
	\section{Introduction}
	\label{sec:1}
	\setcounter{equation}{0}
	\setcounter{figure}{0}
	\setcounter{table}{0}
The aim of this paper is to study the long time dynamics of interacting particles systems. Specifically, we consider a microscopic system consisting of $N$ point particles, interacting through singular interactions, 
\begin{equation}\label{System}
\left\{
\begin{array}{ll}
    &\ds\dD X_i=\frac{1}{N} \sum_{\substack{j=1\\ i\neq j}}^N K\left(X_i-X_j\right)\dD t+\sqrt{2\sigma}\,\dD W_i \\[2.5em]
&\ds X_i(t=0)=X_i^0
\end{array}\right.\;\;,\quad\quad \forall i\in\{1,\cdots, N\}\,.
\end{equation}
	We analyze the behavior of these systems over long time intervals, exploring their stability, and the impact of singularities on overall system behavior. Each particle is described through its location  $X_i$ on $\mathbb{T}^d$,  where $\mathbb{T}$ denotes the $1$-dimensional torus of length $|\T|$, and where $d\geq2$ is the dimension. Particle displacement is influenced by two factors: interactions with other particles, modeled by an interaction kernel $ K: \mathbb{T}^d \rightarrow \mathbb{T}^d $, and diffusion with intensity $ \sigma > 0 $, represented by a collection of independent standard Wiener processes $ (W_i)_{1 \leq i \leq N} $ over $ \mathbb{T}^d $. 
	
	We assume the particles are initially indistinguishable, that is:
	\begin{equation}\label{hyp:f0}
	f^0_N\left(x_1,\dots,x_N\right)\,=\,f^0_N\left(x_{\gamma(1)},\dots,x_{\gamma(N)}\right)\,,\quad\quad \forall (x_1,\dots,x_N)\in\mathbb{T}^{dN}\,,
	\end{equation}
	for all permutation of indices $\gamma$, and
	where $f^0_N\,=\,\cL\left(X_1^0,\dots,X_N^0\right)$ is the joint law of the particles at initial time. Property \eqref{hyp:f0} holds for  the joint law of particles $f_N(t,x_1,\dots,x_N)$ at all times $t\geq0$. The time evolution of this joint distribution $f_N(t)$ is governed by the Liouville or forward Kolmogorov equation:
	\begin{equation}\label{Liouville}
		\partial_t f_N+\frac{1}{N} \sum_{\substack{i,j=1\\ i\neq j}}^N  \udiv_{x_i}\left(K\left(x_i-x_j\right) f_N\right)=\sigma \sum_{i=1}^N \Delta_{x_i}f_N\,.
	\end{equation}
The interacting particle system \eqref{System} has broad applicability across a range of disciplines for both types of interactions: first-order interactions, which are the focus of this paper, and second-order (Newtonian) interactions.  Understanding the behavior of such systems is crucial for advancing the study of complex, large-scale systems with many interacting components. For instance, it can model vortex interactions in a two-dimensional fluid \cite{CagliotiLionsMarchioroPulvirenti1995, FournierHaurayMischler2014}, collective motion in biological systems such as microorganisms \cite{GodinhovQuininao2015, FournierJourdain}, interactions of cytonemes in cell communication \cite{aguirre2022predictive}, and more generally, aggregation phenomena \cite{BolleyGuillinFlorent2010, CarrilloDiFrancescoFigalli2011, DegondFrouvelleLiu2013, BreschJabinWang2019a, BreschJabinWang2023}. Additio\-nally, it finds applications in opinion dynamics in populations \cite{Krause2000, Xia2011, Pouradier2019}, optimization problems \cite{Pinnau2017, Carrillo2021, GrassiPareschi2021}, plasma (Coulomb) and astrophysical (Newtonian) interactions \cite{BreschJabinSoler2023, Serfaty2020}, and even in the training of neural networks \cite{Nguyen2018, Rotskoff_Vanden22}. In the applications we just mentioned, the interaction kernels are typically singular, of the order $ |x|^{-\alpha} $, where, for example, $ \alpha = d-1 $ in the cases of plasma or gravitation, as well as in some cases of cell attraction processes (Keller-Segel). Additionally, $ 0 < \alpha < 1 $ with finite range applies in the case of interactions between cytonemes in cellular communication.

In most of these applications, the number $N$ of interacting particles in \eqref{System} is extremely large. For example, in physical plasmas, $N$ can be on the order of $10^{23}$, while in neural interactions, it may reach around $10^8$. For this reason, it is of great interest to identify reduced models for \eqref{System} in the regime $N\gg 1$. A classical approach consists in  proving propagation of chaos which means that the particles' positions become independent as $N\rightarrow +\infty$, provided  that they are so initially. This can be seen on the marginals: for each fixed $k\geq 1$, the marginal $f_{k,N}$ of $f_N$ defined as:
\begin{equation}\label{def:marg}
	f_{k, N}\left(t, x_1, \ldots, x_k\right)= 
	\int_{\mathbb{T}^{d(N-k)} } f_N\left(t, x_1, \ldots, x_N \right) \dD x_{k+1} \ldots \dD x_N,
\end{equation}
converges weakly towards a chaotic, or tensorized, distribution as $N\rightarrow +\infty$
\begin{equation}\label{DefinitionPropagationOfChaos}
	f_{k}(t,X^k)\underset{N\rightarrow +\infty}{\longrightarrow}\bar{f}^{\otimes k}(t,X^k):= \prod_{i=1}^k\bar{f}(t,x_i) \,,
\end{equation}
where, for simplicity, $f_k$ denotes $f_{k,N}$ and $X^k=\left(x_1,\dots,x_k\right)\in\mathbb{T}^{dk}$. 

Propagation of chaos~\eqref{DefinitionPropagationOfChaos} is expected to follow from the mean-field scaling $1/N$ of the interaction term in \eqref{System}. Indeed, when $N\gg 1$, the exact field $\frac{1}{N}\,\sum_{i=1}^N K(X_i-X_j)$ is expected to approach an average or mean field that particles generate as a whole. Hence, the dynamics of the mean-field limit distribution $\bar{f}$ in \eqref{DefinitionPropagationOfChaos} are driven by the following  non linear equation:
\begin{equation}\label{VlasovFokkerPlanck}
	\partial_t\bar{f}+\udiv_x\left(\left(K\star \bar{f}\right)\bar{f}\,\right)\,=\,\sigma\,\Delta_x \bar{f}\,,
\end{equation}
where the mean-field generated by particles is now computed thanks to the convolution product $\star$ between $K$ and the limiting density $\bar{f}$ over $\mathbb{T}^d$. 
In the '$50s$, Kac \cite{Kac1956} pioneered the mathematical study of \textit{propagation of chaos}, which involves proving \eqref{DefinitionPropagationOfChaos} for time $ t > 0 $ given that it holds at the initial time. For a comprehensive introduction to this subject and related developments, we refer the reader to \cite{Hauray2014,Jabin2014,Golse2016}.

We now introduce the literature relevant to our study and refer to \cite{ChaintronDiez2022a, ChaintronDiez2022b} for more comprehensive and exhaustive reviews of recent advances in the field. For regular interaction kernels, such as $ K \in W^{1,\infty} $, the mathematical analysis of the mean-field regime is well established \cite{McKean1967, Neunzert1974, Dobrushin1979, Sznitman1991, Spohn2004}. However, as previously noted, in many applications, the kernel $ K $ lacks such regularity. For first order systems, there has been significant progress on the mean-field limit with singular interaction kernels, at least provided the singularity in $K$ is at the origin. 

The convergence of the 2d point vortices to the incompressible Euler had already been established in  
\cite{CotGooHou, Goodman91, GooHouLow90, HouLowShe93} for deterministic initial positions and in \cite{Scho95, Scho96} for random initial positions. In the stochastic case, the Navier-Stokes equations and propagation of chaos were famously derived as early as \cite{Osada}, with a smallness condition that was removed in \cite{FHM-JEMS}. The mean-field limit for general vortices approximation without $K$ being anti-symmetric was also obtained in~\cite{Hau09}.

One important set of recent results for 1st order systems with singular interactions resolves around the so-called \textit{modulated energy} method,  that leverages the physical properties of the system. This allows to obtain the convergence of solutions to \eqref{System} to the mean-field limit for Riesz and Coulomb kernels, see~\cite{Duerinckx2016} in dimensions 1 and 2, and the seminal~\cite{Serfaty2020} for higher dimension, both without diffusion ($ \sigma = 0 $). 
A relative entropy method was also developed in~\cite{JabinWang2018} and provided quantitative propagation of chaos  with $ W^{-1,\infty} $ kernels in the stochastic case. This result applies to the Biot-Savart kernel on the torus and the 2D vortex model. \cite{BreschJabinWang2019a, BreschJabinWang2020, BreschJabinWang2023} extended these entropy estimates by incorporating the \textit{modulated energy} method to establish quantitative propagation of chaos for kernels with a large smooth part, a small attractive singular part and a large repulsive singular part, as seen in the Patlak-Keller-Segel system under subcritical regimes.

The question of extending these results to uniform-in-time estimates has garnered significant interest due to its wide range of applications, from classical physics problems to emerging developments in machine learning \cite{Andrieu2003, Rotskoff_Vanden22, Qin_Li_Yunan23}. \cite{Guillin2024} advanced this field by deriving uniform-in-time propagation of chaos for divergence-free kernels $ K \in W^{-1,\infty} $, refining the arguments initially presented in \cite{JabinWang2018}. Uniform-in-time convergence to the mean-field limit for Riesz-type interactions was addressed in \cite{RosenzweigSerfaty2021}, and \cite{CourcelRosenzweigSerfaty2023}, building on~\cite{NguyenRosenzweifSerfaty2022}. The issue of uniform-in-time propagation of chaos for attractive logarithmic-type kernels on the torus has been recently addressed in \cite{CourcelRosenzweigSerfaty2023log}. The paper establishes the existence of a critical value $\sigma$ such that, for $\sigma$ exceeding this threshold, a uniform-in-time rate of propagation of chaos is achieved, provided the initial data are sufficiently close to being chaotic. Furthermore, the authors provide a counterexample demonstrating that this result can not be extended to cases with sufficiently small diffusion, emphasizing the necessity of the critical $\sigma$, which will be discussed in Remark \ref{RemarkSizeSigma}.

Recently, a new set of approaches has been developed by taking advantage of diffusion to prove estimates directly on the marginals. This also naturally lead to stronger notions of propagation of chaos where the convergence in \eqref{DefinitionPropagationOfChaos} is established in Lebesgue spaces. \cite{Lacker} first derived relative entropy estimates on the marginal, showing improved rates of convergence to the mean-field limit with the kernel $K$ in the Orlicz space $exp$.  \cite{LackerLeFlem2023} achieved uniform-in-time propagation of chaos for $ L^\infty_{loc} $ kernels with sharp rates as $ N \to +\infty $.  \cite{Han2023} subsequently extended these results to $ L^p $ kernels, with $ p > d $, under a divergence-free constraint. \cite{ChildsRowan2024} even managed to derive global-in-time clustering expansions for the dynamics when $K\in L^\infty$ recovering the seminal results in~\cite{Duer} with more singular kernels but adding diffusion. Unfortunately, it appears difficult to extend those methods to cases where diffusion is degenerate. \cite{BreschJabinSoler2023} still showed propagation of chaos in weighted $ L^p $ spaces for second-order systems with singular interaction kernels, though only on short time intervals.

To conclude this section, we stress that, as seen in the references above, the behavior of particle systems \eqref{System} is much better understood when the initial configurations are close to chaos, meaning that \eqref{DefinitionPropagationOfChaos} holds at $ t = 0 $. This is largely because, in such cases, the dynamics are effectively governed by the mean-field limit. However, the behavior of interacting particle systems that start far from chaos remains an open question and represents a significant challenge in this field.

In this article, we focus on the long-time behavior of the particle system \eqref{System} across a broad range of configurations, including both the mean-field regime \eqref{DefinitionPropagationOfChaos} as well as configurations far from chaos, where the mean-field approximation \eqref{DefinitionPropagationOfChaos}-\eqref{VlasovFokkerPlanck} no longer applies. Specifically, we show that the marginals $ \left(f_{k,N}\right)_{1\leq k \leq N} $ remain uniformly bounded in $ L^2 $, both in time and  with respect to the number of particles, regardless of whether the mean-field approximation \eqref{DefinitionPropagationOfChaos}-\eqref{VlasovFokkerPlanck} is valid. In such situations, these uniform estimates prove to be critical, as we offer a counter example where they remain valid even though uniform in time propagation of chaos fails (see Proposition \ref{prop:counterex}).

Moreover, we consider a broad class of singular interactions, allowing for kernels $ K \in W^{\frac{-2}{d+2}, d+2} $ with negative regularity, thereby extending beyond the $ L^d $ regularity barrier frequently encountered in the literature \cite{Hauray09,Hauray_Jabin15,Han2023} (see a detailed discussion following Theorem \ref{th:1}). In addition, we address the highly singular case where $ K \in H^{-1} $ in the high-temperature regime. A central element of our approach involves establishing a sharp Sobolev inequality on $ \mathbb{T}^{dN} $ for $ N \gg 1 $.

As a result of our findings, we improve the uniform-in-time propagation of chaos from $ L^1 $ to stronger $ L^p $ convergence, employing a straightforward interpolation argument. Specifically, for divergence-free kernels, we extend the global-in-time propagation of chaos in $ L^1 $ established in \cite{JabinWang2018} to uniform-in-time propagation of chaos in $ L^p $, for any $ 1 \leq p < 2 $. This enhancement broadens the scope of the original results, providing stronger control over the convergence properties of the system.

The article is organized as follows. In Section \ref{sec:2}, we present our two main results, which establish uniform-in-time propagation of the $L^2$ norms of the marginals in two distinct scenarios. Theorem \ref{th:1} addresses the case where $ K \in W^{\frac{-2}{d+2},\,d+2}(\mathbb{T}^d) $, while Theorem \ref{th:2} examines more singular kernels in the high-temperature regime, specifically $ K \in H^{-1}(\mathbb{T}^d) $ for sufficiently large $\sigma > 0$. From these results, we derive uniform-in-time propagation of chaos in $L^p$ for any $1 \leq p < 2$ in Corollary \ref{cr:1}. Section \ref{sec:3} is dedicated to the proofs of Theorems \ref{th:1} and \ref{th:2}, while Corollary \ref{cr:1} is established in Section \ref{sec:4}. In Section \ref{Sec:counter:ex}, we prove Proposition \ref{prop:counterex}, demonstrating the failure of uniform-in-time propagation of chaos even in settings with highly regular kernels. The article concludes with two appendices essential for proving Theorems \ref{th:1} and \ref{th:2}. Appendix \ref{Proof:th:sob:ineq} provides a sharp constant for Sobolev's inequality on the torus, while Appendix \ref{App:Interp} examines the spaces that result from interpolating between $ W^{1,\infty}(\mathbb{T}^d) $ and $ L^2(\mathbb{T}^d) $, along with their dual spaces. In Appendix \ref{App:WP}, we rigorously justify the formal computations carried out on Theorem \ref{th:1} and \ref{th:2}.

	\section{Main results}
	\label{sec:2}
	\setcounter{equation}{0}
	\setcounter{figure}{0}
	\setcounter{table}{0}

Let us begin this section with our assumptions on the kernel $K$. We present two different sets of assumptions: first we  consider highly singular kernels $K\in H^{-1}(\mathbb{T}^d)$ that are divergence-free, meaning: 
	\begin{equation}\label{hyp:K:sing}
		K\,=\,\udiv_{x} \phi\quad\textrm{and}\quad
  \udiv_x(K)\,=\,0\,,
	\end{equation}
    for some matrix field $\phi : \mathbb{T}^d \rightarrow \mathbb{T}^{d\times d}$ in $L^2\left(\mathbb{T}^d\right)$. The $H^{-1}$-norm of $K$ is then defined as:
    \[
    \left\|
    K
    \right\|_{H^{-1}\left(\mathbb{T}^d\right)}
    \,:=\,
    \inf_{\phi}\,
    \left\|
    \phi
    \right\|_{L^{2}\left(\mathbb{T}^d\right)},
    \]
    where the infimum is taken over all $\phi$ that satisfy the first relation in \eqref{hyp:K:sing}.
    
    We also consider singular kernels that are not divergence free, where we assume that:
	\begin{equation}\label{K:1}
		\ds K\in W^{-\theta, \frac{2}{\theta}}\left(\mathbb{T}^d\right)\,,\quad \textrm{with}\quad \theta = \frac{2}{d+2}\,,
	\end{equation}
	where the Sobolev space $W^{-\theta, \frac{2}{\theta}}\left(\mathbb{T}^d\right)$ includes all the vector fields $K$ for which exists a matrix potential $\phi$ such that 
	\begin{equation}\label{hyp:K:div}
		K\,=\,\udiv_{x} \phi\,,
	\end{equation}
	with $\phi \in W^{1-\theta, \frac{2}{\theta}}\left(\mathbb{T}^d\right)$, meaning that:
	\[
	\left\|
	\phi
	\right\|_{L^{\frac{2}{\theta}}\left(\mathbb{T}^d\right)}\,+\,
	\left(
	\int_{\mathbb{T}^{2d}}
	\frac{|\phi(y)-\phi(z)|^{\frac{2}{\theta}}}{|y-z|^{d+\frac{2(1-\theta)}{\theta}}}\,\dD y\,\dD z
	\right)^{\frac{\theta}{2}}<+\infty\,.
	\]
	Hence, $W^{-\theta, \frac{2}{\theta}}\left(\mathbb{T}^d\right)$ defines a Banach space under the norm
	\[
	\left\|
	K
	\right\|_{W^{-\theta, \frac{2}{\theta}}\left(\mathbb{T}^d\right)}
	\,:=\,
    \inf_{\phi}\,
    \left[
    \left\|
	\phi
	\right\|_{L^{\frac{2}{\theta}}\left(\mathbb{T}^d\right)}\,+\,
	\left(
	\int_{\mathbb{T}^{2d}}
	\frac{|\phi(y)-\phi(z)|^{\frac{2}{\theta}}}{|y-z|^{d+\frac{2(1-\theta)}{\theta}}}\,\dD y\,\dD z
	\right)^{\frac{\theta}{2}}
 \right],
	\]
    where the infimum is taken over all $\phi$ which satisfy \eqref{hyp:K:div}. 
    
 The second constraint on $ K $ outlines the nature of the interactions we consider. We differentiate between attractive and repulsive interactions through the following decomposition of $ K $:
 \begin{equation}\label{K:2}
K(x)\,=\,K_-(x)+K_+(x)\,,\;\quad\forall\, x \in \mathbb{T}^{d}\,,
 \end{equation}
where $ K_- $ accounts for repulsive interactions, while $ K_+ $ represents attractive interactions. We impose the following assumptions on $ K_+ $ and $ K_- $:
 \begin{subequations}
		\begin{numcases}{}			\label{K:-}
			\left(\udiv_{x} K_-\right)_-\,\in\,L^{\infty}\left(\mathbb{T}^d\right)\;,\\[0.8em]
			\label{K:+}
K_+\,\in\,L^{q}\left(\mathbb{T}^d\right)\;,
		\end{numcases}
	\end{subequations}
 for some $q\geq d$ with $q>2$, where $\left(\cdot\right)_-$ denotes the non-positive part of a real number. Our main results focus on the uniform in time propagation of the $ L^2 $ norms of marginals for various configurations, both near and far from chaos. In the following theorem, we address the case where $ K \in W^{\frac{-2}{d+2},\,d+2}\left(\mathbb{T}^d\right) $.
 \begin{theorem}\label{th:1}
  Assume that the interaction kernel $K$ satisfies \eqref{K:1}-\eqref{K:+}, that we have the exchangeability condition \eqref{hyp:f0} on the sequence of initial data $ \left(f^0_N\right)_{N\geq1} $, along with the following super-exponential growth constraint on the marginals \eqref{def:marg}: there exists a constant $ \beta > 0 $ such that 
 \begin{equation}\label{hyp:f0:growth}
\sup_{2\leq N}\sup_{ k\leq  N}
\frac{\left\|f^0_{k,N}\right\|_{L^2\left(\mathbb{T}^{d k}\right)}
}{k^{\beta k}}
\,<\,+\infty\,.
 \end{equation}
 Then, there exists a unique sequence of weak solutions $\left(f_N\right)_{N\geq 1}$ to \eqref{Liouville} in the sense of \eqref{Liouville:weak}  with initial data $\left(f^0_N\right)_{N\geq1}$ and such that for all $N\geq 1$ it holds
 \[f_N \in L_{loc}^{\infty}\left(\R_+,L^2\left(\T^{dN}\right)\right)\cap L^{2}_{loc}\left(\R_+,H^1(\T^{dN})\right).\]
Furthermore, the sequence of solution has uniformly bounded marginals $\left(f_{k,N}\right)_{1\leq k\leq N}$. Specifically, the following result
 \begin{equation}\label{MarginalsBound}
\sup_{t\in \R^+} 
\sup_{1\leq N}\sup_{ k\leq  N}
\frac{\left\|f_{k,N}(t,\cdot)\right\|_{L^2\left(\mathbb{T}^{d k}\right)}
}{k^{\alpha k}}
\,\leq\,C\,,\quad \textrm{for all} \quad \alpha>\max\left(\beta,d/4\right)\,,
 \end{equation}
holds, for some constant $C$ depending on $d$, $|\mathbb{T}|$, $K$, $\sigma$, $\alpha$ and the implicit constant in \eqref{hyp:f0:growth}. 
 \end{theorem}
  
The uniform estimate \eqref{MarginalsBound} constitutes the cornerstone of Theorem \ref{th:1}. To prove \eqref{MarginalsBound}, we demonstrate an optimal Sobolev inequality on $\mathbb{T}^{dN}$ as $N$ approaches infinity. We leverage this inequality to control the interactions between particles along by the dissipation induced by diffusion on the right-hand side of the Liouville equation \eqref{Liouville}. We postpone the justification of the existence and uniqueness result to Appendix \ref{App:WP} as it not our primary focus.

The primary contribution of Theorem \ref{th:1} is the establishment of uniform estimates concerning both time and the number of particles in strong Lebesgue norms, specifically $L^2$. Moreover, Theorem \ref{th:1} is valid in configurations that are significantly removed from chaos, in addition to the standard chaotic scenarios. Indeed, the super-exponential growth constraint outlined in Assumption \eqref{hyp:f0:growth} extends beyond tensorized or chaotic initial data $f^0_{k,N} = (f^0)^{\otimes k}$, which exhibit at most exponential growth:
\begin{equation*}
	\sup_{2\leq N}\sup_{ k\leq  N}
	\frac{\left\|f^0_{k,N}\right\|_{L^2\left(\mathbb{T}^{d k}\right)}
	}{R^{k}}
	\,<\,+\infty\,,\quad\textrm{for some }\quad R>0\,.
\end{equation*}
Furthermore, Theorem \ref{th:1} applies to a broad class of singular kernels $K \in W^{\frac{-2}{d+2},\,d+2}$ with negative derivatives . Consequently, it encompasses all kernels $K$ in $L^p$ for $p > d$, which are commonly referenced in the literature \cite{Hauray_Jabin15, Han2023, Hauray09}. It is important to note that the case $K \in L^d$ is not directly included in our framework, as the Sobolev embedding of $L^d$ into $W^{\frac{-2}{d+2},\,d+2}$ is not valid. Nonetheless, a minor adjustment in our proof would permit the inclusion of kernels $K$ of the following form:
\[
K\in W^{\frac{-2}{d+2},\,d+2}\left(\mathbb{T}^d\right)\,+\, L^{d}\left(\mathbb{T}^d\right)\,.
\]
For the sake of simplicity, we do not pursue this avenue, but we provide additional details in Remark \ref{case:K:Ld}. 

 We also point out that, under the assumption of Theorem \ref{th:1}, the uniform control over the marginals \eqref{MarginalsBound} is the "best" one can hope for, since uniform in time propagation of chaos fails in general. We support our claim with a detailed counterexample formalized in Proposition \ref{prop:counterex} below.

 In the following result, we extend Theorem \ref{th:1} to highly singular kernels $K$ in $H^{-1}$ within the high-temperature regime, specifically for sufficiently large values of $\sigma$ in \eqref{Liouville}.
\begin{theorem}\label{th:2}
Assume~\eqref{hyp:K:sing} on the kernel $ K $,~ \eqref{hyp:f0} on the initial data $ (f_N^0)_{N \geq 1} $, and the following exponential growth constraint on the marginals \eqref{def:marg} of the initial data:
\begin{equation}\label{SmallnessH-1}
\sup_{2\leq N}\sum_{ k= 1}^N
\frac{\left\|f^0_{k,N}\right\|_{L^2\left(\mathbb{T}^{d k}\right)}^2
}{R^{2k}}
\,<\,C^2\,,
 \end{equation}
 for some positive constants $R$ and $C$. Then, there exists a unique sequence of weak solutions $\left(f_N\right)_{N\geq 1}$ to \eqref{Liouville} in the sense of \eqref{Liouville:weak}  with initial data $\left(f^0_N\right)_{N\geq1}$ and such that for all $N\geq 1$ it holds
 \[f_N \in L_{loc}^{\infty}\left(\R_+,L^2\left(\T^{dN}\right)\right)\cap L^{2}_{loc}\left(\R_+,H^1(\T^{dN})\right).\]
 Furthermore, there exists a constant $ \sigma_0 $ such that, for all diffusion coefficients $ \sigma \geq \sigma_0 $, the solutions $ (f_N)_{N \geq 1} $ exhibits uniformly bounded marginals in $ L^2 $. More specifically, it holds
  \begin{equation}\label{MarginalsBoundH-1}
\sup_{t\in \R^+} 
\sup_{1\leq N}\sup_{ k\leq  N}
\frac{\left\|f_{k,N}(t,\cdot)\right\|_{L^2\left(\mathbb{T}^{d k}\right)}
}{R^{k}}
\,\leq\,C\,.
 \end{equation}
Furthermore, $\sigma_0$ can be explicitly chosen as
\begin{equation}\label{ConditionSigmaH-1}
	 \sigma_0\,=\,R\,\|K\|_{H^{-1}\left(\mathbb{T}^d\right)}. 
\end{equation}
\end{theorem}
The core of our proof is to derive \eqref{MarginalsBoundH-1}. To do so, we control the interaction term on the left-hand side of \eqref{Liouville} using the dissipation from the diffusion operator on the right-hand side of \eqref{Liouville} for sufficiently large coefficient $\sigma$. We postpone the proof of the existence and uniqueness result to Appendix \ref{App:WP} as it not our primary focus.

In Theorem \ref{th:2}, we derive uniform estimates in both time and the number of particles in the strong $L^2$-norm of marginals, under the class of kernels specified by Assumption \eqref{hyp:K:sing}. This result enables us to handle more singular kernels than those considered in Theorem \ref{th:1}. For example, the Biot-Savart and Coulomb kernels in dimension two fall within the scope of that theorem.

Additionally, we note that Theorems \ref{th:1} and \ref{th:2} imply a mean-field limit result similar to those in \cite{Duerinckx2016,Serfaty2020,BreschJabinSoler2023}, without requiring regularity of the solution to the limiting equation \eqref{VlasovFokkerPlanck}. However, here we focus on a stronger result, specifically, strong propagation of chaos (see Corollary \ref{cr:1} below).

An interesting consequence of our results is that we can strengthen uniform-in-time $L^1$ propagation of chaos into stronger $L^p$ convergence using a simple interpolation argument in the case where $K$ is divergence free. For instance, in Corollary \ref{cr:1}, we show uniform--in- -time propagation of chaos as a direct consequence of the uniform estimates established in Theorems \ref{th:1} and \ref{th:2}, combined with the global-in-time propagation of chaos (see \cite{JabinWang2018}, for example). We obtain explicit decay rates in both $N$ and $t \geq 0$, ensuring that the marginal $f_{k,N}$ converges to the tensorized limit $\bar{f}^{\otimes k}$, as defined in \eqref{DefinitionPropagationOfChaos}, as $t \to +\infty$ and $N \to +\infty$ simultaneously.

Moreover, we prove that \eqref{DefinitionPropagationOfChaos} holds in the strong $L^p$-topology for $1 \leq p < 2$, provided that \eqref{DefinitionPropagationOfChaos} is initially satisfied in the weaker entropic sense. This result highlights the robustness of our approach in controlling the chaotic behavior of the system under more stringent conditions, when adding  to \eqref{hyp:f0:growth} or \eqref{SmallnessH-1} the weak entropy assumption:
\begin{equation}\label{hyp:chaos}
\sup_{N\geq 1}{\left[N\, 	\mathcal{H}_N\left(f^0_N|(\bar{f}^0)^{\otimes N}\right)\right]}\,<\,\infty\,,
\end{equation}
for some initial distribution $\bar{f}^0$, where the relative entropy $\cH_N$ is defined  as follows:
\begin{equation*}
	\mathcal{H}_N(f|g)=\frac{1}{N}\int_{\mathbb{T}^{dN} 
	} f(X^N)\log{\left(\frac{f(X^N)}{g(X^N)}\right)}\dD X^N\,,
\end{equation*}
for any two positive functions $(f,g)\in L^1(\mathbb{T}^{dN})$ with integral one.

\begin{corollary}[Uniform propagation of chaos in $L^p$]\label{cr:1} Consider an interaction kernel $K\in W^{-1,\infty}(\mathbb{T}^d)$ and a sequence of initial data $\left(f^0_N\right)_{N\geq1}$ satisfying \eqref{hyp:chaos}. Furthermore,
	assume  either the assumptions of Theorem \ref{th:1}  with  $\udiv_x (K)=0$ \underline{or}
	the assumptions of Theorem \ref{th:2}, and consider a solution $\bar{f}$ to equation \eqref{VlasovFokkerPlanck} with some initial data $\bar{f}^0$ satisfying:
	\begin{equation}\label{f0Regularity}
		\bar{f}^0\in C^{\infty}(\mathbb{T}^d)\,,\quad
		\inf_{x \in \mathbb{T}^d} \bar{f}^0(x) > 0 \,,\quad\textrm{and}\quad
		\int_{\mathbb{T}^d}\bar{f}^0(x)\,\dD x \,=\,1\,.
	\end{equation}
	Then, each finite marginal $f_{k,N}$ converges to $\bar{f}^{\otimes k}$ uniformly in time in $L^p$, for all $1\leq p<2 $, as $N\rightarrow +\infty$. More precisely, for all $(k,N)\in\N^2$, with $1\leq k\leq N$, and all time $t\geq 0$, the marginals $(f_{k,N})_{1\leq k\leq N}$ satisfy:
	\begin{equation*}
		\|f_{k,N}(t,\cdot)-\bar{f}^{\otimes k}(t,\cdot)\|_{L^{p}(\mathbb{T}^{dk})}\leq X_k^\frac{2(p-1)}{p}\left(\frac{C\sqrt{k}e^{-\beta t}}{N^\gamma}\right)^{\frac{2-p}{p}},
	\end{equation*}
	for some positive constants $C, \beta, \gamma$ which only depend on $K$, $d$ , $\sigma$, the implicit constant in \eqref{MarginalsBound} (resp. \eqref{MarginalsBoundH-1}) and \eqref{hyp:chaos}, and the norms of $\bar{f}_0$. Furthermore,  $X_k\,=\,C k^{\alpha k}+\|\bar{f}^0\|_{L^2(\mathbb{T}^{dk})}^k$, under the assumptions of Theorem \ref{th:1}, and $
	X_k\,=\,C R^k+\|\bar{f}^0\|_{L^2(\mathbb{T}^{dk})}^{k}$, under the assumptions of Theorem \ref{th:2}.
\end{corollary}

The key idea in our proof is that for divergence-free kernels $K$, standard relative entropy estimates guarantee that the marginal $f_{k,N}$ and the tensorized limit $\bar{f}^{\otimes k}$ converge to the same stationary state as $t \to +\infty$. When combined with the global-in-time propagation of chaos result from \cite[Theorem 1]{JabinWang2018}, this ensures that $f_{k,N} \to \bar{f}^{\otimes k}$ as $N \to +\infty$ and $t \to +\infty$ simultaneously in $L^1$.
We then use an interpolation argument to strengthen this $L^1$-convergence. By interpolating between the $L^1$-convergence and our uniform $L^2$-estimates from Theorems \ref{th:1} and \ref{th:2}, we upgrade the convergence from $L^1$ to $L^p$ for all $1 \leq p < 2$. This interpolation approach allows us to control the chaotic behavior of the system in stronger norms, reinforcing the convergence properties.

\begin{remark}
    The assumptions regarding the regularity of the initial data $\bar{f}_0$ could be relaxed, but doing so would introduce additional technical complications that do not contribute significantly to the main objectives of this paper. For the sake of clarity and focus, we will therefore retain the regularity assumptions in \eqref{f0Regularity}.
\end{remark}

To conclude this section, we expand on the cases where uniform in time propagation of chaos \eqref{DefinitionPropagationOfChaos} fails and where the particle system \eqref{System} is not governed by its mean field limit anymore after enough time elapsed. This shows that, under the assumption of Theorem \ref{th:1}, uniform control over the marginals is optimal, in the sense that one cannot expect uniform in time propagation of chaos. In the following proposition, we formalize a counter example which meets the assumption of Theorem \ref{th:1} and where uniform in time propagation of chaos fails.
\begin{proposition}\label{prop:counterex}
Fix the dimension to $d=1$ and suppose that $K$ is given by the Kuramoto kernel:
\[
K(x)\,=\,-\sin(x)\,,\quad\forall\, x \in \T\,.
\]
For $\sigma>0$ small enough, there exists a probability distribution $\bar{f}^0\in \scC^2\left(\T\right)$ such that the solution $\bar{f}(t)$ to \eqref{VlasovFokkerPlanck} with initial condition $\bar{f}^0$ and the solutions $(f_N(t))_{N\geq 2}$ to \eqref{Liouville} with the following chaotic initial configurations:
\[
f^0_N\,=\,
\left(\bar{f}^{0}\right)^{\otimes N}
\] 
do not satisfy uniform in time propagation of chaos in the following sense:
\begin{equation*}
\limsup_{|(t,N)|\rightarrow +\infty}
	\|f_{1,N}(t,\cdot)-\bar{f}(t,\cdot)\|_{L^1(\mathbb{T})}>0\,.
\end{equation*}
\end{proposition}
To prove this result, we use that, with the Kuramoto kernel, there exist two distinct stationary states to the limiting equation \eqref{VlasovFokkerPlanck}  for $\sigma>0$ small enough \cite[Theorem $4.1$]{Ha_Shim_Zhang20} whereas the Liouville equation admits a unique stable equilibrium. Stability is obtained thanks to a logarithmic Sobolev inequality demonstrated in \cite[Lemma $2$]{Guillin2024}. We postpone the proof to Section \ref{Sec:counter:ex}.
	\section{Uniform \texorpdfstring{$L^2$}{L2}-estimates}
	\label{sec:3}
	\setcounter{equation}{0}
	\setcounter{figure}{0}
	\setcounter{table}{0}
	In this section, we establish uniform-in-time and uniform-in-number-of-particles $L^2$ -estimates for the marginals $(f_{k,N})_{1 \leq k \leq N}$. Our proof relies on the analysis of the BBGKY hierarchy satisfied by the marginals, which is derived by integrating \eqref{Liouville} with respect to $(x_{k+1}, \dots, x_N)$ and leveraging the exchangeability condition \eqref{hyp:f0}:
	\begin{equation}\label{BBGKY}
		\begin{split}
			\partial_t f_{k,N}&+\frac{1}{N}\sum_{\substack{i,j=1\\ i\neq j}}^k \udiv_{x_i}( K(x_i-x_j) f_{k,N})\\
			&+ \frac{N-k}{N}\sum_{i=1}^k \udiv_{x_i}\left(\int_{\mathbb{T}^d} K(x_i-x_{k+1})f_{k+1,N}\dD x_{k+1} \right)=\sigma \sum_{i=1}^k\Delta_{x_i}f_{k,N}\,,
		\end{split}
	\end{equation}
	for all $N\geq1$ and $k\in\{1,\dots,N\}$, with $f_{N+1,N}=0$. This approach allows us to handle the complexity of particle interactions systematically and derive the desired estimates. The main challenge is estimating the terms in \eqref{BBGKY} that involve the higher-order marginal $ f_{k+1,N} $:
	\begin{equation}\label{hierarchy:term}
		\int_{\mathbb{T}^{d}} K(x_i-x_{k+1})f_{k+1,N}(t,x_1,\dots,x_{k+1})\,\dD x_{k+1}\,,\quad i\in\left\{1,\cdots,k\right\}\,.
	\end{equation}
For a chaotic marginal $f_{k+1}$, that is $f_{k+1}=F^{\otimes (k+1)}$, a naive estimate of the interaction term in \eqref{hierarchy:term} would yield: 
\begin{equation}\label{naive:estim}
	\left\|
	\int_{\mathbb{T}^d} K(x_i-x_{k+1})f_{k+1}(t,x_1,\dots,x_{k+1})\,\dD x_{k+1}
	\right\|
	\lesssim \left\| f_{k+1}\right\|\lesssim
	\left\| F\right\|^{k+1}
	\lesssim
	\left\| f_{k}\right\|^{1+\frac{1}{k}}
	\,,
\end{equation}
which introduces a nonlinear dependence on $ f_k $ in equation \eqref{BBGKY}. \\
To address this issue, we propose two different methods. In Section \ref{sec:K:H-1}, we handle the case where $ K \in H^{-1} $ with a sufficiently large $\sigma > 0$, and in Section \ref{sec:W:theta}, we treat the case where $ K \in W^{\frac{-2}{d+2}, d+2} $, without any constraint on $\sigma > 0$. These approaches allow us to control the higher-order terms and avoid the blow-up scenario. The computations presented in Sections \ref{sec:K:H-1} and \ref{sec:W:theta} are formal, as we deal with weak solutions to the Liouville equation \eqref{Liouville}. However, we justify these formal computations in Appendix \ref{App:WP}.

In both cases, we leverage the dissipation induced by the diffusion on the right-hand side of \eqref{BBGKY} to control the term \eqref{hierarchy:term}. We begin with the case $K \in H^{-1}$, which is less technically demanding from a mathematical perspective but encapsulates the core idea of our approach.

\subsection{Proof of Theorem \ref{th:2}: the case \texorpdfstring{$K\in H^{-1}$}{KH-1}}\label{sec:K:H-1} In this section, we establish uniform-in-time and uniform-in-number-of-particles $ L^2 $-estimates for the marginals of the system $ (f_{k,N})_{1 \leq k \leq N} $ under the assumption that the interaction kernel $ K $ in \eqref{System} belongs to $ H^{-1} $. The main challenge is estimating \eqref{hierarchy:term}, as it involves the higher-order marginal $f_{k+1}$ in the equation \eqref{BBGKY}. We demonstrate that when $K \in H^{-1}$, it is feasible to control \eqref{hierarchy:term} using the dissipation induced by the diffusion on the right-hand side of \eqref{BBGKY} for sufficiently large values of $\sigma > 0$.

Let us outline our strategy in the case of a chaotic marginal $f_{k+1}$, specifically when $f_{k+1} = F^{\otimes (k+1)}$. We establish that the dissipation of the $L^2$-norm induced by the Laplace operator, denoted as $\mathcal{C}$ in our proof below, satisfies the following properties:
\begin{equation*}
	\left|\cC\right|^{\frac{1}{2}}
	\sim
	\sigma^{\frac{1}{2}}\left\| f_{k}\right\|
	\,.
\end{equation*}
Hence, our approach boils down to taking $\sigma^{\frac{1}{2}} \geq \left\| F\right\|$ , which yields:
\begin{equation*}
	\left|\cC\right|^{\frac{1}{2}}\,\sim
	\left\| F\right\|\left\| f_{k}\right\|
	\sim
	\left\| f_{k}\right\|^{1+\frac{1}{k}}
	\,,
\end{equation*}
and allows to compensate the right hand side in \eqref{naive:estim}.
\begin{proof}[Proof of Theorem \ref{th:2}] 
	We fix $t\geq0$ and $(k,N)\in\left(\N^\star\right)^2$ such that $1\leq k\leq N$. To estimate the $L^2$-norm of $f_k$, we compute its time derivative by multiplying equation \eqref{BBGKY} by $f_k$ and integrating over $\mathbb{T}^{dk}$. This yields the following result:
	\begin{equation*}
		\frac{1}{2}\frac{\dD}{\dD t}  \left\|f_k(t,\cdot)\right\|_{L^2(\mathbb{T}^{dk})}^2
		\,=\;\cA\,+\,\cB\,+\,\cC\,,
	\end{equation*}
	where $\cA$, $\cB$ and $\cC$ are given by
	\[
	\left\{\hspace{-0,4cm}
	\begin{array}{ll}
		&\ds\cA\,=\,-\,\frac{1}{N}\,\sum_{\substack{i,j=1\\ i\neq j}}^k\int_{\mathbb{T}^{dk}}   \udiv_{x_i}\left(K(x_i-x_j) f_{k}\left(t,X^k\right)\right)f_{k}\left(t,X^k\right)\,\dD X^k\,,  \\[1.em]
		&\ds\cB\,=\, -\,\frac{N-k}{N} \,\sum_{i=1}^k \int_{\mathbb{T}^{dk}}\udiv_{x_i}\left( \int_{\mathbb{T}^d} K(x_i-x_{k+1}) f_{k+1}\left(t,X^{k+1}\right) \dD x_{k+1}\right) f_{k}(t,X^{k})\,\dD X^k \,,\\[1.em]
		&\ds\cC\,=\, \sigma\,\sum_{i=1}^k\int_{\mathbb{T}^{dk}} f_k\left(t,X^k\right)\Delta_{x_i} f_{k}\left(t,X^k\right)\,\dD X^k\,. 
	\end{array}
	\right.
	\]
	First, we integrate by part with respect to $x_i$, $i\in\{1,\dots,k\}$, in $\cA$, $\cB$ and $\cC$, which yields:
	\[
	\left\{\hspace{-0,4cm}
	\begin{array}{ll}
		&\ds\cA\,=\,\frac{1}{N}\,\sum_{\substack{i,j=1\\ i\neq j}}^k\int_{\mathbb{T}^{dk}}   K(x_i-x_j)\cdot f_{k}(t,X^k)\nabla_{x_i} f_{k}(t,X^k)\,\dD X^k\,,  \\[1.em]
		&\ds\cB\,=\, \frac{N-k}{N} \sum_{i=1}^k \int_{\mathbb{T}^{dk}}\left( \int_{\mathbb{T}^d} K(x_i-x_{k+1}) f_{k+1}(t,X^{k+1})\, \dD x_{k+1}\right)  \cdot \nabla_{x_i} f_{k}(t,X^{k})\,\dD X^k \,,\\[1.5em]
		&\ds\cC\,=\,- \,\sigma\left\|\nabla_{X^k} f_k(t,\cdot)\right\|^2_{L^{2}\left(\mathbb{T}^{dk}\right)}\,\leq\,0\,. 
	\end{array}
	\right.
	\]
On one hand, we demonstrate that $\mathcal{A}$ vanishes because $K$ is divergence-free, as stated in \eqref{hyp:K:sing}. Conversely, the term $\mathcal{C}$ captures the contribution of diffusion on the right-hand side of \eqref{BBGKY}. This term $\mathcal{C}$ has a definite sign, which we leverage to control the primary contribution $\mathcal{B}$ that depends on $f_{k+1}$.
	
	As mentioned above, $\mathcal{A}$ vanishes due to the divergence free assumption \eqref{hyp:K:sing} on $K$. Indeed, using the relation $f_{k}\nabla_{x_i} f_{k}\,=\, \nabla_{x_i}\left| f_{k}\right|^2/2$ and integrating by parts with respect to $x_i$ in $\cA$, we obtain:
	\[
	\cA\,=\,\frac{1}{2N}\,\sum_{\substack{i,j=1\\ i\neq j}}^k\int_{\mathbb{T}^{dk}}   K(x_i-x_j)\cdot \nabla_{x_i}\left| f_{k}\right|^2(t,X^k) \,\dD X^k\,=\,0\,.
	\]
	Let us estimate $\cB$. First, we point out that $\cB$ vanishes when $k=N$. In the cases $k\leq N-1$, we apply Cauchy-Schwarz inequality and find:
	\[
	\cB\,\leq\, \sum_{i=1}^k \left(\int_{\mathbb{T}^{dk}}\left| \int_{\mathbb{T}^d} K(x_i-x_{k+1}) f_{k+1}(t,X^{k+1}) \dD x_{k+1}\right|^2\dD X^k\right)^{\frac{1}{2}}
	\left\|\nabla_{x_i} f_{k}(t,\cdot) \right\|_{L^2\left(\mathbb{T}^{dk}\right)}.
	\]
	Since, particles are indistinguishable according to \eqref{hyp:f0}, it holds:
	\[\left\|\nabla_{x_i} f_{k}(t,\cdot) \right\|_{L^2\left(\mathbb{T}^{dk}\right)} = \frac{1}{k^{\frac{1}{2}}}\left\|\nabla_{X^k} f_{k}(t,\cdot) \right\|_{L^2\left(\mathbb{T}^{dk}\right)}\,,\quad\forall i\in\{1,\dots,k\}\;.\]
	Therefore, we find the following estimate for $\cB$:
	\begin{equation}\label{estim:B:sec:theta}
		\cB\,\leq\,
		\frac{1}{k^{\frac{1}{2}}} \left\|\nabla_{X^k} f_{k}(t,\cdot) \right\|_{L^2\left(\mathbb{T}^{dk}\right)}\sum_{i=1}^k \left(\int_{\mathbb{T}^{dk}}\left| \int_{\mathbb{T}^d} K(x_i-x_{k+1}) f_{k+1}(t,X^{k+1})\, \dD x_{k+1}\right|^2\dD X^k\right)^{\frac{1}{2}}
		.
	\end{equation}
	To bound the integral in the latter estimate, we replace $K$ with any $\phi$ satisfying $K=\udiv_x (\phi)$, which yields:
	\begin{equation*}
		\cB\leq
		\frac{1}{k^{\frac{1}{2}}} \left\|\nabla_{X^k} f_{k}(t,\cdot) \right\|_{L^2\left(\mathbb{T}^{dk}\right)}\hspace{-0.05cm}\sum_{i=1}^k \hspace{-0.05cm}\left(\int_{\mathbb{T}^{dk}}\left| \int_{\mathbb{T}^d} \udiv_x \phi(x_i-x_{k+1}) f_{k+1}(t,X^{k+1}) \dD x_{k+1}\right|^2 \dD X^k \hspace{-0.1cm} \right)^{\frac{1}{2}}\hspace{-0.1cm}
		.
	\end{equation*}
	 Then, we integrate by parts with respect to variable $x_{k+1}$ and apply the Cauchy-Schwarz inequality, which yields
	\begin{equation*}
		\cB\,\leq\,
		\frac{1}{k^{\frac{1}{2}}} \left\|\nabla_{X^k} f_{k}(t,\cdot) \right\|_{L^2\left(\mathbb{T}^{dk}\right)}\sum_{i=1}^k\left\|\phi\right\|_{L^{2}(\mathbb{T}^d)} \|\nabla_{x_{k+1}}f_{k+1}(t,\cdot)\|_{L^2\left(\mathbb{T}^{d(k+1)}\right)}
		.
	\end{equation*}
	Next, we take the infimum over all $\phi$ satisfying $K\,=\,\udiv_x (\phi)$ and sum over $i$ between $1$ and $k$, which yields:
	\begin{equation*}
		\cB\,\leq\,
		\frac{1}{k^{\frac{1}{2}}} \left\|\nabla_{X^k} f_{k}(t,\cdot) \right\|_{L^2\left(\mathbb{T}^{dk}\right)} k \left\|K\right\|_{H^{-1}(\mathbb{T}^d)}  \|\nabla_{x_{k+1}}f_{k+1}(t,\cdot)\|_{L^2\left(\mathbb{T}^{d(k+1)}\right)}
		.
	\end{equation*}
	We estimate the latter right hand side thanks to the relation
	 $\|\nabla_{x_{k+1}}f_{k+1}(t,\cdot)\|_{L^2} =\|\nabla_{X^{k+1}}f_{k+1}(t,\cdot)\|_{L^2}/\sqrt{k+1}$, which holds according to  \eqref{hyp:f0}. This yields:
	\begin{equation*}
		\mathcal{B}\,\leq\, \left\|K\right\|_{H^{-1}(\mathbb{T}^d)}  \|\nabla_{X^{k+1}}f_{k+1}(t,\cdot)\|_{L^2(\mathbb{T}^{d(k+1)})}\|\nabla_{X^{k}}f_{k}(t,\cdot)\|_{L^2\left(\mathbb{T}^{dk}\right)}.
	\end{equation*}
	Gathering our estimates on $\cA$ and $\cB$, we find the following differential inequality:
	\begin{equation*}
		\begin{split}
		\frac{1}{2}\frac{\dD}{\dD t}  \left\|f_k(t,\cdot)\right\|_{L^{2}(\mathbb{T}^{dk})}^2
		\leq&\left\|K\right\|_{H^{-1}(\mathbb{T}^{d})}  \|\nabla_{X^{k+1}}f_{k+1}(t,\cdot)\|_{L^2(\mathbb{T}^{d(k+1)})} \|\nabla_{X^{k}}f_{k}(t,\cdot)\|_{L^2(\mathbb{T}^{dk})}\\
		&- \,\sigma\left\|\nabla_{X^k} f_k(t,\cdot)\right\|^2_{L^{2}(\mathbb{T}^{dk})}.
		\end{split}
	\end{equation*}
	We apply Young inequality to estimate the product in the latter right hand side and obtain:
	\begin{equation*}
		\frac{1}{2}\frac{\dD}{\dD t}  \left\|f_k(t,\cdot)\right\|_{L^{2}(\mathbb{T}^{dk})}^2
		\leq\frac{\left\|K\right\|^2_{H^{-1}(\mathbb{T}^{dk})}}{2\,\sigma}  \|\nabla_{X^{k+1}}f_{k+1}(t,\cdot)\|^2_{L^2(\mathbb{T}^{d(k+1)})} 
		- \,\frac{\sigma}{2}\left\|\nabla_{X^k} f_k(t,\cdot)\right\|^2_{L^{2}(\mathbb{T}^{dk})}.
	\end{equation*}
	Our strategy is to compensate for the higher-order norm of $f_{k+1}$ in the latter estimate with the dissipation term resulting from diffusion. To achieve this, we divide the latter estimate by $R^{2k}$, where $R > 0$ is specified in \eqref{SmallnessH-1}, and then take the sum for $k$ ranging from $1$ to $N$. After re-indexing, this yields:
	\begin{equation*}
		\frac{\dD}{\dD t} \sum_{k=1}^N\frac{\|f_k(t,\cdot)\|_{L^2(\mathbb{T}^{dk})}^2}{R^{2k}}
		\leq 
		\sum_{k=1}^{N} \left(\frac{\|K\|^2_{H^{-1}(\mathbb{T}^{d})}}{\sigma R^{2(k-1)}}-\frac{\sigma}{R^{2k}}\right)\|\nabla_{X^k} f_k(t,\cdot)\|^2_{L^2(\mathbb{T}^{dk})}
		\,.
	\end{equation*}
	Since $\sigma \geq\sigma_0$, where $\sigma_0=R\,\|K\|_{H^{-1}\left(\mathbb{T}^d\right)}$ is given in Theorem \ref{th:2}, we find:
	\begin{equation*}
		\frac{\dD}{\dD t} \sum_{k=1}^N\frac{\|f_k(t,\cdot)\|_{L^2(\mathbb{T}^{dk})}^2}{R^{2k}}\leq  0\,.
	\end{equation*}
	Then, we integrate the latter estimate between $0$ and $t$, which yields:
	\begin{equation*}
		\sum_{k=1}^N\frac{\|f_{k,N}(t,\cdot)\|_{L^2(\mathbb{T}^{dk})}^2}{R^{2k}}
		\,\leq\,
		\sum_{k=1}^N\frac{\|f_{k,N}^0\|_{L^2(\mathbb{T}^{dk})}^2}{R^{2k}} \,,
	\end{equation*}
	for all time $t>0$. Next, we bound the latter right hand side thanks to assumption \eqref{SmallnessH-1}, which yields
	\begin{equation*}
		\sum_{k=1}^N\frac{\|f_{k,N}(t,\cdot)\|_{L^2(\mathbb{T}^{dk})}^2}{R^{2k}}
		\,\leq\,C^2\,,
	\end{equation*}
	for all time $t\geq 0$ and where $C$ is given in \eqref{SmallnessH-1}. To conclude our proof, we fix some $k$ between $1$ and $N$, lower bound the latter left hand side by $\|f_{k,N}(t,\cdot)\|_{L^2(\mathbb{T}^{dk})}^2/R^{2k}$,  and take the square root on both sides of the inequality, which yields the expected result
	\begin{equation*}
		\frac{\|f_{k,N}(t,\cdot)\|_{L^2(\mathbb{T}^{dk})}}{R^{k}}
		\,\leq\,C\,, \quad \forall t\in \R^+\,,\quad \forall k \in \left\{1,\dots, N\right\}\,.
	\end{equation*}
\end{proof}
\begin{remark}\label{RemarkSizeSigma}
	Although this result is not directly comparable in terms of techniques or objectives to \cite[Theorem 4]{CourcelRosenzweigSerfaty2023log}, when we restrict our analysis to the specific case of the attractive logarithmic kernel on $\mathbb{T}^2$, the condition \eqref{ConditionSigmaH-1} becomes: 
	\[
		\sigma_0\geq 0.18748\cdot R.
	\]
	In the same context of the attractive logarithmic kernel, uniform-in-time propagation of chaos is established in \cite[Theorem 4]{CourcelRosenzweigSerfaty2023log} for $\sigma > 0.25$, assuming the existence of a sufficiently regular solution to the limiting equation with initial data $\bar{f}_0$ that is close to chaos and satisfies the smallness condition  $\|\log(\bar{f}_0)\|_{L^\infty(\mathbb{T}^2)} \leq C\left(\sigma - \frac{1}{2\Pi}\right)$, where $C$ is a constant independent of $\sigma$.
\end{remark}

In the next section, we demonstrate that it is possible to remove the high-temperature constraint $\sigma\geq \sigma_0$ when the kernel $K$ belongs to $W^{\frac{-2}{d+2},\,d+2}$. Although our overall strategy remains unchanged, this section is technically more involved: we extend our previous computations, which are valid for $K \in H^{-1}$, using an interpolation argument that allows to address cases where $K \in W^{-\theta, \frac{2}{\theta}}$ for any value of $\theta \in [0,1]$. The case $\theta = \frac{2}{d+2}$, which corresponds to $K \in W^{\frac{-2}{d+2},\,d+2}$, emerges as a limiting case where there is no constraint on the size of $\sigma$.

 \subsection{Proof of Theorem \ref{th:1}: the case \texorpdfstring{$K\in W^{\frac{-2}{d+2},\,d+2}$}{KW} }\label{sec:W:theta}
In this section, we derive uniform-in-time and uniform-in-number-of-particles $L^2$-estimates for  $(f_{k,N})_{1 \leq k \leq N}$ when the interaction kernel $K$ in \eqref{System} belongs to $W^{\frac{-2}{d+2},\,d+2}$. 

As previously mentioned, the main challenge is estimating \eqref{hierarchy:term}, which involves the higher-order marginal $f_{k+1}$ in the equation \eqref{BBGKY}. Instead of relying on large $\sigma$ to compensate for the naive estimate \eqref{naive:estim}, as we did in Section \ref{sec:K:H-1}, we refine \eqref{naive:estim} when $K \in W^{\frac{-2}{d+2},\,d+2}$ in the key result of this section: Lemma \ref{Lemma:hierarchy:term} below. More specifically, we obtain:
\[
\left\|
\int_{\mathbb{T}^d} K(x_i-x_{k+1})f_{k+1}(t,x_1,\dots,x_{k+1})\,\dD x_{k+1}
\right\|
\lesssim
\left\| f_{k}(t,\cdot)\right\|^{1+O\left(\frac{1}{k^2}\right)}\,,\quad\textrm{as}\quad k\rightarrow +\infty\,,
\]
for every $t>0$.

Thanks to Lemma \ref{Lemma:hierarchy:term}, we are able to control the term depending on $f_{k+1}$ in \eqref{BBGKY} with the dissipation induced by the diffusion on the right hand side of \eqref{BBGKY} for any value of $\sigma>0$.

To demonstrate Lemma \ref{Lemma:hierarchy:term} we apply a Sobolev inequality with explicit and optimal dependence with respect to the dimension $dk$ of $\mathbb{T}^{dk}$, as $k \rightarrow +\infty$. A huge literature is dedicated to Sobolev inequalities on general Riemannian manifolds \cite{Aubin1976,Talenti1976,Hebey_Vaugon95,Druet98,Hebey99}. However, up to our knowledge, these results are not optimal in the particular case where the manifold is $\mathbb{T}^{dk}$. In the following result, we focus on the $n$-dimensional torus case and prove, thanks to an explicit construction, that the Sobolev inequality holds with explicit constants relative to $n$.

\begin{theorem}\label{Sob:ineq:Pi}
 The following inequality
\[
\left\|
f
\right\|_{L^{2^\star}(\mathbb{T}^n)}
\,\leq
\,
\sqrt{2e}\,K_n
\left(
\left\|
\nabla_{X^n}
f
\right\|^2_{L^{2}(\mathbb{T}^n)}
+
\frac{4n^2}{|\mathbb{T}|^2}
\left\|
f
\right\|^2_{L^{2}(\mathbb{T}^n)}
\right)^{\frac{1}{2}}\,,
\]
holds, for all function $f\in H^1\left(\mathbb{T}^n\right)$ with $n\geq3$. The critical Sobolev exponent  $2^{\star}$ is given by
$2^\star=2n/(n-2)$. Here, 
 $|\mathbb{T}|$ denotes the length of the torus, and $K_n$ the optimal Sobolev constant on $\R^n$. Specifically, $K_n$ is (see \cite{Aubin1976,Talenti1976})
\[
K_n\,=\,
\left(\frac{1}{n(n-2)}\right)^{\frac{1}{2}}
\left(
\frac{\Gamma(n+1)}{\Gamma\left(\frac{n}{2}\right)\Gamma\left(\frac{n}{2}+1\right)\omega_{n-1}}
\right)^{\frac{1}{n}}\,,
\]
where $\Gamma$ denotes the factorial Gamma function and $\omega_{n-1}$ the volume of the $n-1$ unit sphere.
\end{theorem}
In the proof of Theorem \ref{Sob:ineq:Pi}, which is deferred to Appendix \ref{Proof:th:sob:ineq}, we apply the optimal Sobolev inequality on $\mathbb{R}^n$ to the periodic extension of $f \in H^1\left(\mathbb{T}^n\right)$, truncated using an appropriate step function.

A weaker version of the inequality in Theorem \ref{Sob:ineq:Pi} suffices for our purposes, as we are primarily focused on the dependence on the dimension $n$. Specifically, the Stirling approximation of the Gamma function $\Gamma$ ensures that:
\[
K_n\;\underset{N\rightarrow +\infty}{=}\;
O\left(\frac{1}{\sqrt{n}}\right)\,.
\]
Hence, Theorem \ref{Sob:ineq:Pi} with $n = dk$, combined with the previous estimate, ensures that there exists a constant $C$ depending only on the size of the box $|\mathbb{T}|$ and the dimension $d$, such that 
\begin{equation}\label{sob:ineq:eff}  
\left\|
f
\right\|_{L^{2_k^\star}(\mathbb{T}^{dk})}
\,\leq
\,
C
\left(
\frac{1}{\sqrt{k}}
\left\|
\nabla_{X^k}
f
\right\|_{L^{2}(\mathbb{T}^{dk})}
+
\sqrt{k}
\left\|
f
\right\|_{L^{2}(\mathbb{T}^{dk})}
\right)\,,
\end{equation}
holds, for all $f\in H^1\left(\mathbb{T}^{dk}\right)$, where $2_k^\star = 2dk/(dk-2)$, as soon as $d$ and $k$ are greater or equal to $2$.

Thanks to the Sobolev inequality \eqref{sob:ineq:eff}, we can estimate the term \eqref{hierarchy:term}, which involves the higher-order marginal $f_{k+1}$ in equation \eqref{BBGKY}. The key idea is to "trade off" powers for derivatives of $f_k$ by applying the Sobolev inequality \eqref{sob:ineq:eff}.
\begin{lemma}\label{Lemma:hierarchy:term}
Under assumption \eqref{hyp:f0} on $(f^0_N)_{N \geq 1}$ and assumption \eqref{K:1} on $K$, consider the solutions $(f_N)_{N \geq 1}$ to the Liouville equation \eqref{Liouville} with initial conditions $(f^0_N)_{N \geq 1}$. There exists a constant $C_\delta$, for any $\delta > 0$, depending on $K$ such that  the marginals $(f_{k,N})_{1 \leq k \leq N}$ satisfy:
\[
\begin{split}
	\Big(\int_{\mathbb{T}^{dk}}  \Big|&\int_{\mathbb{T}^d}  K(x_i- x_{k+1})f_{k+1}\left(t,X^{k+1}\right)\dD x_{k+1}\Big|^2\dD X^k\Big)^{\frac{1}{2}} \\
	\leq & \frac{C\delta}{k^{\frac{1}{2}}}\,
	\max{\left(
		\left\|
		\nabla_{X^k}
		f_k(t,\cdot)
		\right\|^{1-\theta- \frac{\theta}{k}
			+\frac{2\theta}{k(dk+2)}}_{L^{2}(\mathbb{T}^{dk})}
		\,,\,
		C^{k}k^{\frac{d k}{4}(1-\theta)}\right)}\left\|\nabla_{X^{k+1}}f_{k+1}(t,\cdot)\right\|^{\theta}_{L^{2}(\mathbb{T}^{d(k+1)})}\\
	&+\frac{C_\delta}{k^{\frac{1}{2}}}\max{\left(
		\left\|
		\nabla_{X^k}
		f_k(t,\cdot)
		\right\|^{\frac{dk}{dk+2}}_{L^{2}(\mathbb{T}^{dk})}
		\,,\,
		C^{k}k^{\frac{dk}{4}}\right)}\,.
\end{split} 
\]
for any time $t \geq 0$, any $(k,N) \in \mathbb{N}^2$, with $2 \leq k \leq N$, and  any $1\leq i \leq k$, where $\theta = 2/(d+2)$, and
for some positive constant $C>0$, depending only on $d$ and the size of the torous $|\mathbb{T}|$.
\end{lemma}
We emphasize that our estimate of \eqref{hierarchy:term} in Lemma \ref{Lemma:hierarchy:term} is homogeneous with respect to $f_k$ up to $O(1/k^2)$. Indeed, for a chaotic marginal $f_{k+1} = F^{\otimes (k+1)}$, the leading order satisfies:
\[
\begin{split}
\|\nabla_{X^k} f_k(t,\cdot)\|_{L^{2}(\mathbb{T}^{dk})}^{1-\theta- \frac{\theta}{k}
	+\frac{2\theta}{k(dk+2)}}&\left\|\nabla_{X^{k+1}}f_{k+1}(t,\cdot)\right\|^{\theta}_{L^2(\mathbb{T}^{d(k+1)})}\\&
\lesssim
\|F(t,\cdot)\|_{H^{1}(\mathbb{T}^{d(k+1)})}^{k+\frac{2\theta}{dk+2}}
\lesssim
\|f_k(t,\cdot)\|_{H^{1}(\mathbb{T}^{dk})}^{1\,+\,O\left(\frac{1}{k^2}\right)}
\,,\quad\textrm{as}\quad k\rightarrow \infty\, ,
\end{split}
\]
for every $t>0$.

This result allows us to control the variations in the $L^2$-norm of $(f_{k,N})_{1\leq k\leq N}$, caused by the term \eqref{hierarchy:term}, with the dissipation resulting from the diffusion term on the right-hand side of \eqref{BBGKY}.
\begin{proof}
Throughout this proof, we fix $(i,k,N) \in \left(\mathbb{N}^\star\right)^3$ such that $i \leq k \leq N-1$, and consider the marginal $f_{k+1}(t)$ of the solution $f_N$ to the Liouville equation \eqref{Liouville} at time $t \geq 0$. 

The first key point in our proof is to handle the singularity of $K$ using an interpolation argument. Specifically, we rewrite the $L^2\left(\mathbb{T}^{dk}\right)$-norm of \eqref{hierarchy:term} as follows:
\[
\int_{\mathbb{T}^{dk}}\left|\int_{\mathbb{T}^d} K(x_i-x_{k+1})f_{k+1}(t,x_1,\dots,x_{k+1})\dD x_{k+1}\right|^2\dD X^k
\,=\,\int_{\mathbb{T}^{dk}}\left|T(\phi)\right|^2\dD X^k\,,
\]
where $\phi$ is given in \eqref{hyp:K:div} and the linear operator $T$ reads
\[
T: \psi\; \longmapsto \;\int_{\mathbb{T}^d} \left(\udiv_{x}\psi\right)(x_i-x_{k+1})\,f_{k+1}(t,x_1,\cdots,x_{k+1})\,\dD x_{k+1}\,.
\]
Since $\phi$ belongs to $W^{1-\theta,2/\theta}$, we prove that $T$ is bounded from $W^{1-\theta,2/\theta}$ to $L^2\left(\mathbb{T}^{dk}\right)$. First, we isolate the small regions of $\mathbb{T}^d$ where $\phi$ is singular, using a density argument. More precisely, since $\mathcal{C}^{\infty}\left(\mathbb{T}^d\right)$ is dense in $W^{1-\theta,2/\theta}\left(\mathbb{T}^d\right)$, for any $\delta > 0$, there exists a matrix field $\phi_\delta$ such that
\begin{equation}\label{phi:decomp}
\left\{
\begin{array}{ll}
     &\ds \left(\phi-\phi_\delta\right)\,\in\,W^{1,\infty}\left(
     \mathbb{T}^d\right)\,,\\[1.5em]
     &\ds \phi_\delta\,\in\,W^{1-\theta,\frac{2}{\theta}}\left(\mathbb{T}^d\right)\,,\quad \textrm{and}\quad 
\left\|\phi_\delta\right\|_{W^{1-\theta,\frac{2}{\theta}}\left(\mathbb{T}^d\right)}\,\leq \,\delta\,.
\end{array}
\right.
\end{equation}
Therefore, the integral $T(\phi)$ admits the following bound:
\begin{equation}\label{phi:decomp:2}
\left\|T(\phi)\right\|_{L^2\left(\mathbb{T}^{dk}\right)}
\,\leq\,
\left\|T(\phi_\delta)\right\|_{L^2\left(\mathbb{T}^{dk}\right)} + \left\|T(\phi-\phi_\delta)\right\|_{L^2\left(\mathbb{T}^{dk}\right)}\,.
\end{equation}
To estimate the contribution of the regular part $\phi - \phi_\delta$, we take the $L^2\left(\mathbb{T}^{dk}\right)$-norm of $T(\phi - \phi_\delta)$ and estimate $\nabla_{x}\phi - \phi_\delta$ using its supremum over $\mathbb{T}^d$. Given that $f_{k+1}(t)$ assumes non-negative values, we obtain:
\[
\left\|T(\phi-\phi_\delta)\right\|_{L^2(\mathbb{T}^{dk})}\leq			 \left\|\phi-\phi_\delta\right\|_{W^{1,\infty}}
\left(
\int_{\mathbb{T}^{dk}}\left|\int_{\mathbb{T}^d} f_{k+1}(t,x_1,\dots,x_{k+1})\,\dD x_{k+1}\right|^2\dD X^k
\right)^{\frac{1}{2}}
.
\]
 Then, we use the relation $\ds\int_{\mathbb{T}^d} f_{k+1}\dD x_{k+1}=f_{k}$ , which yields:
\begin{equation}\label{estim:T:0}
\left\|T(\phi-\phi_\delta)\right\|_{L^2(\mathbb{T}^{dk})}\leq			C_{\delta}\left\|f_{k}(t,\cdot)\right\|_{L^2(\mathbb{T}^{dk})}\,,
		\end{equation}
where $C_\delta = \left\|\phi-\phi_\delta\right\|_{W^{1,\infty}}$.

In the rest of the proof, we estimate the contribution $\left\|T(\phi_\delta)\right\|_{L^2(\mathbb{T}^{dk})}$ from the small yet singular component $\phi_\delta$. To achieve this, we interpolate $T$ between $W^{1,\infty}(\mathbb{T}^{d})$ and $L^{2}(\mathbb{T}^{d})$. On one hand, similar computations to those used to derive \eqref{estim:T:0} confirm that $T$ is bounded from $W^{1,\infty}\left(\mathbb{T}^{d}\right)$ to $L^2\left(\mathbb{T}^{dk}\right)$, specifically:
\begin{equation}\label{estim:T:1}
\left\|T(\psi)\right\|_{L^2(\mathbb{T}^{dk})}\leq			\left\|\psi\right\|_{W^{1,\infty}(\mathbb{T}^{d})}\left\|f_{k}(t,\cdot)\right\|_{L^2(\mathbb{T}^{dk})}\,,\quad\forall\,\psi \in W^{1,\infty}\left(\mathbb{T}^{d}\right)\,.
		\end{equation}
On the other hand, we demonstrate that $T$ maps $L^2(\mathbb{T}^d)$ continuously onto $L^2(\mathbb{T}^{dk})$. To do this, we apply integration by parts with respect to $x_{k+1}$ in the definition of $T$, which leads to the following expression:
		\[
		T\left(\psi\right)\,=\,\int_{\mathbb{T}^d}\psi(x_i-x_{k+1})\,\nabla_{x_{k+1}}f_{k+1}(t,x_1,\dots,x_{k+1})\,\dD x_{k+1}\,.
		\]
		Next, we take the $L^2\left(\mathbb{T}^{dk}\right)$ norm of the above expression and apply the Cauchy-Schwarz inequality, yielding:
\begin{equation}\label{estim:T:2}		\left\|T(\psi)\right\|_{L^2(\mathbb{T}^{dk})}\leq
\left\|\psi\right\|_{L^2(\mathbb{T}^d)}\left\|\nabla_{x_{k+1}}f_{k+1}(t,\cdot)\right\|_{L^2(\mathbb{T}^{d(k+1)})}\,,\quad\forall\,\psi \in L^{2}\left(\mathbb{T}^{d}\right)\,.
\end{equation}
According to \eqref{estim:T:1}-\eqref{estim:T:2}, $T$ defines a bounded mapping from $L^2(\mathbb{T}^d) + W^{1,\infty}(\mathbb{T}^d)$ onto $L^2\left(\mathbb{T}^{dk}\right)$. Consequently, $T$ is also bounded on the space $\left(L^2(\mathbb{T}^d), W^{1,\infty}(\mathbb{T}^d)\right)_{1-\theta, 2/\theta}$, which is obtained through real interpolation with the index $\theta$ given in \eqref{K:1}. This can be shown using references such as \cite[Section 2.4]{Triebel83} or \cite[Theorem 3.1.2]{BerghLofstrom1976}, specifically indicating that:
\begin{equation*}
	\left\|T(\psi)\right\|_{L^2\left(\mathbb{T}^{dk}\right)}
	\leq
	\left\|\psi\right\|_{\left(L^2(\mathbb{T}^d),W^{1,\infty}(\mathbb{T}^{d})\right)_{1-\theta,2/\theta}}\left\|f_{k}(t,\cdot)\right\|^{1-\theta}_{L^2(\mathbb{T}^{dk})}\left\|\nabla_{x_{k+1}}f_{k+1}(t,\cdot)\right\|^{\theta}_{L^2(\mathbb{T}^{d(k+1)})}\,.
\end{equation*}
In this manner, we apply \eqref{prop:interp}, which guarantees that the interpolation space $\big(L^2(\mathbb{T}^d),$ $ W^{1,\infty}(\mathbb{T}^{d})\big)_{1-\theta, 2/\theta}$ corresponds to $W^{1-\theta, 2/\theta}(\mathbb{T}^{d})$, with equivalent norms. Therefore, for all $\psi \in W^{1-\theta, 2/\theta}(\mathbb{T}^{d})$, we have:
\begin{equation*}
\left\|T(\psi)\right\|_{L^2\left(\mathbb{T}^{dk}\right)}
\leq
C\left\|\psi\right\|_{W^{1-\theta,\frac{2}{\theta}}(\mathbb{T}^{d})}\left\|f_{k}(t,\cdot)\right\|^{1-\theta}_{L^2(\mathbb{T}^{dk})}\left\|\nabla_{x_{k+1}}f_{k+1}(t,\cdot)\right\|^{\theta}_{L^2(\mathbb{T}^{d(k+1)})}\,,
\end{equation*}
for some constant $ C $ that depends only on the dimension $ d $ and the size of the box $ \mathbb{T} $, where $ \theta $ is given in \eqref{K:1}. We evaluate the latter estimate with $ \psi = \phi_\delta $ (as defined in \eqref{phi:decomp}) and find: 
\begin{equation}\label{estim:T:3}
\left\|T(\phi_\delta)\right\|_{L^2\left(\mathbb{T}^{dk}\right)}
\leq
C\delta\left\|f_{k}(t,\cdot)\right\|^{1-\theta}_{L^2(\mathbb{T}^{dk})}\left\|\nabla_{x_{k+1}}f_{k+1}(t,\cdot)\right\|^{\theta}_{L^2(\mathbb{T}^{d(k+1)})}\,,
\end{equation}
for some constant $C>0$ depending only on $d$ and the size of the box $|\mathbb{T}|$. 

We now proceed to the second key point in our proof, which involves applying the Sobolev inequality \eqref{sob:ineq:eff} to estimate the $ L^2 $-norm of $ f_k(t) $ in \eqref{estim:T:3} using the norm of its gradient. First, we employ Jensen's inequality to bound the $ L^2 $-norm of $ f_k(t) $ by its $ L^{2^{\star}_k} $-norm, where $ 2^{\star}_k $ is defined below \eqref{sob:ineq:eff},
		\[
\|f_k(t,\cdot)\|_{L^2\left(\mathbb{T}^{dk}\right)}\leq\|f_k(t,\cdot)\|_{L^{2^{\star}_k}\left(\mathbb{T}^{dk}\right)}^{\frac{dk}{dk+2}}\,.
		\]
Then, we estimate the $ L^{2^{\star}_k} $-norm of $ f_k(t) $ on the right-hand side using \eqref{sob:ineq:eff}, resulting in the following expression:
\begin{equation*}
\|f_k(t,\cdot)\|_{L^2\left(\mathbb{T}^{dk}\right)}\leq C\left(\frac{1}{\sqrt{k}}
\left\|
\nabla_{X^k}
f_k(t,\cdot)
\right\|_{L^{2}(\mathbb{T}^{dk})}
+
\sqrt{k}
\left\|
f_k(t,\cdot)
\right\|_{L^{2}(\mathbb{T}^{dk})}\right)^{\frac{dk}{dk+2}},
		\end{equation*}
  for some constant $C$ depending only on the size of the box $|\mathbb{T}|$ and the dimension $d$. From this estimate, we can deduce that:
\begin{equation*}
\|f_k(t,\cdot)\|_{L^2\left(\mathbb{T}^{dk}\right)}\leq C\max{\left(\frac{1}{\sqrt{k}}
\left\|
\nabla_{X^k}
f_k(t,\cdot)
\right\|^{\frac{dk}{dk+2}}_{L^{2}(\mathbb{T}^{dk})}
\,,\,
\sqrt{k}
\left\|
f_k(t,\cdot)
\right\|^{\frac{dk}{dk+2}}_{L^{2}(\mathbb{T}^{dk})}\right)}.
		\end{equation*}
If the second constraint is satisfied in the latter relation, specifically that
$$
\left\|
f_k(t,\cdot)
\right\|_{L^{2}(\mathbb{T}^{dk})}\leq C\sqrt{k}
\left\|
f_k(t,\cdot)
\right\|_{L^{2}(\mathbb{T}^{dk})}^{\frac{dk}{dk+2}}\,,
$$
then, straightforward calculations yield $\left\|
f_k(t,\cdot)
\right\|_{L^{2}(\mathbb{T}^{dk})}
\leq
C^{\frac{dk+2}{2}}k^{\frac{dk+2}{4}}$.
Thus, we can bound the $ L^2 $-norm of $ f_k(t) $ within the $ \max $ in our previous estimate by $ C^{\frac{dk+2}{2}} k^{\frac{dk+2}{4}} $, leading to the following result:
\begin{equation*}
\|f_k(t,\cdot)\|_{L^2\left(\mathbb{T}^{dk}\right)}\leq C\max{\left(\frac{1}{\sqrt{k}}
\left\|
\nabla_{X^k}
f_k(t,\cdot)
\right\|^{\frac{dk}{dk+2}}_{L^{2}(\mathbb{T}^{dk})}
\,,\,
C^{\frac{dk}{2}}k^{\frac{dk+2}{4}}\right)}.
		\end{equation*}
By choosing $ C $ sufficiently large, while ensuring it depends solely on the size of the box $ |\mathbb{T}| $ and the dimension $ d $, we obtain:
\begin{equation}\label{Sob:ineq}
\|f_k(t,\cdot)\|_{L^2\left(\mathbb{T}^{dk}\right)}\leq \frac{C}{k^{\frac{1}{2}}}\max{\left(
\left\|
\nabla_{X^k}
f_k(t,\cdot)
\right\|^{\frac{dk}{dk+2}}_{L^{2}(\mathbb{T}^{dk})}
\,,\,
C^{k}k^{\frac{dk}{4}}\right)}.
		\end{equation}
We can estimate the norm of $ f_k(t) $ in \eqref{estim:T:3} using the previous inequality, leading to:
		\begin{equation*}
\left\|T(\phi_\delta)\right\|_{L^2(\mathbb{T}^{dk})}
			\leq
			\frac{C\delta}{k^{\frac{1-\theta}{2}}}
\max\hspace{-0.1cm}{\left(
\hspace{-0.1cm}\left\|
\nabla_{X^k}
f_k(t,\cdot)
\right\|^{\frac{dk(1-\theta)}{dk+2}}_{L^{2}(\mathbb{T}^{dk})}
,
C^{k(1-\theta)}k^{\frac{dk}{4}(1-\theta)}\hspace{-0.1cm}\right)}\hspace{-0.1cm}\left\|\nabla_{x_{k+1}}f_{k+1}(t,\cdot)\right\|^{\theta}_{L^2(\mathbb{T}^{dk})}.
		\end{equation*}
Next, we note that since $ \theta = \frac{2}{d+2} $, the following equality:
\[
\frac{dk(1-\theta)}{dk+2}
\,=\,
1-\theta- \frac{\theta}{k}
+\frac{2\theta}{k(dk+2)}\,
\]
holds, for all $k\geq2$.
Therefore, our previous estimate of $ T(\phi_\delta) $ can be rewritten as follows:
\begin{equation*}
	\begin{split}
	\Big\|T(&\phi_\delta)\Big\|_{L^2(\mathbb{T}^{dk})}\\
	&\leq
	\frac{C\delta}{k^{\frac{1-\theta}{2}}}
	\max{\left(
		\left\|
		\nabla_{X^k}
		f_k(t,\cdot)
		\right\|^{1-\theta- \frac{\theta}{k}
			+\frac{2\theta}{k(dk+2)}}_{L^{2}(\mathbb{T}^{dk})}
		\,,\,
		C^{k}k^{\frac{d k}{4}(1-\theta)}\right)}\left\|\nabla_{x_{k+1}}f_{k+1}(t,\cdot)\right\|^{\theta}_{L^2(\mathbb{T}^{d(k+1)})}\,.
	\end{split}
\end{equation*}
Furthermore, since the particles are indistinguishable according to \eqref{hyp:f0}, we obtain:
\[\left\|\nabla_{x_{k+1}}f_{k+1}(t,\cdot)\right\|_{L^2(\mathbb{T}^{d(k+1)})} =\frac{\left\|\nabla_{X^{k+1}}f_{k+1}(t,\cdot)\right\|_{L^2(\mathbb{T}^{d(k+1)})}}{\sqrt{k+1}}\,.\]
Substituting the previous relation into our earlier estimate for $ T(\phi_\delta) $ results in:
\begin{equation*}
\begin{split}
\Big\|T(\phi_\delta)&\Big\|_{L^2(\mathbb{T}^{dk})}\\
			&\leq 
   \frac{C\delta}{k^{\frac{1}{2}}}
\max{\Big(
\left\|
\nabla_{X^k}
f_k(t,\cdot)
\right\|^{1-\theta- \frac{\theta}{k}
+\frac{2\theta}{k(dk+2)}}_{L^{2}(\mathbb{T}^{dk})}
\,,\,
C^{k}k^{\frac{d k}{4}(1-\theta)}\Big)}\left\|
\nabla_{X^{k+1}}f_{k+1}(t,\cdot)\right\|^{\theta}_{L^2(\mathbb{T}^{d(k+1)})}.		
\end{split}
\end{equation*}
We can also estimate the $ L^2 $-norm of $ f_k(t) $ in \eqref{estim:T:0} using \eqref{Sob:ineq}, leading to the following result:
\begin{equation*}
\left\|T(\phi-\phi_\delta)\right\|_{L^2(\mathbb{T}^{dk})}\leq	\frac{C_\delta}{k^{\frac{1}{2}}}\max{\left(
\left\|
\nabla_{X^k}
f_k(t,\cdot)
\right\|^{\frac{dk}{dk+2}}_{L^{2}(\mathbb{T}^{dk})}
\,,\,
C^{k}k^{\frac{dk}{4}}\right)}\,.
		\end{equation*}
We can now estimate the right-hand side of \eqref{phi:decomp:2} using the two previous estimates: 
\[
\begin{split}
\Big(\int_{\mathbb{T}^{dk}}\Big|&\int_{\mathbb{T}^d} K(x_i- x_{k+1})f_{k+1}\left(t,X^{k+1}\right)\dD x_{k+1}\Big|^2\dD X^k\Big)^{\frac{1}{2}} \\
 \leq & \frac{C\delta}{k^{\frac{1}{2}}}\,
\max{\left(
\left\|
\nabla_{X^k}
f_k(t,\cdot)
\right\|^{1-\theta- \frac{\theta}{k}
+\frac{2\theta}{k(dk+2)}}_{L^{2}(\mathbb{T}^{dk})}
\,,\,
C^{k}k^{\frac{d k}{4}(1-\theta)}\right)}\left\|\nabla_{X^{k+1}}f_{k+1}(t,\cdot)\right\|^{\theta}_{L^2(\mathbb{T}^{d(k+1)})}\\
&+\frac{C_\delta}{k^{\frac{1}{2}}} \max{\left(
\left\|
\nabla_{X^k}
f_k(t,\cdot)
\right\|^{\frac{dk}{dk+2}}_{L^{2}(\mathbb{T}^{dk})}
\,,\,
C^{k}k^{\frac{dk}{4}}\right)}\, ,
\end{split}
\] 	
which conclude the proof.\end{proof}

\begin{remark}\label{case:K:Ld}
	The previous computations can be adapted to the case where $ K \in L^d(\mathbb{T}^d) $. To achieve this, we interpolate the operator $ T $ between $ W^{1,\infty} $ and $ W^{1,2} $, yielding the following results:
	\begin{equation*}
		\left\|T(\psi)\right\|_{L^2\left(\mathbb{T}^{dk}\right)}
		\leq
		C\left\|\psi\right\|_{W^{1,d}(\mathbb{T}^{d})}\left\|f_{k}(t,\cdot)\right\|^{1-\frac{2}{d}}_{L^2(\mathbb{T}^{dk})}\left\|f_{k+1}(t,\cdot)\right\|^{\frac{2}{d}}_{L^2(\mathbb{T}^{d(k+1)})}\,.
	\end{equation*} 
	We then apply the Sobolev inequality \eqref{Sob:ineq} to both $ f_k(t) $ and $ f_{k+1}(t) $ in the previous estimates, resulting in:
	\begin{align*}
		\left\|T(\psi)\right\|_{L^2\left(\mathbb{T}^{dk}\right)}
		\leq
		\frac{C}{k^{\frac{1}{2}}}\left\|\psi\right\|_{W^{1,d}(\mathbb{T}^{d})}&\max{\left(
			\left\|
			\nabla_{X^k}
			f_k(t,\cdot)
			\right\|^{\frac{(d-2)k}{dk+2}}_{L^{2}(\mathbb{T}^{dk})}
			\,,\,
			C^{k}k^{\frac{(d-2)k}{4}}\right)}\\
			\times &\max{\left(
				\left\|
				\nabla_{X^{k+1}}
				f_k(t,\cdot)
				\right\|^{\frac{2(k+1)}{d(k+1)+2}}_{L^{2}(\mathbb{T}^{d(k+1)})}
				\,,\,
				C^{k}k^{\frac{k}{2}}\right)}\,.
	\end{align*} 
	Like Lemma \ref{Lemma:hierarchy:term}, the recent estimate is homogeneous to $ f_k $ up to $ O(1/k^2) $. Specifically, for a chaotic marginal given by $ f_{k+1} = F^{\otimes (k+1)} $, the leading order in this estimate   satisfies:
	\[
	\begin{split}
	\|\nabla_{X^k} f_k(t,\cdot)\|_{L^{2}(\mathbb{T}^{dk})}^{\frac{(d-2)k}{dk+2}}&\left\|\nabla_{X^{k+1}}f_{k+1}(t,\cdot)\right\|^\frac{2(k+1)}{d(k+1)+2}_{L^2(\mathbb{T}^{d(k+1)})}\\
	&\lesssim
	\|F(t,\cdot)\|_{H^{1}(\mathbb{T}^{d(k+1)})}^{k+O\left(\frac{1}{k}\right)}
	\lesssim
	\|f_k(t,\cdot)\|_{H^{1}(\mathbb{T}^{dk})}^{1\,+\,O\left(\frac{1}{k^2}\right)}
	\,,\quad\textrm{as}\quad k\rightarrow \infty\, ,
	\end{split}
	\]
	for every $t>0$.
	
	This property enables us to carry out the same computations as in the proof of Theorem \ref{th:1} below for the case where $ K \in L^d\left(\mathbb{T}^d\right) $.
\end{remark}

In this second Lemma, we estimate the contribution of the interaction term in equation \eqref{BBGKY} that only involves the marginals $(f_{k})_{1\leq k \leq N}$, that is:
 \[
 K(x_i-x_{j})f_{k}(t,x_1,\dots,x_{k})\,,\quad (i,j)\in\left\{1,\cdots,k\right\}^2\,.
 \]
We address the contributions from the attractive and repulsive components of the kernel $ K $ (as outlined in equation \eqref{K:2}) separately.
	\begin{lemma}\label{Lemma:non:hierarchy:term}
    Under the assumptions \eqref{hyp:f0} regarding $(f^0_N)_{N\geq1}$ and conditions \eqref{K:2}-\eqref{K:+} concerning $K$, we consider the solutions $(f_N)_{N\geq1}$ to the Liouville equation \eqref{Liouville}. There exists a constant $C_\delta$, valid for all $\delta > 0$, such that the marginals $(f_{k,N})_{1\leq k\leq N}$ satisfy the following estimate:
\begin{align*}
\int_{\mathbb{T}^{dk}} K(x_i-x_{j})&f_{k}\left(t,X^{k}\right)\nabla_{x_i} f_{k}(t,X^k)\,\dD X^k\\
&\,\leq\,
\frac{C_\delta}{k}\max{\left(
	\left\|
	\nabla_{X^k}
	f_k(t,\cdot)
	\right\|^{\frac{2dk}{dk+2}}_{L^{2}(\mathbb{T}^{dk})}
	\,,\,
	C^{k}k^{\frac{dk}{2}}\right)}
+C\,\delta
\left\|
\nabla_{x_j}f_{k}(t,\cdot)
\right\|^2_{L^2(\mathbb{T}^{dk})}.
\end{align*}
for all times $t \geq 0$, all $N \geq 2$, and for any $k \in \{2, \ldots, N\}$, as well as for all pairs $(i,j) \in \{1, \ldots, k\}^2$ with $i \neq j$. The constant $C > 0$ depends solely on $d$ and $|\mathbb{T}|$, while the constant $C_\delta$ also depends on $K$ and $\delta$.
	\end{lemma}
\begin{proof}
We fix $(k,N)\in\left(\N^\star\right)^2$ such that $2\leq k\leq N$, $(i,j)\in\left\{1,\dots,k\right\}^2$ such that $i\neq j$, and consider the marginal $f_{k}(t)$ of the solution $f_N$ to the Liouville equation \eqref{Liouville} at time $t\geq 0$. We first decompose the integral with respect to the attractive and repulsive components of $K$:
\[
\int_{\mathbb{T}^{dk}} K(x_i-x_{j})f_{k}\left(t,X^{k}\right)\nabla_{x_i} f_{k}(t,X^k)\,\dD X^k
\,=\,
\cK_- + \cK_+\,,
\]
where
\[
\left\{
\begin{array}{ll}
     &\ds \cK_-\,=\,\int_{\mathbb{T}^{dk}} K_-(x_i-x_{j})f_{k}\left(t,X^{k}\right)\nabla_{x_i} f_{k}(t,X^k)\,\dD X^k\,,\\[1.5em]
     &\ds \cK_+\,=\,\int_{\mathbb{T}^{dk}} K_+(x_i-x_{j})f_{k}\left(t,X^{k}\right)\nabla_{x_i} f_{k}(t,X^k)\,\dD X^k \,. 
\end{array}
\right.
\]

Assumption \eqref{K:-} provides a sufficient framework for estimating  $\cK_-$, corresponding to the repulsive contribution. In contrast, to estimate $\cK_+$, we rely on assumption \eqref{K:+} concerning $K_+$ in conjunction with Sobolev injections.

To proceed with the estimation of $\cK_-$, we utilize the identity:
\[
f_k \nabla_{x_i} f_k = \nabla_{x_i} \left| f_k \right|^2 / 2
\]
and perform integration by parts. This yields the following result: 
\[
\cK_-
=
-\frac{1}{2}
\int_{\mathbb{T}^{dk}} \udiv_{x_i} K_-(x_i-x_{j}) \left|f_{k}(t,X^k)\right|^2\dD X^k.
\]
Next, we apply assumption \eqref{K:-} regarding the repulsive part $ K_- $ of $ K $ to estimate the right-hand side of the previous relation. We obtain:
\[
\cK_-\,\leq\,
\frac{1}{2}
\left\|\left(\udiv_{x} K_-\right)_-\right\|_{L^\infty\left(\mathbb{T}^{d}\right)}
\left\|f_{k}(t,\cdot)\right\|_{L^2\left(\mathbb{T}^{dk}\right)}^2\,.
\]
We now focus on estimating $\cK_+$. For simplicity in our notation, we will perform our computations in the case where $i=k$, noting that other cases can be handled using the same approach. Our strategy involves isolating the small regions of $\mathbb{T}^d$ where the attractive component $K_+$ exhibits singular behavior, utilizing a density argument. Specifically, since $L^\infty\left(\mathbb{T}^d\right)$ is dense in $L^q\left(\mathbb{T}^d\right)$, for any $\delta>0$, we can find two vector fields, $R_\delta$ and $S_\delta$, such that
\begin{equation}\label{K:+:decomp}
K_+\,=\,R_\delta\,+\,S_\delta\,,\quad\textrm{with}\,\quad
\left\{
\begin{array}{ll}
     &\ds R_\delta\,\in\,L^{\infty}\left(
     \mathbb{T}^d\right)\,,\\[1.5em]
     &\ds S_\delta\,\in\,L^{q}\left(
     \mathbb{T}^d\right)\,,\quad \textrm{and}\quad 
     \left\|S_\delta\right\|_{L^{q}\left(
     \mathbb{T}^d\right)}\,\leq \,\delta\,,
\end{array}
\right.
\end{equation}
where $q$ is given in \eqref{K:+}. Therefore, the integral $\cK_+$ admits the following decomposition:
\begin{equation*}
\cK_+
\,=\,
\cR + \cS\,,
\quad
\textrm{where}
\quad
\left\{
\begin{array}{ll}
     &\ds \cR\,=\,\int_{\mathbb{T}^{dk}} R_\delta(x_k-x_{j})f_{k}\left(t,X^{k}\right)\nabla_{x_k} f_{k}(t,X^k)\,\dD X^k\\[1.5em]
     &\ds \cS\,=\,\int_{\mathbb{T}^{dk}} S_\delta(x_k-x_{j})f_{k}\left(t,X^{k}\right)\nabla_{x_k} f_{k}(t,X^k)\,\dD X^k 
\end{array}
\right.\;.
\end{equation*}
In this decomposition, $\mathcal{R}$ accounts for the contribution from the regular part of $K_+$ and is straightforward to estimate, while $\mathcal{S}$ encompasses the contribution from the singular part of $K_+$. 

To estimate $\mathcal{R}$, we first bound the integral by its absolute value, then take the supremum of $R_\delta$ over $\mathbb{T}^d$, and finally apply Young's inequality, resulting in:
\begin{equation}\label{estim:R}
|\cR|
\,\leq\,\frac{1}{2\,\delta}\left\|R_\delta\right\|^2_{L^{\infty}(\mathbb{T}^d)}\left\|f_{k}(t,\cdot)
\right\|_{L^2(\mathbb{T}^{dk})}^{2}
+
\frac{\delta}{2}
\left\|
\nabla_{x_k}f_{k}(t,\cdot) 
\right\|_{L^2(\mathbb{T}^{dk})}^{2}
\,.
\end{equation}
We will now estimate $\mathcal{S}$. To accomplish this, we define the exponent $r^\star$ as follows:
\[
\frac{1}{r^\star}+\frac{1}{2}+\frac{1}{q}\,=\,1\;,
\]
where $q$ is given in \eqref{K:+}. We have $r^\star>1$ since $q>2$. Hence, we can apply H\"older's inequality with respect to $x_k\in\T^d$, resulting in the following estimate: 
\begin{equation*}
|\cS|
\,\leq\,\left\|S_\delta\right\|_{L^{q}(\mathbb{T}^d)}
\int_{\mathbb{T}^{d(k-1)}}
\left\|
f_{k}\left(t,X^{k-1},\cdot\right)
\right\|_{L^{r^\star}(\mathbb{T}^d)}
\left\|
\nabla_{x_k}f_{k}\left(t,X^{k-1},\cdot\right)
\right\|_{L^2(\mathbb{T}^d)}
\dD X^{k-1}
\,. 
\end{equation*}
Furthermore, assumption \eqref{K:+} guarantees that $1/r^\star \geq 1/2 - 1/d$, since $q \geq d$, and that $2 < r^\star < \infty$ because $q > 2$. Thus, the Sobolev inequality on $\mathbb{T}^d$ (see \cite[Corollary 1.2]{Benyi_Oh13}) ensures that the $L^{r^\star}$-norm of $f_k$ in the previous estimate is controlled by its $H^1$-norm, as follows:
\begin{equation*}
|\cS|
\,\leq\,C\left\|S_\delta\right\|_{L^{q}(\mathbb{T}^d)}
\int_{\mathbb{T}^{d(k-1)}}
\left\|f_{k}\left(t,X^{k-1},\cdot\right)
\right\|_{H^{1}(\mathbb{T}^d)}
\left\|
\nabla_{x_k}f_{k}\left(t,X^{k-1},\cdot\right)
\right\|_{L^2(\mathbb{T}^d)}
\dD X^{k-1}
\,,
\end{equation*}
for some positive constant $C$ depending only on $d$ and $|\mathbb{T}|$. Applying Young's inequality to the latter integral gives the following estimate:
\begin{equation}\label{estim:S}
|\cS|
\,\leq\,C\left\|S_\delta\right\|_{L^{q}(\mathbb{T}^d)}\left(
\left\|f_{k}(t,\cdot)
\right\|^2_{L^2(\mathbb{T}^{dk})}
+
\left\|
\nabla_{x_k}f_{k}(t,\cdot)
\right\|^2_{L^2(\mathbb{T}^{dk})}\right)
\,.
\end{equation}
Next, we sum the estimates in \eqref{estim:R} and \eqref{estim:S} and bound the norms of $R_\delta$ and $S_\delta$ according to \eqref{K:+:decomp}. From this, we deduce the following bound for $\cK_+$:
\begin{equation*}|\cK_+|
\,\leq\,
C_\delta\left\|f_{k}(t,\cdot)
\right\|^2_{L^2(\mathbb{T}^{dk})}
+C\,\delta
\left\|
\nabla_{x_k}f_{k}(t,\cdot)
\right\|^2_{L^2(\mathbb{T}^{dk})}
\,,
\end{equation*}
for some positive constant $C$ depending only on $d$ and $|\mathbb{T}|$, while the constant $C_\delta$ also depends on $K_+$ and $\delta$. Hence, combining our estimates on $\cK_-$ and $\cK_+$, we obtain the following bound:
\[
\int_{\mathbb{T}^{dk}} K(x_i-x_{j})f_{k}\left(t,X^{k}\right)\nabla_{x_i} f_{k}(t,X^k)\dD X^k
\,\leq\,
C_\delta\left\|f_{k}(t,\cdot)
\right\|^2_{L^2(\mathbb{T}^{dk})}
+C\delta
\left\|
\nabla_{x_i}f_{k}(t,\cdot)
\right\|^2_{L^2(\mathbb{T}^{dk})}.
\]
To conclude our proof, we apply the Sobolev inequality \eqref{Sob:ineq} to bound the norm of $f_k$ in the final estimate. This leads to the desired result:
\begin{align*}
\int_{\mathbb{T}^{dk}} K(x_i-x_{j})&f_{k}\left(t,X^{k}\right)\nabla_{x_i} f_{k}(t,X^k)\,\dD X^k\\
&\,\leq\,
\frac{C_\delta}{k}\max{\left(
	\left\|
	\nabla_{X^k}
	f_k(t,\cdot)
	\right\|^{\frac{2dk}{dk+2}}_{L^{2}(\mathbb{T}^{dk})}
	\,,\,
	C^{k}k^{\frac{dk}{2}}\right)}
+C\,\delta
\left\|
\nabla_{x_i}f_{k}(t,\cdot)
\right\|^2_{L^2(\mathbb{T}^{dk})},
\end{align*}
completing the proof.
\end{proof}
To prove Theorem \ref{th:1}, we compile the results of Lemmas \ref{Lemma:hierarchy:term} and \ref{Lemma:non:hierarchy:term}. These results enable us to control the variations in the $ L^2 $-norm of $ (f_{k,N})_{1\leq k\leq N} $ arising from interactions between particles, alongside the dissipation attributed to the diffusion term on the right-hand side of \eqref{BBGKY}.

\begin{proof}[Proof of Theorem \ref{th:1}]
We fix $ t \geq 0 $ and $ (k,N) \in \left(\mathbb{N}^\star\right)^2 $ such that $ 2 \leq k \leq N $. To estimate the $ L^2 $-norm of $ f_k $ at time $ t $, we calculate its time derivative by multiplying equation \eqref{BBGKY} by $ f_k $ and integrating over $ \mathbb{T}^{dk} $. This results in the following expression:
 \begin{equation*}
\frac{1}{2}\frac{\dD}{\dD t}  \left\|f_k(t,\cdot)\right\|_{L^2(\mathbb{T}^{dk})}^2
\,=\;\cA\,+\,\cB\,+\,\cC\,,
 \end{equation*}
	where, following the same computations as in the proof of Theorem \ref{th:2} in Section \ref{sec:K:H-1}, $\cA$, $\cB$ and $\cC$ are given by
\[
\left\{
\begin{array}{ll}
	&\ds\cA\,=\,\frac{1}{N}\,\sum_{\substack{i,j=1\\ i\neq j}}^k\int_{\mathbb{T}^{dk}}   K(x_i-x_j)\cdot f_{k}(t,X^k)\nabla_{x_i} f_{k}(t,X^k)\,\dD X^k\,,  \\[1.em]
	&\ds\cB\,=\, \frac{N-k}{N} \sum_{i=1}^k \int_{\mathbb{T}^{dk}}\left( \int_{\mathbb{T}^d} K(x_i-x_{k+1}) f_{k+1}(t,X^{k+1})\, \dD x_{k+1}\right)  \cdot \nabla_{x_i} f_{k}(t,X^{k})\,\dD X^k \,,\\[1.5em]
	&\ds\cC\,=\,- \,\sigma\left\|\nabla_{X^k} f_k(t,\cdot)\right\|^2_{L^{2}\left(\mathbb{T}^{dk}\right)}\,\leq\,0\,. 
\end{array}
\right.
\] 
The main contribution arises from $ \cB $, as it depends on $ f_{k+1} $. We estimate this term using Lemma \ref{Lemma:hierarchy:term}. For the lower-order term $ \cA $, we apply Lemma \ref{Lemma:non:hierarchy:term}. Finally, $ \cC $ represents the contribution from diffusion in \eqref{BBGKY}. It has a signed contribution that we leverage to control both $ \cA $ and $ \cB $.

To estimate the primary contribution $ \cB $, we begin with the same calculations as in the proof of Theorem \ref{th:2} in Section \ref{sec:K:H-1}. When $ k = N $, we find that $ \cB = 0 $. However, for $ k \leq N-1 $, we can use estimate \eqref{estim:B:sec:theta}, which states that:
\begin{equation*}
	\cB\,\leq\,
	\frac{1}{k^{\frac{1}{2}}} \left\|\nabla_{X^k} f_{k}(t,\cdot) \right\|_{L^2\left(\mathbb{T}^{dk}\right)}\sum_{i=1}^k \left(\int_{\mathbb{T}^{dk}}\left| \int_{\mathbb{T}^d} K(x_i-x_{k+1}) f_{k+1}(t,X^{k+1})\, \dD x_{k+1}\right|^2\dD X^k\right)^{\frac{1}{2}}
	.
\end{equation*}
Then, we bound the convolution term on the right-hand side using Lemma \ref{Lemma:hierarchy:term}, which yields:
\[\begin{split}
	\cB\,\leq\,
	C\delta_1\,
	&\max{\left(
		\left\|
		\nabla_{X^k}
		f_k(t, \cdot)
		\right\|^{1-\theta- \frac{\theta}{k}
			+\frac{2\theta}{k(dk+2)}}_{L^{2}(\mathbb{T}^{dk})}
		\,,\,
		C^{k}k^{\frac{d k}{4}(1-\theta)}\right)} \\
		& \hspace{1,3cm}\times\left\|\nabla_{X^{k+1}}f_{k+1}(t,\cdot)\right\|^{\theta}_{L^{2}(\mathbb{T}^{d(k+1)})}
		\left\|
		\nabla_{X^k}
		f_k(t, \cdot)
		\right\|_{L^{2}(\mathbb{T}^{dk})}\\
	+\,C_{\delta_1}&\max{\left(
		\left\|
		\nabla_{X^k}
		f_k(t,\cdot)
		\right\|^{\frac{dk}{dk+2}}_{L^{2}(\mathbb{T}^{dk})}
		\,,\,
		C^{k}k^{\frac{dk}{4}}\right)}\left\|
		\nabla_{X^k}
		f_k(t, \cdot)
		\right\|_{L^{2}(\mathbb{T}^{dk})}\,,
\end{split} 
\]
for all $\delta_1 > 0$, where $\theta$ is specified in assumption \eqref{K:1}. The constant $C$ depends on $d$ and the size of the box $|\mathbb{T}|$, while $C_{\delta_1} > 0$ depends on $K$ and $\delta_1$. We apply Young's inequality to the last term on the right-hand side, obtainig:
\[
\begin{split}
	\cB\leq
	C\delta_1
	&\max{\left(
		\left\|
		\nabla_{X^k}
		f_k(t,\cdot)
		\right\|^{1-\theta- \frac{\theta}{k}
			+\frac{2\theta}{k(dk+2)}}_{L^{2}}
		,
		C^{k}k^{\frac{d k}{4}(1-\theta)}\right)}
		\\ & \hspace{1,3cm} \times \left\|\nabla_{X^{k+1}}f_{k+1}(t,\cdot)\right\|^{\theta}_{L^{2}}
	\left\|
	\nabla_{X^k}
	f_k(t,\cdot)
	\right\|_{L^{2}(\mathbb{T}^{dk})}\\
	+\,\frac{C_{\delta_1}}{\eta}&\max{\left(
		\left\|
		\nabla_{X^k}
		f_k(t,\cdot)
		\right\|^{\frac{2dk}{dk+2}}_{L^{2}}
		\,,\,
		C^{k}k^{\frac{dk}{2}}\right)}
		\,+\,\eta
		\left\|
	\nabla_{X^k}
	f_k(t,\cdot)
	\right\|_{L^{2}(\mathbb{T}^{dk})}^2\,,
\end{split} 
\]
for any $\eta>0$. We now estimate $\cA$ using Lemma \ref{Lemma:non:hierarchy:term}, which gives us:
\begin{align*}
	\cA \,\leq\,\frac{1}{N}\,\sum_{\substack{i,j=1\\ i\neq j}}^k\left(\frac{C_{\delta_2}}{k}\max{\left(
		\left\|
		\nabla_{X^k}
		f_k(t,\cdot)
		\right\|^{\frac{2dk}{dk+2}}_{L^{2}}
		\,,\,
		C^{k}k^{\frac{dk}{2}}\right)}
	+C\,\delta_2
	\left\|
	\nabla_{x_j}f_{k}(t,\cdot)
	\right\|^2_{L^2}\right),
\end{align*}
for all $\delta_2 > 0$ and for some positive constant $C > 0$ depending only on $d$ and $|\mathbb{T}|$, while $C_{\delta_2}$ also depends on $K$ and $\delta_2$. We explicitly compute the sum and utilize the fact that $k/N \leq 1$ to derive: 
\[
\begin{split}
	\cA \,\leq\,C_{\delta_2} \max{\left(
		\left\|
		\nabla_{X^k}
		f_k(t,\cdot)
		\right\|^{\frac{2dk}{dk+2}}_{L^{2}}
		\,,\,
		C^{k}k^{\frac{dk}{2}}\right)}
	 +C\,\delta_2
	\left\|
	\nabla_{X^k}f_{k}(t,\cdot)
	\right\|^2_{L^2}.
\end{split}
\]
Taking the sum between our estimates for $\cA$ and $\cB$, we obtain:
\[
\begin{split}
	\frac{1}{2}\frac{\dD}{\dD t}\Big\|f_k(t,\cdot)&\Big\|_{L^2(\mathbb{T}^{dk})}^2 \\
	\leq &\, C\delta_1\,
	\max{\left(
		\left\|
		\nabla_{X^k}
		f_k(t,\cdot)
		\right\|^{1-\theta- \frac{\theta}{k}
			+\frac{2\theta}{k(dk+2)}}_{L^{2}}
		\,,\,
		C^{k}k^{\frac{d k}{4}(1-\theta)}\right)}
		\\& \hspace{1cm}\times\left\|\nabla_{X^{k+1}}f_{k+1}(t,\cdot)\right\|^{\theta}_{L^{2}}
	\left\|
	\nabla_{X^k}
	f_k(t,\cdot)
	\right\|_{L^{2}(\mathbb{T}^{dk})}\\
	&+\left(\frac{C_{\delta_1}}{\eta}+C_{\delta_2}\right)\max{\left(
		\left\|
		\nabla_{X^k}
		f_k(t,\cdot)
		\right\|^{\frac{2dk}{dk+2}}_{L^{2}}
		\,,\,
		C^{k}k^{\frac{dk}{2}}\right)}
	\\& +\left(\eta+C\delta_2-\sigma\right)
	\left\|
	\nabla_{X^k}
	f_k(t,\cdot)
	\right\|_{L^{2}(\mathbb{T}^{dk})}^2\,.
\end{split}
\] 
We then decompose the resulting estimate as follows:
\[
\frac{1}{2}\frac{\dD}{\dD t}  \left\|f_k(t,\cdot)\right\|_{L^2(\mathbb{T}^{dk})}^2\,\leq\,\cD_1\,+\,\cD_2\,+\,\cD_3\,+\,\cD_4
\,+\left(\eta+C\delta_2-\sigma\right)
\left\|
\nabla_{X^k}
f_k(t,\cdot)
\right\|_{L^{2}(\mathbb{T}^{dk})}^2
\,,
\]
where $\cD_1$, $\cD_2$, $\cD_3$ and $\cD_4$ are defined as follows:
\[
\left\{
\begin{array}{ll}
     &\ds\cD_1\,=\,C\delta_1\,
     	\left\|
     	\nabla_{X^k}
     	f_k(t,\cdot)
     	\right\|^{2-\theta- \frac{\theta}{k}
     		+\frac{2\theta}{k(dk+2)}}_{L^{2}}
     	\left\|\nabla_{X^{k+1}}f_{k+1}(t,\cdot)\right\|^{\theta}_{L^{2}}
 \\[1.em]
 &\ds\cD_2\,=\,C\delta_1\,
 	C^{k}k^{\frac{d k}{4}(1-\theta)}
 	\left\|
 	\nabla_{X^k}
 	f_k(t,\cdot)
 	\right\|_{L^{2}(\mathbb{T}^{dk})}\left\|\nabla_{X^{k+1}}f_{k+1}(t,\cdot)\right\|^{\theta}_{L^{2}(\mathbb{T}^{d(k+1)})}
 \\[1.em]
 &\ds\cD_3\,=\,\left(\frac{C_{\delta_1}}{\eta}+C_{\delta_2}\right)
 	\left\|
 	\nabla_{X^k}
 	f_k(t,\cdot)
 	\right\|^{2 - \frac{2}{dk+2}}_{L^{2}(\mathbb{T}^{dk})}
 \\[1.em]
     &\ds\cD_4\,=\, \left(\frac{C_{\delta_1}}{\eta}+C_{\delta_2}\right)
     	C^{k}k^{\frac{dk}{4}}\left\|
     \nabla_{X^k}
     f_k(t,\cdot)
     \right\|_{L^{2}(\mathbb{T}^{dk})}
\end{array}
\right.
.
\]
The main contribution comes from $ \cD_1 $, while the terms $ \cD_j $ for $ j \geq 2 $ are lower-order contributions. To estimate $ \cD_1$, we apply Young's inequality with the following triplet of exponents:
\[
\frac{\theta d}{2(dk+2)}
\,+\,
\left(1-\frac{\theta}{2}- \frac{\theta}{2k}
+\frac{\theta}{k(dk+2)}\right)\,+\,\frac{\theta}{2}
\,=\,1\,.
\]
This results in the following estimate for $\cD_1$:
\[
\cD_1\,\leq
\left(
\frac{C\delta_1
}{\eta^{1-\frac{\theta}{2}- \frac{\theta}{2k}
	+\frac{\theta}{k(dk+2)}}
\eta^{\frac{\theta}{2}}_k
}
\right)^{\frac{2(dk+2)}{\theta d}}+
\eta
\|\nabla_{X^k} f_k(t,\cdot)\|_{L^{2}(\mathbb{T}^{dk})}^{2}
+
\eta_k
\left\|\nabla_{X^{k+1}}f_{k+1}(t,\cdot)\right\|^{2}_{L^2(\mathbb{T}^{d(k+1)})}\,,
\]
for all positive $\eta$ and $\eta_k$.
Simplifying the exponents in the previous estimate gives:
\[
\cD_1\,\leq\,
\frac{\left(C
\delta_1\right)^{\frac{2k}{\theta}}}{\eta^{(d+1)k+\frac{1}{1-\theta}}
\eta^{k+\frac{2}{d}}_k
}+\,
\eta\,
\|\nabla_{X^k} f_k(t,\cdot)\|_{L^{2}(\mathbb{T}^{dk})}^{2}
\,+\,
\eta_k\,
\left\|\nabla_{X^{k+1}}f_{k+1}(t,\cdot)\right\|^{2}_{L^2\mathbb{T}^{dk})}\,.
\]
We also apply Young inequality to estimate $\cD_2$, this time using the triplet of exponents:
\[
\frac{1-\theta}{2}+
\frac{1}{2} + \frac{\theta}{2}=1\,,
\]
which provides the following estimate for $\cD_2$:
\[
\cD_2\,\leq\,
\frac{\left(C
	\delta_1\right)^{\frac{2}{1-\theta}}}{\eta^{\frac{1}{1-\theta}}
	\eta^{\frac{\theta}{1-\theta}}_k
} k^{\frac{dk}{2}}+\,
\eta\,
\|\nabla_{X^k} f_k(t,\cdot)\|_{L^{2}(\mathbb{T}^{dk})}^{2}
\,+\,
\eta_k\,
\left\|\nabla_{X^{k+1}}f_{k+1}(t,\cdot)\right\|^{2}_{L^2}\,.
\]
In order to evaluate $\cD_3$, we apply Young's inequality using the following pair of exponents:
\[
\frac{1}{dk+2}
\,+\,
\left(1-\frac{1}{dk+2}\right)
\,=\,1\,.
\]
This results in:
\[
\cD_3\,\leq\,
\frac{\left(C_{\delta_1}+\eta C_{\delta_2}\right)^{dk+2}}{\eta^{2dk+3}}+\,
\eta\,
\|\nabla_{X^k} f_k(t,\cdot)\|_{L^{2}(\mathbb{T}^{dk})}^{2}\,.
\]
Likewise, we utilize Young's inequality to assess $\cD_4$, leading to:
\[
\cD_4\,\leq\,
\frac{1}{\eta}\left(\frac{C_{\delta_1}}{\eta}+C_{\delta_2}\right)^2
C^{k}k^{\frac{dk}{2}}
\,+\,
\eta\left\|
\nabla_{X^k}
f_k(t,\cdot)
\right\|_{L^{2}(\mathbb{T}^{dk})}^2\,.
\]
After summing the estimates for $\cD_j$ with $1 \leq j \leq 4$, we obtain:
\begin{align*}
\frac{1}{2}\frac{\dD}{\dD t} & \left\|f_k(t,\cdot)\right\|_{L^2(\mathbb{T}^{dk})}^2\leq\,\\&
\frac{\left(C
\delta_1\right)^{\frac{2k}{\theta}}}{\eta^{(d+1)k+\frac{1}{1-\theta}}
\eta^{k+\frac{2}{d}}_k}
\,+\,
\frac{\left(C
\delta_1\right)^{\frac{2}{1-\theta}}}{\eta^{\frac{1}{1-\theta}}
\eta^{\frac{\theta}{1-\theta}}_k
} k^{\frac{dk}{2}}
\,+\,
\frac{\left(C_{\delta_1}+\eta C_{\delta_2}\right)^{dk+2}}{\eta^{2dk+3}}
\,+\,
\frac{1}{\eta}\left(\frac{C_{\delta_1}}{\eta}+C_{\delta_2}\right)^2
C^{k}k^{\frac{dk}{2}}\\[0.8em]
\,+\,&
\left(5\eta+C\delta_2-\sigma\right)
\left\|
\nabla_{X^k}
f_k(t,\cdot)
\right\|_{L^{2}(\mathbb{T}^{dk})}^2
\,+\,
4\,\eta_k\,
\left\|\nabla_{X^{k+1}}f_{k+1}(t,\cdot)\right\|^{2}_{L^2}\,.
\end{align*}
Next, we fix $\eta$ and $\delta_2$ such that $5\eta=C\delta_2 = \sigma/4$, given rise to:
\begin{align*}
	\frac{1}{2}\frac{\dD}{\dD t}  \left\|f_k(t,\cdot)\right\|_{L^2(\mathbb{T}^{dk})}^2\leq\,&\ds
	\frac{C^k
		\delta_1^{\frac{2k}{\theta}}}{
		\eta^{k+\frac{2}{d}}_k}
	+\,
	\frac{C
		\delta_1^{\frac{2}{1-\theta}}}{
		\eta^{\frac{\theta}{1-\theta}}_k
	} k^{\frac{dk}{2}}
	+\,
	C_{\delta_1}^{dk+2}
	+\,
	C_{\delta_1}^2
	C^{k}k^{\frac{dk}{2}}\\
	& +\,\ds
	4\,\eta_k
	\left\|\nabla_{X^{k+1}}f_{k+1}(t,\cdot)\right\|^{2}_{L^2}
	-\,\frac{\sigma}{2}
	\left\|
	\nabla_{X^k}
	f_k(t,\cdot)
	\right\|_{L^{2}(\mathbb{T}^{dk})}^2\,.
\end{align*}
Our strategy is to compensate for the higher-order norm of $ f_{k+1}(t) $ in the previous estimate with the dissipation term due to diffusion. To achieve this, we divide the estimate by $ k^{\alpha k} $, for some $ \alpha > 0 $ that will be determined later, and then sum over $ k $ from $ 2 $ to $ N $. After re-indexing, this produces:
\begin{align*}
\frac{1}{2}\frac{\dD}{\dD t} 
\sum_{k=2}^{N} \frac{\left\|f_k(t,\cdot)\right\|_{L^2(\mathbb{T}^{dk})}^2}{k^{\alpha k}}
\,\leq\,
&\sum_{k=2}^{N} 
k^{-\alpha k }\left(
\frac{C^k
	\delta_1^{\frac{2k}{\theta}}}{
	\eta^{k+\frac{2}{d}}_k}
+\,
\frac{C
	\delta_1^{\frac{2}{1-\theta}}}{
	\eta^{\frac{\theta}{1-\theta}}_k
} k^{\frac{dk}{2}}
+\,
C_{\delta_1}^{dk+2}
+\,
C_{\delta_1}^2
C^{k}k^{\frac{dk}{2}}\right)\\
& \,+\,
\sum_{k=2}^{N} 
\left(\frac{k^{\alpha k} }{(k-1)^{\alpha(k-1)}}\eta_{k-1} - \frac{\sigma}{2}\right)
\frac{\left\|\nabla_{X^{k}}f_k(t,\cdot)\right\|_{L^2(\mathbb{T}^{dk})}^2}{k^{\alpha k}}
\,.
 \end{align*}
We set $ \eta_k = \frac{\sigma}{4 e^{2\alpha} k^\alpha} $ and bound the second sum on the right-hand side using the following estimate: 
 \[
 \frac{k^{\alpha k} }{(k-1)^{\alpha(k-1)}}\,=\,(k-1)^{\alpha}\left(\frac{k}{k-1}\right)^{\alpha k}\,\leq\,
 (k-1)^{\alpha}e^{2\alpha}\,.
 \]
We then have:
\begin{align*}
\frac{1}{2}\frac{\dD}{\dD t} 
\sum_{k=2}^{N} \frac{\left\|f_k(t,\cdot)\right\|_{L^2(\mathbb{T}^{dk})}^2}{k^{\alpha k}}
\leq
&\sum_{k=2}^{N} 
C_\alpha^k\delta_1^{\frac{2}{\theta}k}
k^{\frac{2\alpha}{d}}
+
\left(C_\alpha
	\delta_1^{\frac{2}{1-\theta}}
	+
	C_{\delta_1}^2
	C^{k}
	k^{\frac{\alpha\theta}{1-\theta}}
	\right)
k^{k\left(\frac{d}{2}-\alpha\right)}
+
k^{-\alpha k }
C_{\delta_1}^{dk+2}\\
& \,-\,
\frac{\sigma}{4}
\sum_{k=2}^{N} \frac{\left\|\nabla_{X^k} f_k(t,\cdot)\right\|_{L^2(\mathbb{T}^{dk})}^2}{k^{\alpha k}}
\,.
 \end{align*}
Now, we impose the constraint $\alpha > d/2$ to ensure that the second and third terms in the first sum on the right-hand side lead to a convergent series. We find: 
\begin{equation*}
	\frac{1}{2}\frac{\dD}{\dD t} 
	\sum_{k=2}^{N} \frac{\left\|f_k(t,\cdot)\right\|_{L^2(\mathbb{T}^{dk})}^2}{k^{\alpha k}}
	\,\leq\,
	C_{\alpha,\delta_1}+
	\sum_{k=2}^{N} 
	C_\alpha^k\delta_1^{\frac{2}{\theta}k}
	k^{\frac{2\alpha}{d}}
	\,-\,
	\frac{\sigma}{4}
	\sum_{k=2}^{N} \frac{\left\|\nabla_{X^k} f_k(t,\cdot)\right\|_{L^2(\mathbb{T}^{dk})}^2}{k^{\alpha k}}
	\,,
\end{equation*} 
for some constant $C_{\alpha,\delta_1} > 0$ that depends on our choice of $\alpha > d/2$ and $\delta_1 > 0$. To conclude, we set $\delta_1$ such that $C_\alpha \delta_1^{\frac{2}{\theta}} < 1$, ensuring that the first sum also results in a convergent series. We obtain:
\begin{equation*}
\frac{1}{2}\frac{\dD}{\dD t} 
\sum_{k=2}^{N} \frac{\left\|f_k(t,\cdot)\right\|_{L^2(\mathbb{T}^{dk})}^2}{k^{\alpha k}}
\,\leq\,
C_\alpha
\,-\,
\frac{\sigma}{4}
\sum_{k=2}^{N} \frac{\left\|\nabla_{X^k} f_k(t,\cdot)\right\|_{L^2(\mathbb{T}^{dk})}^2}{k^{\alpha k}}
\,,
 \end{equation*}
for some constant $C_\alpha > 0$ that depends on $d$, $|\mathbb{T}|$, $K$, $\sigma$, and $\alpha$. We can lower bound the sum on the right-hand side, utilizing the Poincar\'e inequality on $\mathbb{T}^{dk}$, which guarantees that:
\begin{equation*}
\begin{split}
	\left\|f_k(t,\cdot)\right\|_{L^2(\mathbb{T}^{dk})}^2
	-\left(\int_{\mathbb{T}^{dk}} f_k(t,X^k)\,\dD X^k\right)^2
	&=
	\left\|f_k(t,\cdot)
	-\int_{\mathbb{T}^{dk}} f_k(t,X^k)\dD X^k
	\right\|_{L^2(\mathbb{T}^{dk})}^2\\
	&\leq
	\left\|\nabla_{X^k} f_k(t,\cdot)\right\|_{L^2(\mathbb{T}^{dk})}^2.
\end{split}
\end{equation*}

Since $\ds\int_{\mathbb{T}^{dk}} f_k(t,X^k)\,\dD X^k\,=\,1$, we find:
\begin{equation*}
\frac{1}{2}\frac{\dD}{\dD t} 
\sum_{k=2}^{N} \frac{\left\|f_k(t,\cdot)\right\|_{L^2(\mathbb{T}^{dk})}^2}{k^{\alpha k}}
\,\leq\,
C_\alpha
\,-\,
\frac{\sigma}{4}
\sum_{k=2}^{N} \frac{\left\|f_k(t,\cdot)\right\|_{L^2(\mathbb{T}^{dk})}^2}{k^{\alpha k}}
\,.
 \end{equation*}
We then multiply the preceding estimate by $ e^{\sigma t / 2} $ and integrate with respect to $ t \geq 0 $. This results in:
 \begin{equation*}
\sum_{k=2}^{N} \frac{\left\|f_k(t,\cdot)\right\|_{L^2(\mathbb{T}^{dk})}^2}{k^{\alpha k}}
\,\leq\,
e^{-\frac{\sigma t }{2}}
\sum_{k=2}^{N} \frac{\left\|f_k^0\right\|_{L^2(\mathbb{T}^{dk})}^2}{k^{\alpha k}}
+
\left(1-e^{-\frac{\sigma t }{2}}\right)C_\alpha
\,.
 \end{equation*}
To conclude our proof, we select an $\alpha$ such that $\alpha > \max(2\beta, d/2)$, where $\beta$ is specified in \eqref{hyp:f0:growth}. This choice ensures that the sum on the right-hand side remains uniformly bounded as $N \to +\infty$. This leads to:
\begin{equation*}
\sum_{k=2}^{N} \frac{\left\|f_k(t,\cdot)\right\|_{L^2(\mathbb{T}^{dk})}^2}{k^{\alpha k}}
\,\leq\,C\,,\quad \forall\, t \geq 0
\,,
 \end{equation*}
for some constant $C$ depending on $d$, $|\mathbb{T}|$, $K$, $\sigma$, $\alpha$ and the implicit constant in \eqref{hyp:f0:growth}. We can easily derive the expected result from the preceding estimate:
 \begin{equation*}
\sup_{t\in \R^+} 
\sup_{2\leq N}\sup_{ k\leq  N}
\frac{\left\|f_{k,N}(t,\cdot)\right\|_{L^2(\mathbb{T}^{dk})}
}{k^{\Tilde{\alpha}k}}
\,\leq\,C\,,
 \end{equation*}
 where $\Tilde{\alpha} = \alpha/2 > \max\left(\beta,d/4\right)$, 
for some constant $C$ depending on $d$, $|\mathbb{T}|$, $K$, $\sigma$, $\Tilde{\alpha}$ and the implicit constant in \eqref{hyp:f0:growth}.
\end{proof}

\section{Uniform in time propagation of chaos}\label{sec:4}
In this section, we demonstrate uniform in time propagation of chaos as defined in \eqref{DefinitionPropagationOfChaos} for the particle system described by \eqref{System}. Specifically, we establish quantitative decay rates in both $N$ and $t \geq 0$, which ensure that the marginals $f_{k,N}$ converge to the tensorized limit $\bar{f}^{\otimes k}$ in $L^p(\mathbb{T}^{dk})$ for $1 \leq p < 2$. This convergence occurs simultaneously as $N \to +\infty$ and $t \to +\infty$. The main steps in the proof outlined in this section are summarized as follows:

$(i)$ The first key point is that for divergence-free kernels $K$, the marginal $f_{k,N}$ and the tensorized limit $\bar{f}^{\otimes k}$ converge to the same stationary state in $L^1$ as $t \to +\infty$. Specifically, we establish this result in Lemma \ref{EntropyEstimatesLemma} below:
\[
f_{k,N}(t) = \bar{f}^{\otimes k}(t) + O\left(e^{-Ct}\right)\,,\quad\textrm{in}\quad L^1\left(\mathbb{T}^{dk}\right)\,,\quad\textrm{as}\quad t,N\;\longrightarrow +\infty\,;
\]

$(ii)$ Next, we combine the aforementioned estimate with \cite[Theorem 1]{JabinWang2018}, which guarantees propagation of chaos with exponential growth, specifically:
\[
f_{k,N}(t) = \bar{f}^{\otimes k}(t) + O\left(\frac{e^{Ct}}{\sqrt{N}}\right)\,,\quad\textrm{in}\quad L^1\left(\mathbb{T}^{dk}\right)\,,\quad\textrm{as}\quad t,N\;\longrightarrow +\infty\,.
\]
The combination of these two results leads to the simultaneous convergence of $ f_{k,N} $ toward $ \bar{f}^{\otimes k} $ in $ L^1 $:
\[
f_{k,N}(t) = \bar{f}^{\otimes k}(t) + O\left(\frac{e^{-Ct}}{N^{\alpha}}\right)\,,\quad\textrm{in}\quad L^1\left(\mathbb{T}^{dk}\right)\,,\quad\textrm{as}\quad t,N\;\longrightarrow +\infty\,,
\]
for some $\alpha>0$ ;

$(iii)$ We then interpolate the previous estimate with our uniform $ L^2 $-estimates obtained in Theorems \ref{th:1} and \ref{th:2}, which enhances the $ L^1 $ convergence into $ L^p $ convergence for all $ 1 \leq p < 2 $.

For divergence-free kernels $ K $, the marginals $ f_{k,N} $ and the tensorized limit $ \bar{f}^{\otimes k} $ converge to the same stationary state as $ t \rightarrow +\infty $. This stationary state is represented by $ f^{\otimes k}_{\infty} $, where $ f_\infty $ is the uniform probability distribution over $ \mathbb{T}^{d} $. This  result is the subject of the following lemma, whose proof is provided for completeness, as it relies on a classical entropy estimate. 
\begin{lemma}\label{EntropyEstimatesLemma}
    Assume that the interaction kernel satisfies $\udiv_x(K) = 0$. Under the assumptions \eqref{hyp:f0} and \eqref{hyp:chaos} for $(f^0_N)_{N \geq 1}$ and the assumption \eqref{f0Regularity} on $\bar{f}^0$, consider the solutions $(f_N)_{N \geq 1}$ to the Liouville equation \eqref{Liouville} with initial conditions $(f^0_N)_{N \geq 1}$, as well as the solution $\bar{f}$ to the limiting equation \eqref{VlasovFokkerPlanck} with initial condition $\bar{f}^0$. For all $N \geq 2$ and all time $t \geq 0$, the marginals $(f_{k,N})_{1 \leq k \leq N}$ satisfy:
\begin{equation*}
  \|f_{k,N}(t,\cdot)-\bar{f}^{\otimes k}(t,\cdot)\|_{L^1({\mathbb{T}^{dk}})}\leq C\sqrt{2\,k}\,e^{-\frac{4\pi^2\sigma}{|\mathbb{T}|^2} t}\,,
\end{equation*}
for some constant $C$ depending on $\bar{f}^0$ and the implicit constant in assumption \eqref{hyp:chaos}.
\end{lemma}
\begin{proof} We fix $(k,N)\in\left(\N^\star\right)^2$ such that $N\geq 2$ and $1\leq k\leq N$. To estimate the distance between $f_{k,N}(t)$ and $\bar{f}^{\otimes k}(t)$, we decompose it as follows:
\begin{equation*}
	\|f_{k,N}(t,\cdot)-\bar{f}^{\otimes k}(t,\cdot)\|_{L^1({\mathbb{T}^{dk}})}\,\leq\,
	\|f_{k,N}(t,\cdot)-f^{\otimes k}_{\infty}(t,\cdot)\|_{L^1({\mathbb{T}^{dk}})}
	\,+\,
	\|f^{\otimes k}_{\infty}(t,\cdot)-\bar{f}^{\otimes k}(t,\cdot)\|_{L^1({\mathbb{T}^{dk}})}\,,
\end{equation*}
for all $t\geq 0$. We estimate each term on the right hand side separately, starting with $\|f_{k,N}(t,\cdot)-f^{\otimes k}_{\infty}(t,\cdot)\|_{L^1({\mathbb{T}^{dk}})}$.

First, we derive a classical relative entropy estimate that ensures the solution $f_N$ to the Liouville equation \eqref{Liouville} relaxes exponentially toward $f_\infty^{\otimes N}$. To accomplish this, we multiply equation \eqref{Liouville} by $\ln\left(\frac{f_N}{f_\infty^{\otimes N}}\right)$, integrate over $\mathbb{T}^{dN}$, and apply mass conservation for the solution to \eqref{Liouville}. This yields the following result:
\begin{equation*}
\frac{\dD}{\dD t}\, \mathcal{H}_N(f_N(t)|f^{\otimes N}_{\infty})\,=\,
\cA\,+\,\cB\,,
\end{equation*}
for all $t \geq 0$, where $\cH_N$ is defined below \eqref{hyp:chaos}, and $\cA$ and $\cB$ are given as follows:
\[
\left\{
\begin{array}{ll}
	&\ds\cA\,=\,\frac{1}{N}\,\sum_{\substack{i,j=1\\ i\neq j}}^N\int_{\mathbb{T}^{dN}}   K(x_i-x_j)\cdot \nabla_{x_i}f_{N}(t,X^N) \log{\left(\frac{f_{N}(t,X^N)}{f_\infty^{\otimes N}(X^N)}\right)}\dD X^N\,,  \\[1.em]
	&\ds\cB\,=\,\frac{\sigma }{N}\sum_{i=1}^N \int \Delta_{x_i}f_N(t,X^N)\log(f_N(t,X^N))\,\dD X^N\,. 
\end{array}
\right.
\] 
We note that $\mathcal{A}$ vanishes due to the divergence-free assumption on $K$. Specifically, by employing the relation:
\[
\nabla_{x_i} \left( f_{N} \log\left(\frac{f_{N}}{f_\infty^{\otimes N}}\right) \right) = \nabla_{x_i}\left( f_{N} \log{(f_{N})} - f_{N} \right),
\]
and integrating by parts with respect to $x_i$ in $\mathcal{A}$, we obtain:
\[
\cA\,=\,\frac{1}{2N}\,\sum_{\substack{i,j=1\\ i\neq j}}^N\int_{\mathbb{T}^{dN}}   K(x_i-x_j)\cdot \nabla_{x_i}\left(f_{N}\log{(f_{N})}-f_{N}\right)(t,X^N) \,\dD X^N\,=\,0\,.
\]
Next, we integrate by parts in $\mathcal{B}$, which leads to the following result:
\begin{equation*}
        \frac{\dD}{\dD t}\, \mathcal{H}_N(f_N(t)|f^{\otimes N}_{\infty})=-\frac{\sigma}{N} \int_{\mathbb{T}^{dN}} \frac{|\nabla_{X^N} f_N(t,X^N)|^2}{f_N(t,X^N)}\,\dD X^N.
\end{equation*}
We can lower bound the right-hand side of the latter inequality using the logarithmic Sobolev inequality associated with the probability measure $f_N$. This inequality states that (for the proof in dimension 1, see \cite[Proposition 5.7.5]{Bakry2013}, and for higher dimensions, refer to \cite[Proposition 5.2.7]{Bakry2013}):
\begin{equation*}
     \mathcal{H}_N(f_N(t)|f^{\otimes N}_{\infty})\leq \frac{|\mathbb{T}|^2}{8\pi^2}\frac{1}{N}\int_{\mathbb{T}^{dN}} \frac{|\nabla_{X^N} f_N(t,X^N)|^2}{f_N(t,X^N)}\,\dD X^N\, .
\end{equation*}
This leads to the following differential inequality:
\begin{equation*}
        \frac{\dD}{\dD t} \,\mathcal{H}_N(f_N(t)|f^{\otimes N}_{\infty})\,\leq\, -\frac{8\pi^2}{|\mathbb{T}|^2}\sigma\,\mathcal{H}_N(f_N(t)|f^{\otimes N}_{\infty})\,.
\end{equation*}
We multiply the latter inequality by $ e^{\frac{8\pi^2\sigma}{|\mathbb{T}|^2} t} $ and integrate over the interval $[0, t]$. This leads us to conclude that the relative entropy of $ f_N $ decays exponentially as $ t \rightarrow +\infty $, specifically:
\begin{equation*}
	\mathcal{H}_N(f_N(t)|f^{\otimes N}_{\infty})\leq \mathcal{H}_N(f^0_N|f^{\otimes N}_{\infty})e^{-\frac{8\pi^2\sigma}{|\mathbb{T}|^2} t
	}\,,
\end{equation*}
for all times $ t \geq 0 $ and for all $ N \geq 2 $. We then deduce exponential decay for the marginals $ (f_{k,N})_{1 \leq k \leq N} $ by employing the sub-additivity property of entropy \cite[Proposition 21]{Jabin2014}, which allows us to lower bound the left-hand side as follows:
\begin{equation*}
    \mathcal{H}_k(f_{k,N}(t)|f^{\otimes k}_{\infty})\,\leq\, \mathcal{H}_N(f_N(t)|f^{\otimes N}_{\infty})\,,
\end{equation*}
which ensures:
\begin{equation*}
	\mathcal{H}_k(f_{k,N}(t)|f^{\otimes k}_{\infty})\leq \mathcal{H}_N(f^0_N|f^{\otimes N}_{\infty})e^{-\frac{8\pi^2\sigma}{|\mathbb{T}|^2} t
	}\,.
\end{equation*}
To conclude, we apply the Csiszár-Kullback-Pinsker inequality \cite{Csiszar67,Kullback67}, which states that:
\begin{equation*}
    \|f_{k,N}(t,\cdot)-f^{\otimes k}_{\infty}(t,\cdot)\|^2_{L^1\left(\mathbb{T}^{d k}\right)} \,\leq\, 2\,k \,\mathcal{H}_k(f_{k,N}(t) | f^{\otimes k}_{\infty})\,,
\end{equation*}
to provide a lower bound for the left-hand side in the previous estimate. This results in the following estimate for the $ L^1 $ distance between $ f_{k,N} $ and $ f^{\otimes k}_{\infty} $:
\begin{equation*}
  \|f_{k,N}(t,\cdot)-f^{\otimes k}_{\infty}(t,\cdot)\|_{L^1\left(\mathbb{T}^{d k}\right)}\leq \sqrt{2\,k\,\mathcal{H}_N(f^0_N|f^{\otimes N}_{\infty})}\,e^{-\frac{4\pi^2\sigma}{|\mathbb{T}|^2}t}\,,
\end{equation*}
for all $ t \geq 0 $, all $ N \geq 2 $, and all $ 1 \leq k \leq N $.

To estimate the distance $ \|f^{\otimes k}_{\infty}(t,\cdot) - \bar{f}^{\otimes k}(t,\cdot)\|_{L^1(\mathbb{T}^{dk})} $, we employ a similar approach as in the previous paragraph. This leads to:
\begin{equation*}
	\|\bar{f}_{k}(t,\cdot)-f^{\otimes k}_{\infty}(t,\cdot)\|_{L^1\left(\mathbb{T}^{d k}\right)}\leq \sqrt{2\,k\,\mathcal{H}_1(\bar{f}^0|f_{\infty})}\,e^{-\frac{4\pi^2\sigma}{|\mathbb{T}|^2}t}\,,
\end{equation*}
for all times $t\geq 0$ and all $k\geq 1$.

We now sum the two preceding estimates and obtain:
\begin{equation}\label{fk:fkbar:intermediate}
\|f_{k,N}(t,\cdot)-\bar{f}^{\otimes k}(t,\cdot)\|_{L^1\left(\mathbb{T}^{d k}\right)}\leq \sqrt{2\,k}\left(\sqrt{\mathcal{H}_N(f^0_N|f^{\otimes N}_{\infty})}+ \sqrt{\mathcal{H}_1(\bar{f}^0|f_{\infty})}\right)e^{-\frac{4\pi^2\sigma}{|\mathbb{T}|^2}t}\,.
\end{equation}
To conclude, we estimate the initial relative entropy in the right-hand side. On one hand, since $\bar{f}^0$ satisfies assumption \eqref{f0Regularity}, we have:
\begin{equation}\label{estim:H1}\mathcal{H}_1(\bar{f}^0|f_{\infty})\leq C\,,\end{equation}
for some constant $C$ depending on $\bar{f}^0$. On the other hand, it is possible to express  $\mathcal{H}_N(f^0_N|f^{\otimes N}_{\infty})$ as follows:
\[
\mathcal{H}_N(f^0_N|f^{\otimes N}_{\infty})\,=\,
\mathcal{H}_N(f^0_N|(\bar{f}^0)^{\otimes N})\,+\,
\frac{1}{N}\int_{\mathbb{T}^{dN} 
} f_N^0\log{\left(\frac{(\bar{f}^0)^{\otimes N}}{f^{\otimes N}_{\infty}}\right)}\dD X^N
\,.
\]
The first term on the right-hand side is uniformly bounded for $N \geq 1$ according to \eqref{hyp:chaos}. For the second term, we can utilize \eqref{f0Regularity}, which guarantees that:
 \[\sup_{   \mathbb{T}^{dN}}\left|\log{\left(\frac{(\bar{f}^0)^{\otimes N}}{f^{\otimes N}_{\infty}}\right)}\right|\leq C\,N\,,\quad \textrm{and therefore}\quad 
 \frac{1}{N}\int_{\mathbb{T}^{dN} 
 } f_N^0\log{\left(\frac{(\bar{f}^0)^{\otimes N}}{f^{\otimes N}_{\infty}}\right)}\dD X^N
 \leq C\,,
 \] 
 for all $N\geq 1$ and for some constant $C$ depending on $\bar{f}^0$. Then, we find:
 \begin{equation}\label{estim:HN}
 \mathcal{H}_N(f^0_N|f^{\otimes N}_{\infty})\,\leq\,
 C
 \,,
 \end{equation}
 for all $N\geq 2$ and for some constant $C$ depending on $\bar{f}^0$ and the implicit constant in \eqref{hyp:chaos}. We then estimate $\mathcal{H}_1(\bar{f}^0|f_{\infty})$ using \eqref{estim:H1} and $\mathcal{H}_N(f^0_N|f^{\otimes N}_{\infty})$ with \eqref{estim:HN} in the right-hand side of \eqref{fk:fkbar:intermediate}, leading to the expected result:
 \[
 \|f_{k,N}(t,\cdot)-\bar{f}^{\otimes k}(t,\cdot)\|_{L^1\left(\mathbb{T}^{d k}\right)}\leq C\sqrt{2\,k}\,e^{-\frac{4\pi^2\sigma}{|\mathbb{T}|^2}t}\,,
 \]
for all $ t \geq 0 $, all $ N \geq 2 $, all $ 1 \leq k \leq N $, and for some constant $ C $ that depends on $ \bar{f}^0 $ and the implicit constant in \eqref{hyp:chaos}.
\end{proof}

Before proceeding with the proof of Corollary \ref{cr:1}, we summarize its two main steps. First, we combine \cite[Theorem 1]{JabinWang2018} with Lemma \ref{EntropyEstimatesLemma}. Specifically, \cite[Theorem 1]{JabinWang2018} guarantees that:
\[
f_{k,N}(t) = \bar{f}^{\otimes k}(t) + O\left(\frac{e^{Ct}}{\sqrt{N}}\right)\,,\quad\textrm{in}\quad L^1\left(\mathbb{T}^{dk}\right)\,,\quad\textrm{as}\quad t,N\;\longrightarrow +\infty\,,
\]
while Lemma \ref{EntropyEstimatesLemma} ensures:
 \[
 f_{k,N}(t) = \bar{f}^{\otimes k}(t) + O\left(e^{-Ct}\right)\,,\quad\textrm{in}\quad L^1\left(\mathbb{T}^{dk}\right)\,,\quad\textrm{as}\quad t,N\;\longrightarrow +\infty\,.
 \]
The combination of these two results establishes the simultaneous convergence of $ f_{k,N} $ toward $ \bar{f}^{\otimes k} $ in $ L^1 $, specifically:
 \[
 f_{k,N}(t) = \bar{f}^{\otimes k}(t) + O\left(\frac{e^{-Ct}}{N^{\alpha}}\right)\,,\quad\textrm{in}\quad L^1\left(\mathbb{T}^{dk}\right)\,,\quad\textrm{as}\quad t,N\;\longrightarrow +\infty\,,
 \]
 for some positive $\alpha>0$. We then interpolate the previous estimate with the uniform $ L^2 $-estimates derived in Theorems \ref{th:1} and \ref{th:2}, which enhances the $ L^1 $ convergence to $ L^p $ convergence for all $ 1 \leq p < 2 $.
 
\begin{proof}[Proof of Corollary \ref{cr:1}]
We fix $(k,N) \in \left(\mathbb{N}^*\right)^2$ such that $N \geq 2$ and $1 \leq k \leq N$. To estimate the distance between $f_{k,N}(t)$ and $\bar{f}_k(t)$, we begin with \cite[Theorem 1]{JabinWang2018}, which establishes entropic propagation of chaos with exponential growth:
\begin{equation*}
	\mathcal{H}_k(f_{N}(t)|\bar{f}^{\otimes N}(t))\,\leq \, 2\,k\,e^{C_1t}\left(\mathcal{H}_N(f_N^0|(\bar{f}^0)^{\otimes N})+\frac{1}{N}\right),
\end{equation*}
for some constant $C_1$ that depends on $K$, $d$ , $\sigma$, the implicit constant in \eqref{hyp:chaos}, and the following norms of $\bar{f}$ on the interval $[0,t]$:
\[
\sup_{s\in[0,t]}\sup_{p\geq 1}\, \frac{1}{p}\,\| \nabla_x^2\bar{f}(s,\cdot)\|_{L^p(\mathbb{T}^d)}\,,\quad \sup_{s\in[0,t]}\|\bar{f}(s,\cdot) \|_{W^{1,\infty}(\mathbb{T}^d)}\,,\quad \textrm{and} \quad\inf_{s\in[0,t]}\inf_{x\in\mathbb{T}^d} \bar{f}(s,x)\,.
\]
We then invoke \cite[Theorem 2]{Guillin2024}, which establishes that $\bar{f}(t)$ and all its derivatives remain uniformly bounded for $t \in \mathbb{R}^+$ (with $\inf_{t \geq 0} \inf_{\mathbb{T}^d} \bar{f}(t) > 0$). Consequently, the constant $C_1$ is independent of time. Following the approach used in the proof of Lemma \ref{EntropyEstimatesLemma}, we utilize the sub-additivity of entropy \cite[Proposition 21]{Jabin2014} and then apply the Csiszár-Kullback inequality \cite{Csiszar67,Kullback67} to provide a lower bound for the left-hand side, resulting in:
\begin{equation*}
    \|f_{k,N}(t,\cdot)-\bar{f}^{\otimes k}(t,\cdot)\|_{L^1(\mathbb{T}^{dk})}^2\,\leq \, 2\,k\,e^{C_1t}\left(\mathcal{H}_N(f_N^0|(\bar{f}^0)^{\otimes N})+\frac{1}{N}\right).
\end{equation*}
Next, we take the square root of the previous estimate and bound the initial relative entropy on the right-hand side using assumption \eqref{hyp:chaos}, which leads to:
\begin{equation*}
	\|f_{k,N}(t,\cdot)-\bar{f}^{\otimes k}(t,\cdot)\|_{L^1(\mathbb{T}^{dk})}\,\leq \, C\sqrt{\frac{k}{N}}\,e^{C_1t}\,.
\end{equation*}
where $C$ depends on the implicit constant in \eqref{hyp:chaos}.
Subsequently, we apply Lemma \ref{EntropyEstimatesLemma}, resulting in:
\begin{equation}\label{L1DifferenceEstimateEntropy}
    \|f_{k,N}(t,\cdot)-\bar{f}^{\otimes k}(t,\cdot)\|_{L^1(\mathbb{T}^{dk})}\leq 
    C\sqrt{k}\min\left\{ \frac{e^{C_1t}}{\sqrt{N}}, e^{-C_2t}\right\},
\end{equation}
for any $t\geq0$, where $C_2=\frac{4\pi^2\sigma}{|\mathbb{T}|^2}$, and for some constant $C$ depending on $\bar{f}^0$ and the implicit constant in \eqref{hyp:chaos}. To evaluate the minimum on the right-hand side, we define $ T_N = \frac{\log(N)}{2(C_1 + C_2)} $. Subsequently, we select a small positive constant $ \gamma $ such that $ 0 < \gamma < \frac{C_2}{C_1 + C_2} $, and we confirm that the following inequalities:
\[
\frac{e^{C_1 t}}{\sqrt{N}}\,\leq\,
\frac{e^{-(C_2-(C_1+C_2)\gamma)\,t}}{\sqrt{N}^{\,\gamma}}
\quad \textrm{when}\quad t\leq T_N\,,\]
and
\[e^{-C_2t}\,\leq\,
\frac{e^{-(C_2-(C_1+C_2)\gamma)\,t}}{\sqrt{N}^{\,\gamma}}
\quad \textrm{when}\quad t\geq  T_N\,
\]
hold. Combining the two estimates, we conclude that: 
\[
\min\left\{ \frac{e^{C_1t}}{\sqrt{N}}, e^{-C_2t}\right\}
\,\leq\,
\frac{e^{-(C_2-(C_1+C_2)\gamma)\,t}}{N^{\gamma}}\,.
\]
is verified, for all $t\geq 0$. Substituting the previous inequality into \eqref{L1DifferenceEstimateEntropy}, we obtain:
\begin{equation}\label{L1Convergence}
    \|f_{k,N}(t,\cdot)-\bar{f}^{\otimes k}(t,\cdot)\|_{L^1(\mathbb{T}^{dk})}\leq C\sqrt{k}\,\frac{e^{-\beta\,t}}{N^{\gamma}},
\end{equation}
for some positive constants $C, \beta, \gamma$ that only depend on $K$, $d$, $\sigma$, the implicit constant in \eqref{hyp:chaos}, and the norms of $\bar{f}^0$. To conclude the proof, we apply H\"{o}lder inequality, which ensures:
\[
 \|f_{k,N}(t,\cdot)-\bar{f}^{\otimes k}(t,\cdot)\|_{L^p(\mathbb{T}^{dk})}
 \,\leq\,
  \|f_{k,N}(t,\cdot)-\bar{f}^{\otimes k}(t,\cdot)\|_{L^1(\mathbb{T}^{dk})}^{\frac{2-p}{p}}
  \|f_{k,N}(t,\cdot)-\bar{f}^{\otimes k}(t,\cdot)\|_{L^2(\mathbb{T}^{dk})}^{2\frac{p-1}{p}}\,,
\]
for all $1<p<2$. On the right-hand side of the previous inequality, we estimate the $ L^1 $-norm using \eqref{L1Convergence}, yielding:
\begin{equation}\label{DifLpL2}
	\|f_{k,N}(t,\cdot)-\bar{f}^{\otimes k}(t,\cdot)\|_{L^p(\mathbb{T}^{dk})}
	\,\leq\,
	\left(\frac{C\sqrt{k}e^{-\beta t}}{N^\gamma}\right)^{\frac{2-p}{p}}
	\|f_{k,N}(t,\cdot)-\bar{f}^{\otimes k}(t,\cdot)\|_{L^2(\mathbb{T}^{dk})}^{2\frac{p-1}{p}}\,
\end{equation}
To bound the $ L^2 $-norm, we first apply the triangle inequality, allowing us to express it as:
\begin{equation*}
	\|f_{k,N}(t,\cdot)-\bar{f}^{\otimes k}(t,\cdot)\|_{L^2(\mathbb{T}^{dk})}\leq 	\|f_{k,N}(t,\cdot)\|_{L^2(\mathbb{T}^{dk})}+  \|\bar{f}^{\otimes k}(t,\cdot)\|_{L^2(\mathbb{T}^{dk})}.
\end{equation*}
To bound $ \|f_{k,N}(t,\cdot)\|_{L^2(\mathbb{T}^{dk})} $, we can utilize either \eqref{MarginalsBound} from Theorem \ref{th:1} or \eqref{MarginalsBoundH-1} from Theorem \ref{th:2}. Specifically, we have:
\begin{subequations}
	\begin{numcases}{}			\label{BoundFromTh1}
		\|f_{k,N}(t,\cdot)\|_{L^2(\mathbb{T}^{dk})}\leq C k^{\alpha k}\;, \textrm{ under assumptions of Theorem \ref{th:1} , or}\\[0.8em]
		\label{BoundFromTh2}
			\|f_{k,N}(t,\cdot)\|_{L^2(\mathbb{T}^{dk})}\leq C R^k\;,\textrm{ under assumptions of Theorem \ref{th:2},}\,
	\end{numcases}
\end{subequations}
 for some $C$ depending on the implicit constant given by  \eqref{MarginalsBound} or \eqref{MarginalsBoundH-1}, respectively.

Now, we proceed to estimate $\|\bar{f}^{\otimes k}(t)\|_{L^2(\mathbb{T}^{dk})}$. By the definition of the tensorized distribution $\bar{f}^{\otimes k}$ and the fact that $\bar{f}$ solves the equation \eqref{VlasovFokkerPlanck}, we find that $\bar{f}^{\otimes k}$ satisfies:
\begin{equation*}
	\partial_t \bar{f}^{\otimes k}+\sum_{i=1}^{k} \udiv_{x_i}\left( K\star\bar{f}(t,x_i)\bar{f}^{\otimes k}\right)=\sigma\sum_{i=1}^{k}\Delta_{x_i}\bar{f}^{\otimes k}.
\end{equation*}
We multiply the above equation by $\bar{f}^{\otimes k}$ and integrate over $\mathbb{T}^{dk}$  to obtain:
\begin{equation*}
	\frac{1}{2}\frac{\dD}{\dD t} \|\bar{f}^{\otimes k}(t,\cdot)\|_{L^2(\mathbb{T}^{dk})}^2=\cA+\cB,
\end{equation*}
where $\cA$ and $\cB$ are given by:
\begin{equation*}
\left\{
\begin{array}{ll}
	&\ds\cA\,=\,-\,\sum_{i=1}^{k} \int_{\mathbb{T}^{dk}} \udiv_{x_i}\left( K\star\bar{f}(t,x_i)\bar{f}^{\otimes k}(t,X^k)\right)\bar{f}^{\otimes k}(t,X^k)\,\dD X^k,  \\[1.em]
	&\ds\cB\,=\, \sigma\sum_{i=1}^{k}\int_{\mathbb{T}^{dk}}\Delta_{x_i}\bar{f}^{\otimes k}(t,X^k)\bar{f}^{\otimes k}(t,X^k)\,\dD X^k. 
\end{array}
\right.
\end{equation*}
Observe that $\mathcal{A}$ vanishes due to the divergence-free assumption. In fact, the relation $\bar{f}^{\otimes k} \nabla_{x_i} \bar{f}^{\otimes k} = \nabla_{x_i} \left| \bar{f}^{\otimes k} \right|^2 / 2$ allows us to apply integration by parts with respect to $x_i$ in $\mathcal{A}$, yielding:
\begin{equation*}
\cA\,=\,\frac{1}{2}\,\sum_{i=1}^k\int_{\mathbb{T}^{dk}}   K\star\bar{f}(t,x_i)\cdot \nabla_{x_i}\left| \bar{f}^{\otimes k}\right|^2(t,X^k) \,\dD X^k\,=\,0\,.
\end{equation*}
On the other hand, integrating by parts in $\cB$, we deduce:
\begin{equation*}
	\cB\,=-\, \sigma \|\nabla_{X^{k}}\bar{f}^{\otimes k}(t,\cdot)\|_{L^2(\mathbb{T}^{dk})}^2\leq 0.
\end{equation*}
Consequently, the function $\|\bar{f}^{\otimes k}(t,\cdot)\|_{L^2(\mathbb{T}^{dk})}$ is non-increasing in $t$, which leads to the inequality $\|\bar{f}^{\otimes k}(t,\cdot)\|_{L^2(\mathbb{T}^{dk})} \leq \|\bar{f}^{\otimes k}(0)\|_{L^2(\mathbb{T}^{dk})}$. By the definition of $\bar{f}^{\otimes k}$, this implies:
\begin{equation}\label{BoundL2Tensorized}
	\|\bar{f}^{\otimes k}(t,\cdot)\|_{L^2(\mathbb{T}^{dk})}\leq \|\bar{f}^0\|_{L^2(\mathbb{T}^{dk})}^{k}.
\end{equation}
 Inserting estimates \eqref{BoundFromTh1}, \eqref{BoundFromTh2} and  \eqref{BoundL2Tensorized} into inequality \eqref{DifLpL2}, we conclude:
\[
\|f_{k,N}(t,\cdot)-\bar{f}^{\otimes k}(t,\cdot)\|_{L^{p}(\mathbb{T}^{dk})}\leq X_k^\frac{2(p-1)}{p}\left(\frac{C\sqrt{k}e^{-\beta t}}{N^\gamma}\right)^{\frac{2-p}{p}}\,,
\]
where either $ X_k = C k^{\alpha k} + \|\bar{f}^0\|_{L^2(\mathbb{T}^{dk})}^k $ if we are under the assumptions of Theorem \ref{th:1}. Alternatively, $ X_k = C R^k + \|\bar{f}^0\|_{L^2(\mathbb{T}^{dk})}^{k} $ applies under the assumptions of Theorem \ref{th:2}. In both cases, $ C $ represents a constant that depends on the implicit constants provided by \eqref{MarginalsBound} or \eqref{MarginalsBoundH-1}, respectively.
\end{proof}
\section{Failure of uniform propagation of chaos}\label{Sec:counter:ex}
Our counterexample arises in a seemingly favorable configuration: we fix the dimension to $d=1$ and consider the smooth Kuramoto interaction kernel:
\[
K(x)\,=\,-\sin(x)\,,
\]
where $ x\in \T$. With this choice of $K$, two distinct stationary states exist for the limiting equation \eqref{VlasovFokkerPlanck} when $\sigma > 0$ is sufficiently small \cite[Theorem $4.1$]{Ha_Shim_Zhang20}: the homogeneous stationary state $ \bar{f}_{1,\infty} = 1/|\mathbb{T}| $ and a non-homogeneous stationary state, denoted $ \bar{f}_{2,\infty} \in \mathcal{C}^2(\mathbb{T}) $. We now consider the solutions $ (f_N(t))_{N \geq 2} $ to the Liouville equation \eqref{Liouville}, along with the constant solution $ \bar{f}(t) $ to \eqref{VlasovFokkerPlanck}, under the following chaotic initial configurations:
\[
f^0_N\,=\,
\bar{f}_{2,\infty}^{\otimes N}
\,,\quad \textrm{and}\quad
\bar{f}^0
\,=\,
\bar{f}_{2,\infty}
\,.
\] 
Next, we establish that $ f_N(t) $ converges to the unique stationary state $ f_{N,\infty} $ of the Liouville equation \eqref{Liouville}, which is given by:
\[
f_{N,\infty}(x_1,\dots, x_N)\,=\,
c_{\infty,N}\,\exp{\left(\frac{1}{\sigma N}\sum_{i,j=1}^N \cos(x_i-x_j)\right)}\,,
\]
where $c_{\infty,N}$ is a normalizing constant. Indeed, a standard relative entropy estimate, detailed the proof of Lemma \ref{EntropyEstimatesLemma}, gives:
\begin{equation*}
	\frac{\dD}{\dD t}\, \mathcal{H}_N(f_N(t)|f_{N,\infty})=-\frac{\sigma}{N} \int_{\mathbb{T}^{N}} \left|\nabla_{X^N} \ln{\frac{f_N(t,X^N)}{f_{N,\infty}(X^N)}}\right|^2 f_N(t,X^N)\,\dD X^N\,.
\end{equation*}
Then, we apply \cite[Lemma $2$]{Guillin2024}, which ensures that, since the distribution $f_{N,\infty}$ is bounded both from above and below, it satisfies a logarithmic Sobolev inequality, which yields:
\begin{equation*}
	\frac{\dD}{\dD t}\, \mathcal{H}_N(f_N(t)|f_{N,\infty})\leq-c_N\,\mathcal{H}_N(f_N(t)|f_{N,\infty})\,,
\end{equation*}
for some constant $c_N$ depending on $N$. This  inequality confirms the convergence of $ f_N(t) $ to $ f_{N,\infty} $, along with the following estimate, which is derived using the same steps as in the proof of Lemma \ref{EntropyEstimatesLemma}:
\begin{equation*}
	\lim_{t\rightarrow +\infty} \cH_1\left(f_{1,N}(t),f_{1,\infty}\right)
	\,=\, 0\,.
\end{equation*}
In the above relation, we note that the first marginal $f_{1,\infty}$ is in fact homogeneous, that is $f_{1,\infty} =\bar{f}_{1,\infty}$. Indeed, we have:
\begin{equation*}
\begin{split}
	f_{1,\infty}(x+y)&=
	\int_{\mathbb{T}^{N-1}} 
	\Tilde{c}_{\infty,N} \exp{\left(\frac{1}{\sigma N}\left(\sum_{j=2}^N \cos(x+y-x_j)+\sum_{i,j=2}^N \cos(x_i-x_j)\right) \right)}
	\,\dD X^{N-1}\\
	&=\,
	f_{1,\infty}(x)\,,
\end{split}
\end{equation*}
for all $(x,y)\in\T^2$, where the second inequality follows from the  change of coordinates $x_j\leftarrow x_j-y$. This gives:
\begin{equation*}
	\lim_{t\rightarrow +\infty} \cH_1\left(f_{1,N}(t),\bar{f}_{1,\infty}\right)
	\,=\, 0\,.
\end{equation*}
We lower bound the relative entropy with the $L^1$-norm in the latter relation according to the Csiszár-Kullback inequality \cite{Csiszar67,Kullback67}, which yields
\begin{equation}\label{cv:f1N:f1inf}
	\lim_{t\rightarrow +\infty} \|f_{1,N}(t,\cdot)-\bar{f}_{1,\infty}\|_{L^1(\mathbb{T})}
	\,=\, 0\,.
\end{equation}
Next, we use the triangular inequality to lower bound the distance between $f_{1,N}$ and $\bar{f}$
\[ \|\bar{f}(t,\cdot)-f_{1,N}(t,\cdot)\|_{L^1(\mathbb{T})}
\,\geq\,
\|\bar{f}(t,\cdot)-\bar{f}_{1,\infty}\|_{L^1(\mathbb{T})}
-
\|\bar{f}_{1,\infty}-f_{1,N}(t,\cdot)\|_{L^1(\mathbb{T})}
\,,
\]
and take the $\limsup$ as $|(t,N)|\rightarrow +\infty$ in the latter estimate, which yields
\begin{equation*}
\begin{split}
	\limsup_{|(t,N)|\rightarrow +\infty} \|\bar{f}(t,\cdot)-f_{1,N}(t,\cdot)\|_{L^1(\mathbb{T})}
	\,\geq\,&
	\limsup_{|(t,N)|\rightarrow +\infty}
	\|\bar{f}(t,\cdot)-\bar{f}_{1,\infty}\|_{L^1(\mathbb{T})}\\
	-&
	\liminf_{|(t,N)|\rightarrow +\infty}
	\|\bar{f}_{1,\infty}-f_{1,N}(t,\cdot)\|_{L^1(\mathbb{T})}\,.
\end{split}
\end{equation*}
According to \eqref{cv:f1N:f1inf}, the second term on the latter right hand side is $0$, which yields
\[
\limsup_{|(t,N)|\rightarrow +\infty} \|\bar{f}(t,\cdot)-f_{1,N}(t,\cdot)\|_{L^1(\mathbb{T})}
\,\geq\,
\limsup_{|(t,N)|\rightarrow +\infty}
\|\bar{f}(t,\cdot)-\bar{f}_{1,\infty}\|_{L^1(\mathbb{T})}\,.
\]
Since $\bar{f}(t)=\bar{f}_{2,\infty}$, for all $t\geq 0$, and since $\bar{f}_{1,\infty}\neq \bar{f}_{2,\infty}$ we obtain our result, that is:
\[
\limsup_{|(t,N)|\to \infty} \|f_{1,N}(t,\cdot)-\bar{f}(t,\cdot)\|_{L^1(\mathbb{T})}
\,\geq\,
\limsup_{|(t,N)|\to \infty}
\|\bar{f}_{1,\infty}-\bar{f}_{2,\infty}\|_{L^1(\mathbb{T})}>0\,.
\]

\appendix
\section{Sobolev inequality on \texorpdfstring{$\mathbb{T}^n$}{TEXT}: Proof of Theorem \texorpdfstring{\ref{Sob:ineq:Pi}}{TEXT}} \label{Proof:th:sob:ineq}
In this section, we prove the Sobolev inequality stated in Theorem \ref{Sob:ineq:Pi}. We begin by addressing the case where the torus has a size of one, i.e., $ |\mathbb{T}| = 1 $, and then we extend our results to the general case. Thus, we establish the inequality outlined in Theorem \ref{Sob:ineq:Pi} for a fixed function $ f \in H^1(\mathbb{T}^n) $ with $ |\mathbb{T}| = 1 $.

We denote $ f $ as the periodic extension of the function to the entire space, which implies $ f \in H^1_{\text{loc}}(\mathbb{R}^n) $. Next, we introduce a Lipschitz cutoff $ \varphi_n \in \mathcal{C}^0_c(\mathbb{R}^n) $ with compact support, which will be determined later. We then apply the Sobolev inequality on $ \mathbb{R}^n $, which guarantees \cite{Aubin1976, Talenti1976} that:
\[
\left\|f\,
\varphi_n
\right\|^2_{L^{2^\star}(\R^n)}
\,\leq
\,
K_n^2
\left\|
\nabla_{X^n}
\left(f\,
\varphi_n\right)
\right\|^2_{L^{2}(\R^n)}\,,
\]
where $K_n$ and $2^\star$ are given in Theorem \ref{Sob:ineq:Pi}. We then apply the Leibniz rule for products along with Young's inequality to estimate the right-hand side, resulting in the following expression:
\begin{equation}\label{sob:0}
\left\|f\,
\varphi_n
\right\|^2_{L^{2^\star}(\R^n)}
\,\leq
\,
2 K_n^2\left(
\left\|
\varphi_n
\nabla_{X^n}
f
\right\|^2_{L^{2}(\R^n)}
+
\left\|
f\,
\nabla_{X^n}
\varphi_n
\right\|^2_{L^{2}(\R^n)}
\right)\,.
\end{equation}
Since $ f $ is periodic, we can express each of the norms over $ \mathbb{R}^n $ in the previous estimate as norms over the torus. For instance, we have:
\begin{align*}
\left\|
\varphi_n
\nabla_{X^n}
f
\right\|^2_{L^{2}(\R^n)}
\,&=\,
\int_{\R^n}
\left|
\varphi_n(X^n)
\nabla_{X^n}
f(X^n)\right|^2
\dD X^n\\
&=\,
\int_{[0,1]^n}
\left|\nabla_{X^n}
f(X^n)\right|^2
\sum_{k\in\Z^n}
\left|
\varphi_n(k+X^n)\right|^2
\dD X^n\,.
\end{align*}
By applying the above relation to all three terms in \eqref{sob:0}, we obtain:
\begin{align*}
\biggl(
\int_{[0,1]^n}
\left|
f(X^n)\right|^{2^\star}
\omega^{\star}_n(X^n)\,
&\dD X^n
\biggr)^{\frac{2}{2^\star}}
\leq \\
&2
K_n^2
\int_{[0,1]^n}
\left|\nabla_{X^n}
f(X^n)\right|^2
\omega_n(X^n)
+
\left|
f(X^n)\right|^2
\Omega_n(X^n)\,
\dD X^n
,
\end{align*}
where $\omega_n$, $\omega^\star_n$ and $\Omega_n$ are the periodic functions given by:
\[
\left\{
\begin{array}{ll}
     &\ds\omega_n\left(X^n\right)\,=\,\sum_{k\in\Z^n}
\left|
\varphi_n(k+X^n)\right|^2  \\[1.em]
     &\ds\omega^\star_n\left(X^n\right)\,=\,\sum_{k\in\Z^n}
\left|
\varphi_n(k+X^n)\right|^{2^\star} \\[1.em]
     &\ds\Omega_n\left(X^n\right)\,=\,\sum_{k\in\Z^n}
\left|\nabla_{X^n}
\varphi_n(k+X^n)\right|^2
\end{array}
\right.
\;\;,\quad\forall \,X^n\in[0,1]^n\,.
\]
We note that all the computations above remain valid if we replace $\varphi_n$ with its translation $\varphi_n(\cdot - Y^n)$ for any vector $Y^n \in [0, 1]^n$. Consequently, the following estimate holds for all $Y^n \in [0, 1]^n$:
\[
\left|
f\right|^{2^\star}
\star
\omega^{\star}_n\left(Y^n\right)
^{\frac{2}{2^\star}}
\leq
2
K_n^2
\left(
\left|\nabla_{X^n}
f\right|^2\star
\omega_n\left(Y^n\right)
+
\left|
f\right|^2\star
\Omega_n\left(Y^n\right)
\right)
\,,
\]
where $\star$ denotes the convolution product on $\mathbb{T}^n$. By integrating the latter relation with respect to $Y^n \in \mathbb{T}^n$, we find:
\[
\left\|
\left|
f\right|^{2^\star}
\star
\omega^{\star}_n
\right\|_{L^{\frac{2}{2^\star}}\left(\mathbb{T}^n\right)}^{\frac{2}{2^\star}}
\leq
2
K_n^2
\left(
\left\|
\left|\nabla_{X^n}
f\right|^2\star
\omega_n\right\|_{L^1\left(\mathbb{T}^n\right)}
+
\left\|
\left|
f\right|^2\star
\Omega_n\right\|_{L^1\left(\mathbb{T}^n\right)}
\right)
\,.
\]
To estimate the right-hand side, we make two observations. First, since $\omega_n$ and $\Omega_n$ are non-negative functions, we have:
\[
\left\{
\begin{array}{ll}
     \ds\left\|
\left|\nabla_{X^n}
f\right|^2\star
\omega_n\right\|_{L^1\left(\mathbb{T}^n\right)}
\,&=\,
\left\|
\nabla_{X^n}
f
\right\|_{L^2\left(\mathbb{T}^n\right)}^2
\left\|
\omega_n\right\|_{L^1\left(\mathbb{T}^n\right)} \,, \\[1.em]
\ds\left\|
\left|
f\right|^2\star
\Omega_n\right\|_{L^1\left(\mathbb{T}^n\right)}
\,&=\,
\left\|
f
\right\|_{L^2\left(\mathbb{T}^n\right)}^2
\left\|
\Omega_n\right\|_{L^1\left(\mathbb{T}^n\right)}\,.
\end{array}
\right.
\]
Second, the $L^1$-norm of $\omega_n$ can be explicitly expressed in terms of $\varphi_n$, as follows:
\[
\left\|
\omega_n\right\|_{L^1\left(\mathbb{T}^n\right)}
\,=\,
\int_{[0,1]^n}
\sum_{k\in\Z^n}
\left|
\varphi_n(k+X^n)\right|^2
\dD X^n
\,=\,
\int_{\R^n}
\left|
\varphi_n(X^n)\right|^2
\dD X^n
\,=\,
\left\|
\varphi_n\right\|_{L^2\left(\R^n\right)}^2.
\]
Using the same argument, we also find $\left\|
\Omega_n\right\|_{L^1\left(\mathbb{T}^n\right)}
\,=\,
\left\|
\nabla_{X^n}\varphi_n\right\|_{L^2\left(\R^n\right)}^2.$
Taking these two observations into account, our previous estimate is effectively equivalent to the following: 
\begin{equation}\label{proof:sob:ineq:tech:estimate}
\begin{split}
	\left\|
\left|
f\right|^{2^\star}
\star
\omega^{\star}_n
\right\|_{L^{\frac{2}{2^\star}}\left(\mathbb{T}^n\right)}^{\frac{2}{2^\star}}
\leq
2
K_n^2
\Big(&
\left\|
\nabla_{X^n}
f
\right\|_{L^2\left(\mathbb{T}^n\right)}^2
\left\|
\varphi_n\right\|_{L^2\left(\R^n\right)}^2
\\ &+
\left\|
f
\right\|_{L^2\left(\mathbb{T}^n\right)}^2
\left\|
\nabla_{X^n}\varphi_n\right\|_{L^2\left(\R^n\right)}^2
\Big)
\,.
\end{split}
\end{equation}
To estimate the right-hand side of the inequality, we define $\varphi_n$ as a piecewise linear function and explicitly compute both its $L^2$-norm and the $L^2$-norm of its gradient. Specifically, we set $\varphi_n$ for all $X^n = (x_1, \ldots, x_n) \in \mathbb{R}^n$ as follows:
\[
\varphi_n\left(X^n\right)\,=\,
\varphi(x_1)\cdots\varphi(x_n)\,,
\textrm{ with }
\varphi(x)\,=\,
\left\{
\begin{array}{ll}
     \ds 1
\quad &\ds\textrm{if}\quad |x|\leq \frac{1}{2}\\[1.5em]
\ds 0
\quad &\ds\textrm{if}\quad \frac{1+\eta}{2}\leq|x| \\[1.5em]
 \ds \frac{\eta+1}{\eta}-\frac{2|x|}{\eta}\quad&\ds
\textrm{if}\quad \frac{1}{2}\leq |x|\leq \frac{1+\eta}{2}
\end{array}
\right.\;,
\]
for some small parameter $\eta>0$ to be fixed later on. Since $\varphi_n$ is now tensorized, we have:
\[
\left\|
\varphi_n\right\|_{L^2\left(\R^n\right)}^2
\,=\,
\left\|
\varphi\right\|_{L^2\left(\R\right)}^{2n}\,.
\]
Furthermore, since $\varphi$ is constrained between $0$ and $1$, we can bound the square of its $L^2$-norm by the size of its support, which is given by the interval $[-(1+\eta)/2, (1+\eta)/2]$. This results in the following estimate for all $n \geq 1$:
\begin{equation}\label{proof:sob:ineqRHS:1}
\left\|
\varphi_n\right\|_{L^2\left(\R^n\right)}^2
\,\leq\,
(1+\eta)^n\,.
\end{equation}
We now estimate the $L^2$-norm of $\nabla_{X^n}\varphi_n$. First, we observe that $\varphi_n$ is invariant under the permutation of coordinates, which means it satisfies property \eqref{hyp:f0}. Consequently, we have:
\[
\left\|
\nabla_{X^n}\varphi_n\right\|_{L^2\left(\R^n\right)}^2
\,=\,
n
\left\|
\partial_{x_1}\varphi_n\right\|_{L^2\left(\R^n\right)}^2\,.
\]
Second, we leverage the fact that $\varphi_n$ is tensorized along with our previous estimate of $\left\|\varphi_{n-1}\right\|_{L^2\left(\R^{n-1}\right)}^2$ to obtain:
\[
\left\|
\nabla_{X^n}\varphi_n\right\|_{L^2\left(\R^n\right)}^2
\,=\,
n
\left\|
\varphi_{n-1}\right\|_{L^2\left(\R^{n-1}\right)}^2
\left\|
\varphi'\right\|_{L^2\left(\R\right)}^2
\leq
n
\left(
1+\eta\right)^{n-1}
\left\|
\varphi'\right\|_{L^2\left(\R\right)}^2\,.
\]
We calculate the norm of $\varphi'$ in the previous estimate using the following relation, which is verified for all $x \in \mathbb{R}$:
\[
\varphi'(x)
\,=\,
\frac{2}{\eta}\,
\mathds{1}_{\left[-\frac{1+\eta}{2},-\frac{1}{2}\right]}(x)
\,-\,
\frac{2}{\eta}\,
\mathds{1}_{\left[\frac{1}{2},\frac{1+\eta}{2}\right]}(x)\,.
\]
Hence, we obtain $\left\|
\varphi'\right\|_{L^2\left(\R\right)}^2
\,=\,
4/\eta
$, and deduce the following bound for the gradient of $\varphi_n$:
\begin{equation}\label{proof:sob:ineqRHS:2}
\left\|
\nabla_{X^n}\varphi_n\right\|_{L^2\left(\R^n\right)}^2
\,
\leq\,
\frac{4\,n}{\eta}
\left(
1+\eta\right)^{n-1}\,.
\end{equation}
In the final step of this proof, we seek a lower bound for the left-hand side of \eqref{proof:sob:ineq:tech:estimate}. Our approach involves estimating the infimum value of $\omega_n^\star$. To achieve this, we make the following two observations:
\begin{enumerate}
	\item whenever $k+X^n$ lies in $[-1/2,1/2]^n$, we have $\varphi_n(k+X^n)=1$ ;\\[-0.7em]
	\item for all $X^n\in[0,1]^n$, there exists $k\in\Z^n$ such that $(k+X^n)\in [-1/2,1/2]^n$.
\end{enumerate}
 Consequently, we can establish a lower bound for $\omega_n$ of at least one: for all $X^n \in [0,1]^n$, we have:
\[
\omega^\star\left(X^n\right)\,=\,\sum_{k\in\Z^n}
\left|
\varphi_n(k+X^n)\right|^{2^\star} 
\,\geq\,1\,.
\]
Substituting this estimate into the left-hand side of \eqref{proof:sob:ineq:tech:estimate} and utilizing the fact that $|\mathbb{T}| = 1$, we obtain:
\begin{equation}\label{proof:sob:ineqlHS}
\left\|
f
\right\|_{L^{2^\star}\left(\mathbb{T}^n\right)}^2
\leq
\left\|
\left|
f\right|^{2^\star}
\star
\omega^{\star}_n
\right\|_{L^{\frac{2}{2^\star}}\left(\mathbb{T}^n\right)}^{\frac{2}{2^\star}}\,.
\end{equation}
We now incorporate \eqref{proof:sob:ineqRHS:1}, \eqref{proof:sob:ineqRHS:2}, and \eqref{proof:sob:ineqlHS} into the estimate \eqref{proof:sob:ineq:tech:estimate}, yielding:
\begin{equation*}
	\left\|
	f
	\right\|_{L^{2^\star}\left(\mathbb{T}^n\right)}^2
	\leq
	2
	K_n^2
	\left(
	(1+\eta)^n
	\left\|
	\nabla_{X^n}
	f
	\right\|_{L^2\left(\mathbb{T}^n\right)}^2
	+
	\frac{4\,n}{\eta}
	\left(
	1+\eta\right)^{n-1}
	\left\|
	f
	\right\|_{L^2\left(\mathbb{T}^n\right)}^2
	\right)
	\,.
\end{equation*}
Next, we set $\eta = \frac{1}{n}$ and take the square root of the resulting inequality. After performing some straightforward calculations, we obtain the estimate stated in Theorem \ref{Sob:ineq:Pi} for the case when $|\mathbb{T}| = 1$:
\begin{equation}\label{sob:pi1}
	\left\|
	f
	\right\|_{L^{2^\star}\left(\mathbb{T}^n\right)}
	\leq
	\sqrt{2e}
	K_n
	\left(
	\left\|
	\nabla_{X^n}
	f
	\right\|_{L^2\left(\mathbb{T}^n\right)}^2
	+
4\,n^2
	\left\|
	f
	\right\|_{L^2\left(\mathbb{T}^n\right)}^2
	\right)^{\frac{1}{2}}
	\,.
\end{equation}
Using a scaling argument, we can extend the previous inequality to the general case where $|\mathbb{T}| > 0$. We denote the torus of length $|\mathbb{T}| = L > 0$ as $[0,L]^n_\text{per}$. For every $g \in H^1([0,L]^n_\text{per})$, we then proceed as follows:
\[
\left(
f\,:\,X^n\in[0,1]^n_\textrm{per}\longmapsto g(L X^n)\right)\in H^1\left([0,1]^n_\textrm{per}\right)\,.
\]
The Lebesgue norms of $ f $ and $ g $ are explicitly connected through a linear change of variables, namely:
\[
\begin{split}
\left\|
f
\right\|_{L^{p}\left([0,1]^n_\textrm{per}\right)}
\,&=\,
L^{-\frac{n}{p}}
\left\|
g
\right\|_{L^{p}\left([0,L]^n_\textrm{per}\right)}\,,\\
\left\|
\nabla_{X^n} f
\right\|_{L^{p}\left([0,1]^n_\textrm{per}\right)}
\,&=\,
L^{-\frac{n}{p}+1}
\left\|
\nabla_{X^n} 
g
\right\|_{L^{p}\left([0,L]^n_\textrm{per}\right)}\,,
\end{split}
\]
for all $p\geq 1$. Hence, we can apply the Sobolev inequality \eqref{sob:pi1} to $ f \in H^1\left([0,1]^n_\textrm{per}\right) $ and substitute the norms of $ f $ and $ g $ using the previous relation. This results in:
\begin{equation*}
	L^{-\frac{n}{2^\star}}
	\left\|
	g
	\right\|_{L^{2^\star}\left([0,L]^n_\textrm{per}\right)}
	\leq
	\sqrt{2e}
	K_n
	\left(
	L^{-n+2}	
	\left\|
	\nabla_{X^n}
	g
	\right\|_{L^2\left([0,L]^n_\textrm{per}\right)}^2
	+
	L^{-n}
	4\,n^2
	\left\|
	g
	\right\|_{L^2\left([0,L]^n_\textrm{per}\right)}^2
	\right)^{\frac{1}{2}}
	\,.
\end{equation*}
To conclude the proof, we multiply the preceding estimate by $ L^{\frac{n}{2^\star}} $ and utilize the relation $ \frac{2}{2^\star} = 1 - \frac{2}{n} $ to obtain the desired result:
\begin{equation*}
	\left\|
	g
	\right\|_{L^{2^\star}\left([0,L]^n_\textrm{per}\right)}
	\leq
	\sqrt{2e}
	K_n
	\left(
	\left\|
	\nabla_{X^n}
	g
	\right\|_{L^2\left([0,L]^n_\textrm{per}\right)}^2
	+
	\frac{4\,n^2}{L^2}
	\left\|
	g
	\right\|_{L^2\left([0,L]^n_\textrm{per}\right)}^2
	\right)^{\frac{1}{2}}
	\,.
\end{equation*}

\section{Interpolation between 
\texorpdfstring{$W^{1,\infty}$}{TEXT} and \texorpdfstring{$L^2$}{TEXT}}\label{App:Interp}
In this section, we establish relation \eqref{prop:interp} below, which provides justification for the interpolation inequality \eqref{estim:T:3} used in the proof of Lemma \ref{Lemma:hierarchy:term}:
\begin{equation}\label{prop:interp}
\left(L^2(\mathbb{T}^d),W^{1,\infty}(\mathbb{T}^{d})\right)_{1-\theta,\frac{2}{\theta}}\,=\, W^{1-\theta,\frac{2}{\theta}}(\mathbb{T}^{d})\,,\quad \textrm{where}\quad\theta\,=\,\frac{2}{d+2}\,.
\end{equation}
Relation \eqref{prop:interp} follows from \cite[Theorem 2.3]{Cohen2003}, established by A. Cohen, which addresses interpolation with $ L^1 $ spaces. E. Curc\u{a} extended this result to our context using a duality argument. Specifically, as stated in \cite[Theorem 4]{Curca2024}, relation \eqref{prop:interp} is valid on $ \mathbb{R}^d $, which means that:
\begin{equation}\label{curca:interp}
\left(L^2\left(\R^d\right),W^{1,\infty}\left(\R^{d}\right)\right)_{1-\theta,\frac{2}{\theta}}\,=\, W^{1-\theta,\frac{2}{\theta}}\left(\R^{d}\right),\quad \textrm{where}\quad\theta\,=\,\frac{2}{d+2}\,,
\end{equation}
where we used, with the notation of \cite{Curca2024}, the relations 
\[
\begin{split}
F^{0,2}_2(\R^d) &= L^2(\R^d)\\ B^{1-\theta,\,2/\theta}_{2/\theta}(\R^d) &= W^{1-\theta,\,2/\theta}\left(\R^{d}\right), 
\end{split}
\] 
see \cite[(i) and (iii), Theorem $1.5.1$]{Triebel92}. The proof of \eqref{prop:interp} employs a technical truncation argument to reformulate \eqref{curca:interp} in the context of $\mathbb{T}^d$.
\begin{proof}[proof of \eqref{prop:interp}]
We choose a smooth function $\chi \in C_c^{\infty}(\mathbb{R}^d)$ such that 
\begin{equation}\label{chi:construct}
0\leq \chi \leq 1\quad ;\quad\chi=1\;\;\;\textrm{on}\;\; \left[-\frac{|\mathbb{T}|}{2},\frac{|\mathbb{T}|}{2}\right]^d
\quad ;
\quad\chi=0\;\;\;\textrm{on}\;\; \left(\left[-|\mathbb{T}|,|\mathbb{T}|\right]^d\right)^c
\,.
\end{equation}
For all $ L \in \left( L^2\left( \mathbb{T}^d \right), W^{1,\infty}\left( \mathbb{T}^d \right) \right)_{1-\theta,\, 2/\theta} \cap W^{1-\theta, 2/\theta}\left( \mathbb{T}^d \right) $, we denote by $\tilde{L}$ the product of $\chi$ and the periodic extension of $L$ to $\mathbb{R}^d$.

The main challenge is to establish the existence of a constant $ C > 0 $, which depends solely on the choice of $\chi$, such that: 
\[
\left\|L\right\|_{ \left(L^2\left(\mathbb{T}^d\right),W^{1,\infty}\left(\mathbb{T}^{d}\right)\right)_{1-\theta,\frac{2}{\theta}}}
\,\leq \,C\,
\|\Tilde{L}\|_{ \left(L^2\left(\R^d\right),W^{1,\infty}\left(\R^{d}\right)\right)_{1-\theta,\frac{2}{\theta}}}\,.
\]
To derive the latter estimate, we consider a pair $\left( \tilde{L}_1, \tilde{L}_2 \right)$ such that
\[
\Tilde{L}\,=\,\Tilde{L}_1+\Tilde{L}_2\,,\quad\textrm{and}\quad \Tilde{L}_1 \in 
W^{1,\infty}\left(\R^d\right)
\,,\quad\textrm{and}\quad \Tilde{L}_2 \in 
L^2\left(\R^d\right)\,.
\]
Then, we define:
\[
L_i(x)\,=\,
\left(\sum_{k\in\Z^d}
\chi\left(x+|\mathbb{T}|k\right)\right)^{-1}
\sum_{k\in\Z^d}
\Tilde{L}_i\left(x+|\mathbb{T}|k\right)
\chi\left(\frac{1}{2}(x+|\mathbb{T}|k)\right)\,,
\]
for all $x\in\R^d$, with $i\in\{1,2\}$. We verify that $ L_i $ for $ i \in \{1, 2\} $ defines a $ |\mathbb{T}| $-periodic function over $ \mathbb{R}^d $ and that 
\begin{equation}\label{interp:0}
L(x)\,=\,
L_1(x)+L_2(x)\,,\quad
 \forall x \in \left[-\frac{|\mathbb{T}|}{2},\frac{|\mathbb{T}|}{2}\right]^d\,.
\end{equation}
This relation guarantees that 
\begin{equation}\label{interp:1}
K(t,L)\,\leq
\left\|L_2\right\|_{L^{2}\left(\mathbb{T}^{d}\right)}
+
t
\left\|L_1\right\|_{W^{1,\infty}\left(\mathbb{T}^{d}\right)}
	\,,
\end{equation}
where the $K$ functional \cite[Section $1.6.2$]{Triebel92} is defined as:
\[
K(t,L)\,= \inf_{L=L_1^{'}+L_2^{'}}
\,\|L_2^{'}\|_{L^{2}\left(\mathbb{T}^{d}\right)}\,
+\,
t\,
\|L^{'}_1\|_{W^{1,\infty}\left(\mathbb{T}^{d}\right)}
\,,
\]
for all $t>0$. We can estimate the norms of $ L_1 $ and $ L_2 $ in \eqref{interp:1}. Utilizing the first and second properties in \eqref{chi:construct}, we have:
\[
\left|L_i(x)\right|\,\leq\,
\left|\sum_{k\in\Z^d}
\Tilde{L}_i\left(x+|\mathbb{T}|k\right)
\chi\left(\frac{1}{2}(x+|\mathbb{T}|k)\right)\right|,
\]
for all $x\in \R^d$, where $i\in\left\{1,2\right\}$. Due to the third property in \eqref{chi:construct}, only a finite number of terms in the sum on the right-hand side are non-zero. More specifically, we have:
\[
\left|L_i(x)\right|\,\leq\,
\left|\,\sum_{\substack{k\in\Z^d\\ |k|\leq 2}}
\Tilde{L}_i\left(x+|\mathbb{T}|k\right)\,
\right|,\quad \forall x \in \left[-\frac{|\mathbb{T}|}{2},\frac{|\mathbb{T}|}{2}\right]^d\,.
\]
By applying the triangle inequality in the previous estimate, we deduce:
\[
\left\|L_2\right\|_{L^2(\mathbb{T}^d)}\,\leq\,
C\,\|\Tilde{L}_2\|_{L^2(\R^d)}\,,
\quad\textrm{and}\quad
\left\|L_1\right\|_{L^\infty(\mathbb{T}^d)}\,\leq\,
C\,\|\Tilde{L}_1\|_{L^\infty(\R^d)}\,,
\]
for some constant depending only on the dimension $d$. We employ the same approach to estimate the $ L^\infty $-norm of $ \nabla_x L_1 $ and obtain:
\[
\left\|\nabla_x L_1\right\|_{L^\infty(\mathbb{T}^d)}\,\leq\,
C\,\|\Tilde{L}_1\|_{W^{1,\infty}(\R^d)}\,.
\]
We then utilize the two inequalities derived previously to estimate the right-hand side of \eqref{interp:1}, resulting in:
\[
K(t,L)\,\leq\,C\left(
\|\Tilde{L}_2\|_{L^{2}\left(\R^{d}\right)}\,
+
t\,\|\Tilde{L}_1\|_{W^{1,\infty}\left(\R^{d}\right)}\right)
\,,
\]
for some constant depending only on the dimension $d$. We take the infimum in the latter inequality over all pairs $\left(\Tilde{L}_1,\Tilde{L}_2\right)$ that satisfy \eqref{interp:0}, leading to:
\[
K(t,L)\,\leq\,C\,K(t,\Tilde{L})
\,,
\]
for all positive $t>0$. To conclude this step, we raise the inequality to the power $ \frac{2}{\theta} $, multiply by $ t^{-\frac{2(1-\theta)}{\theta} - 1} $, and integrate over all positive $ t > 0 $. This results in:
\begin{equation}\label{interp:Pi:R:theta}
\left\|L\right\|_{ \left(L^2\left(\mathbb{T}^d\right),W^{1,\infty}\left(\mathbb{T}^{d}\right)\right)_{1-\theta,\frac{2}{\theta}}}
\,\leq \,C
\,\|\Tilde{L}\|_{ \left(L^2\left(\R^d\right),W^{1,\infty}\left(\R^{d}\right)\right)_{1-\theta,\frac{2}{\theta}}}\,,
\end{equation}
for some constant $C$ depending only on the choice of $\chi$ and the dimension $d$.

Let us now derive the inverse inequality. For any pair $\left(L_1, L_2\right)$ that satisfies: 
\begin{equation}\label{interp:02}
L\,=\,L_1+L_2\,,\quad\textrm{and}\quad L_1 \in 
W^{1,\infty}\left(\mathbb{T}^d\right)
\,,\quad\textrm{and}\quad L_2 \in 
L^2\left(\mathbb{T}^d\right)\,,
\end{equation}
where we denote $\tilde{L}_i$,  with $i\in\left\{1,2\right\}$, the product between $\chi$ and the periodic extension of $L_i$ to $\R^d$. Thanks to the first and third properties in \eqref{chi:construct}, we can assert that:
\[
\|\Tilde{L}_2\|_{L^2(\R^d)}
\,\leq\,
C\,\left\|L_2\right\|_{L^2(\mathbb{T}^d)}\,,
\quad\textrm{and}\quad
\|\Tilde{L}_1\|_{W^{1,\infty}(\R^d)}
\,\leq\,
C\,
\left\|L_1\right\|_{W^{1,\infty}(\mathbb{T}^d)}\,,
\]
for some constant $C>0$ depending only on $\chi$. The latter inequality guarantees that:
\[
K(t,\Tilde{L})\,\leq\,C\left(
\left\|L_2\right\|_{L^{2}\left(\mathbb{T}^{d}\right)}
+
t
\left\|L_1\right\|_{W^{1,\infty}\left(\mathbb{T}^{d}\right)}\right)
\,,
\]
for all positive $t>0$. We now take the infimum of the latter inequality over all couples $(L_1, L_2)$ that satisfy \eqref{interp:02} and obtain:
\[
K(t,\Tilde{L})\,\leq\,C\,K(t,L)
\,.
\]
We conclude this step by raising the latter inequality to the power of $2/\theta$, multiplying by $t^{-2(1-\theta)/\theta - 1}$, and integrating over all positive $t > 0$. We obtain:
\begin{equation}\label{equi:Pi:R:interp}
\|\Tilde{L}\|_{ \left(L^2\left(\R^d\right),W^{1,\infty}\left(\R^{d}\right)\right)_{1-\theta,\frac{2}{\theta}}}
\,\leq \,C
\left\|L\right\|_{ \left(L^2\left(\mathbb{T}^d\right),W^{1,\infty}\left(\mathbb{T}^{d}\right)\right)_{1-\theta,\frac{2}{\theta}}}
,
\end{equation}
for some constant $C$ depending only on the choice of $\chi$ and the dimension $d$.

Next, we verify that there exists a constant $C$ that depends solely on the choice of $\chi$ such that
\begin{equation}\label{eq:norm}
\frac{1}{C}\,
\|\Tilde{L}\|_{W^{1-\theta,\frac{2}{\theta}}\left(\R^{d}\right)}
\,\leq\,
\left\|L\right\|_{W^{1-\theta,\frac{2}{\theta}}\left(\mathbb{T}^{d}\right)}
\,\leq \,C\,
\|\Tilde{L}\|_{W^{1-\theta,\frac{2}{\theta}}\left(\R^{d}\right)}\,.
\end{equation}

In the final step of this proof, we demonstrate that the $W^{1-\theta,2/\theta}$ and $\left(L^2,W^{1,\infty}\right)_{1-\theta,2/\theta}$ norms are equivalent on $\mathbb{T}^d$. On one hand, we utilize \eqref{eq:norm}, which guarantees that
\[
\left\|L\right\|_{W^{1-\theta,\frac{2}{\theta}}\left(\mathbb{T}^{d}\right)}
\,\leq \,C\,
\|\Tilde{L}\|_{W^{1-\theta,\frac{2}{\theta}}\left(\R^{d}\right)}\,.
\]
We then estimate the norm of $\Tilde{L}$ in the right-hand side using \eqref{curca:interp}, which guarantees that:
\[
\left\|L\right\|_{W^{1-\theta,\frac{2}{\theta}}\left(\mathbb{T}^{d}\right)}
\,\leq \,C\,
\|\Tilde{L}\|_{ \left(L^2\left(\R^d\right),W^{1,\infty}\left(\R^{d}\right)\right)_{1-\theta,\frac{2}{\theta}}}\,.
\]
We estimate the norm of $\Tilde{L}$ in the right-hand side by utilizing \eqref{equi:Pi:R:interp}, which gives us the following result: 
\begin{equation}\label{Interp:Pi:W:theta}
\left\|L\right\|_{W^{1-\theta,\frac{2}{\theta}}\left(\mathbb{T}^{d}\right)}
\,\leq \,C\,
\left\|L\right\|_{ \left(L^2\left(\mathbb{T}^d\right),W^{1,\infty}\left(\mathbb{T}^{d}\right)\right)_{1-\theta,\frac{2}{\theta}}}\,.
\end{equation}
We will now justify the inverse inequality. We start by applying \eqref{interp:Pi:R:theta}, which guarantees that:
\[
\left\|L\right\|_{ \left(L^2\left(\mathbb{T}^d\right),W^{1,\infty}\left(\mathbb{T}^{d}\right)\right)_{1-\theta,\frac{2}{\theta}}}
\,\leq \,C\,
\|\Tilde{L}\|_{ \left(L^2\left(\R^d\right),W^{1,\infty}\left(\R^{d}\right)\right)_{1-\theta,\frac{2}{\theta}}}\,.
\]
Then, we bound the norm of $\Tilde{L}$ in the previous right-hand side thanks to \eqref{curca:interp}, which ensures:
\[
\left\|L\right\|_{ \left(L^2\left(\mathbb{T}^d\right),W^{1,\infty}\left(\mathbb{T}^{d}\right)\right)_{1-\theta,\frac{2}{\theta}}}
\,\leq \,C\,
\|\Tilde{L}\|_{W^{1-\theta,\frac{2}{\theta}}\left(\R^{d}\right)}\,.
\]
We estimate the norm of $\Tilde{L}$ on the right-hand side using \eqref{eq:norm}, yielding: 
\begin{equation}\label{Interp:Pi:theta:W}
\left\|L\right\|_{ \left(L^2\left(\mathbb{T}^d\right),W^{1,\infty}\left(\mathbb{T}^{d}\right)\right)_{1-\theta,\frac{2}{\theta}}}
\,\leq \,C\,
\|L\|_{W^{1-\theta,\frac{2}{\theta}}\left(\mathbb{T}^{d}\right)}\,.
\end{equation}
Inequalities \eqref{Interp:Pi:W:theta} and \eqref{Interp:Pi:theta:W} together yield the desired result.
\end{proof}
\section{Well posedeness of the Liouville equation \eqref{Liouville}}\label{App:WP}
In this section, we establish the existence and uniqueness results for the Liouville equation \eqref{Liouville}, as stated in Theorems \ref{th:1} and \ref{th:2}. Subsequently, we provide a rigorous justification of the formal computations outlined in Sections \ref{sec:K:H-1} and \ref{sec:W:theta}, which are essential to deriving the uniform-in-$N$ estimates presented in Theorems \ref{th:1} and \ref{th:2}.
\subsection{Existence and uniqueness result}
Throughout this section, we emphasize that we work with a fixed $ N \geq 2 $ and do not aim to derive uniform estimates in $ N $. Our first goal is to demonstrate that, under the assumptions of either Theorem \ref{th:1} or Theorem \ref{th:2}, there exists a weak solution $ f_N \in L^\infty_{\text{loc}}\left(\mathbb{R}_+, L^2\left(\mathbb{T}^{dN}\right)\right) \cap L^2_{\text{loc}}\left(\mathbb{R}_+, H^1(\mathbb{T}^{dN})\right) $ to \eqref{Liouville}, in the sense of \eqref{Liouville:weak}. 
To construct such a solution, we employ a regularization procedure by introducing a sequence of mollified kernels $ \left(K_\varepsilon\right)_{\varepsilon > 0} $ defined as follows: 
\[
K^\eps\,=\,\rho_\eps\star K\,,
\]
for some regularizing sequence $\left(\rho_\eps\right)_{\eps>0}$ given by
\[
\rho^\eps(x)\,=\,\frac{1}{\eps^d}\rho\left(\frac{x}{\eps}\right)\,,\quad \forall x \in \T^d\,,
\]
where $\rho$ is a smooth positive function of $\R^d$ with integral $1$ over $\R^d$ and compact support within $\T^d$.
The key idea is to exploit the smoothing effects of the Laplace operator on the right-hand side of \eqref{Liouville} to establish $ L^2_{\text{loc}}\left(\mathbb{R}_+, H^1(\mathbb{T}^{dN})\right) $ estimates that are uniform in $ \varepsilon > 0 $ (though not in $ N \geq 1 $). This gain in $ H^1 $ regularity allows us to pass to the limit as $ \varepsilon \to 0 $. 
Next, we conclude this section by showing that, under the assumptions of either Theorem \ref{th:1} or Theorem \ref{th:2}, the solution $ f_N $ is unique in $ L^\infty_{\text{loc}}\left(\mathbb{R}_+, L^2\left(\mathbb{T}^{dN}\right)\right) \cap L^2_{\text{loc}}\left(\mathbb{R}_+, H^1(\mathbb{T}^{dN})\right) $. Here again, the regularity gain plays a pivotal role.
To unify the treatment of the frameworks of Theorems \ref{th:1} and \ref{th:2}, we adopt a set of weak assumptions. First, we assume that the initial data $ f_N^0 $ satisfies \eqref{hyp:f0} and that $ f_N^0 \in L^2\left(\mathbb{T}^{dN}\right) $. Second, we assume that $ K \in H^{-1}\left(\mathbb{T}^d\right) $, corresponding to the first condition in \eqref{hyp:K:sing}, and that $ K $ satisfies \eqref{K:2}. These assumptions ensure that $ K $ meets the requirements of Theorem \ref{th:1}, due to the Sobolev embedding $ W^{-\frac{2}{d+2}, d+2}\left(\mathbb{T}^d\right) \hookrightarrow H^{-1}\left(\mathbb{T}^d\right) $, as well as the divergence-free condition of Theorem \ref{th:2}, by setting $ K_- = K $ and $ K_+ = 0 $ in \eqref{K:2}.

We begin by proving the existence result. Given that $ f_N^0 \in L^2\left(\mathbb{T}^{dN}\right) $ and $ K^\varepsilon \in \mathcal{C}^\infty\left(\mathbb{T}^d\right) $ for all $ \varepsilon > 0 $, there exists a unique classical solution $ f^\varepsilon_N $ to the following mollified Liouville equation with the initial condition $ f_N^0 $
\begin{equation}\label{Liouville:eps}
	\partial_t f^\eps_N+\frac{1}{N} \sum_{\substack{i,j=1\\ i\neq j}}^N  \udiv_{x_i}\left(K^\eps\left(x_i-x_j\right) f^\eps_N\right)=\sigma \sum_{i=1}^N \Delta_{x_i}f^\eps_N\,.
\end{equation}
Since $ f^\varepsilon_N $ is sufficiently regular with respect to $ X^N \in \mathbb{T}^{dN} $ and $ t > 0 $, the following computations are valid in the strong sense. To estimate the $ L^2 $ norm of $ f^\varepsilon_N $, we multiply \eqref{Liouville:eps} by $ f^\varepsilon_N $, integrate over $ \mathbb{T}^{dN} $, and perform integration by parts with respect to $ x_i $, yielding: 
\begin{equation*}
	\frac{1}{2}\frac{\dD}{\dD t} \|f^\eps_N(t,\cdot)\|_{L^2(\mathbb{T}^{dN})}^2
	\,\leq \,
	\cK^\eps_- \,+\, \cK^\eps_+\,-\,
	\,\sigma\,\|\nabla_{X^N} f^\eps_N(t,\cdot)\|^2_{L^2(\mathbb{T}^{dN})}
	\,,
\end{equation*}
with $\cK^\eps_-$ and $\cK^\eps_+$ defined as: 
\[
\left\{
\begin{array}{ll}
	&\ds \cK^\eps_-\,=\,\frac{1}{N}\sum_{\substack{i,j=1\\ i\neq j}}^N\int_{\mathbb{T}^{dN}} K^\eps_-(x_i-x_{j})f^\eps_{N}\left(t,X^{N}\right)\nabla_{x_i} f^\eps_{N}(t,X^N)\,\dD X^N\,,\\[1.5em]
	&\ds \cK^\eps_+\,=\,\frac{1}{N}\sum_{\substack{i,j=1\\ i\neq j}}^N\int_{\mathbb{T}^{dN}} K^\eps_+(x_i-x_{j})f^\eps_{N}\left(t,X^{N}\right)\nabla_{x_i} f^\eps_{N}(t,X^N)\,\dD X^N \,,
\end{array}
\right.
\]
where $K^\eps_+$ and $K^\eps_-$ are given by 
 \[
 K^\eps_+\,=\,\rho_\eps\star K_+\,,\quad \textrm{and}\quad K^\eps_-\,=\,\rho_\eps\star K_-\,,
 \]
and $K_+$ and $K_-$ are given in \eqref{K:2}. We follow the same approach as in the proof of Lemma \ref{Lemma:non:hierarchy:term} to estimate $ \mathcal{K}^\varepsilon_+ $ and $ \mathcal{K}^\varepsilon_- $. Specifically, analogous computations to those used for estimating $ \mathcal{K}_- $ in the proof of Lemma \ref{Lemma:non:hierarchy:term} give: 
 \[
 \cK_-\,\leq\,
 \frac{N}{2}
 \left\|\left(\udiv_{x} K^\eps_-\right)_-\right\|_{L^\infty\left(\mathbb{T}^{d}\right)}
 \left\|f^\eps_{N}(t,\cdot)\right\|_{L^2\left(\mathbb{T}^{dN}\right)}^2\,.
 \]
 To estimate the $L^\infty$ norm in the latter inequality, we use assumption \eqref{K:-} on $K_-$ which guarantees the existence of a constant  $C>0$ such that:
 \[
 -\udiv_{x} \left(K_-\right)(y)\,\leq \, C \,,\quad \forall y \in \T^d\,.
 \]
 Then, we multiply the latter relation by $ \rho^\eps(x-y) $ and integrate with respect to $ y \in \T^d $. Noting that $ \rho^\eps $ is positive and has an integral equal to $ 1 $, this gives:  
 \[
 -\udiv_{x} \left(K^\eps_-\right)(x)\,\leq \, C \,,\quad \forall x \in \T^d\,,
 \]
 where $C$ is independent of $\eps$. Taking the supremum in $x$ and $\eps>0$ in the latter estimate, we deduce 
\[
\sup_{\eps>0} \left\|\left(\udiv_{x} K^\eps_-\right)_-\right\|_{L^\infty\left(\mathbb{T}^{d}\right)}\,<\,+\infty\,.
\]
Hence, there exists a constant $C$ independent of $\eps>0$ such that
\[
\cK_-\,\leq\,
C
\left\|f^\eps_{N}(t,\cdot)\right\|_{L^2\left(\mathbb{T}^{dN}\right)}^2\,.
\]
To estimate $ \cK_+^\eps $, we decompose the integral in the same way as in the proof of Lemma \ref{Lemma:non:hierarchy:term}, that is, 
\begin{equation*}
	\cK^\eps_+
	\,=\,
	\cR^\eps + \cS^\eps\,,
	\hspace{0.08cm}\textrm{with}\hspace{0.08cm}
	\left\{
	\begin{array}{ll}
		&\ds \cR^\eps\,=\,\frac{1}{N}\sum_{\substack{i,j=1\\ i\neq j}}^N \int_{\mathbb{T}^{dN}} R^\eps_\delta(x_i-x_{j})f^\eps_{N}\left(t,X^{N}\right)\nabla_{x_i} f^\eps_{N}(t,X^N)\,\dD X^N\\[1.5em]
		&\ds \cS^\eps\,=\,\frac{1}{N}\sum_{\substack{i,j=1\\ i\neq j}}^N\int_{\mathbb{T}^{dN}} S^\eps_\delta(x_i-x_{j})f^\eps_{N}\left(t,X^{N}\right)\nabla_{x_i} f^\eps_{N}(t,X^N)\,\dD X^N
	\end{array}
	\right.
\end{equation*}
where $R^\eps_\delta$ and $S^\eps_\delta$ are given by 
\[
R^\eps_\delta\,=\,\rho_\eps\star R_\delta\,,\quad \textrm{and}\quad S^\eps_\delta\,=\,\rho_\eps\star S_\delta\,,
\]
and where $\delta>0$, $R_\delta$ and $S_\delta$ are defined in \eqref{K:+:decomp}. To estimate $\cR^\eps$, we use the same computations as those used to justify \eqref{estim:R}, which yields
\begin{equation*}
	|\cR^\eps|
	\,\leq\,\frac{N}{2\,\delta}\left\|R^\eps_\delta\right\|^2_{L^{\infty}(\mathbb{T}^d)}\left\|f^\eps_{N}(t,\cdot)
	\right\|_{L^2(\mathbb{T}^{dN})}^{2}
	+
	\frac{\delta}{2}
	\left\|
	\nabla_{X^N}f^\eps_{N}(t,\cdot) 
	\right\|_{L^2(\mathbb{T}^{dN})}^{2}
	\,.
\end{equation*}
Next, we apply Young's inequality for convolutions to estimate the norm of $ R^\eps_\delta $, which gives:
\[
\left\|R^\eps_\delta\right\|_{L^{\infty}(\mathbb{T}^d)}
\,=\,
\left\|\rho^\eps \star R_\delta\right\|_{L^{\infty}(\mathbb{T}^d)}\leq
\left\|\rho^\eps\right\|_{L^{1}(\mathbb{T}^d)}
\left\|R_\delta\right\|_{L^{\infty}(\mathbb{T}^d)}
\leq
\left\|R_\delta\right\|_{L^{\infty}(\mathbb{T}^d)}
\,,
\]
where we used that $\rho^\eps$ is positive and has integral $1$ to obtain the last inequality. Hence, we deduce 
\begin{equation*}
	|\cR^\eps|
	\,\leq\,C_\delta\left\|f^\eps_{N}(t,\cdot)
	\right\|_{L^2(\mathbb{T}^{dN})}^{2}
	+
	\frac{\delta}{2}
	\left\|
	\nabla_{X^N}f^\eps_{N}(t,\cdot) 
	\right\|_{L^2(\mathbb{T}^{dN})}^{2}
	\,,
\end{equation*}
for some constant $C_\delta$ depending on $\delta$ but not on $\eps>0$.
To estimate $\cS^\eps$, we follow the same computations as those used to justify \eqref{estim:S}, which yields
\begin{equation*}
	|\cS^\eps|
	\,\leq\,C\left\|S_\delta^\eps\right\|_{L^{q}(\mathbb{T}^d)}
	\left(N
	\left\|f^\eps_{N}(t,\cdot)
	\right\|_{L^2(\mathbb{T}^{dN})}^2
	+
	\left\|
	\nabla_{X_N}f^\eps_{N}(t,\cdot)
	\right\|_{L^2(\mathbb{T}^{dN})}^2\right)
	\,.
\end{equation*}
Next, we use Young's  inequality for convolution to estimate the norm of $S^\eps_\delta$, that is
\[
\left\|S^\eps_\delta\right\|_{L^{q}(\mathbb{T}^d)}
\,=\,
\left\|\rho^\eps \star S_\delta\right\|_{L^{q}(\mathbb{T}^d)}\leq
\left\|\rho^\eps\right\|_{L^{1}(\mathbb{T}^d)}
\left\|S_\delta\right\|_{L^{q}(\mathbb{T}^d)}
\leq
\left\|S_\delta\right\|_{L^{q}(\mathbb{T}^d)}
\,.
\]
Since the norm of $S_\delta$ is less than $\delta$ according to \eqref{K:+:decomp}, we deduce that
\begin{equation*}
	|\cS^\eps|
	\,\leq\,C\delta\left(\left\|f^\eps_{N}(t,\cdot)
	\right\|_{L^2(\mathbb{T}^{dN})}^{2}
	+
	\left\|
	\nabla_{X^N}f^\eps_{N}(t,\cdot) 
	\right\|_{L^2(\mathbb{T}^{dN})}^{2}\right)
	\,,
\end{equation*}
for some constant $C$ independent of $\delta>0$ and $\eps>0$. Combining our estimates on $\cR^\eps$ and $\cS^\eps$, we deduce the following bound for $\cK_+^\eps$
\begin{equation*}
	|\cK^\eps_+|
	\,\leq\,C_\delta\left\|f^\eps_{N}(t,\cdot)
	\right\|_{L^2(\mathbb{T}^{dN})}^{2}
	+
	C\delta
	\left\|
	\nabla_{X^N}f^\eps_{N}(t,\cdot) 
	\right\|_{L^2(\mathbb{T}^{dN})}^{2}
	\,.
\end{equation*}
Gathering our estimates on $\cK^\eps_+$ and $\cK^\eps_-$, we obtain the following differential inequality for the $L^2$ norm of $f^\eps_N$
\begin{equation*}
	\frac{1}{2}\frac{\dD}{\dD t} \|f^\eps_N(t,\cdot)\|_{L^2(\mathbb{T}^{dN})}^2
	\,\leq \,
	C_\delta\left\|f^\eps_{N}(t,\cdot)
	\right\|_{L^2(\mathbb{T}^{dN})}^{2}
	+
	(C\delta-\sigma)
	\left\|
	\nabla_{X^N}f^\eps_{N}(t,\cdot) 
	\right\|_{L^2(\mathbb{T}^{dN})}^{2}
	\,.
\end{equation*}
We choose $\delta = \sigma/(2C)$ in the latter inequality and obtain 
\begin{equation*}
	\frac{1}{2}\frac{\dD}{\dD t} \|f^\eps_N(t,\cdot)\|_{L^2(\mathbb{T}^{dN})}^2
	\,\leq \,
	C\left\|f^\eps_N(t,\cdot)
	\right\|_{L^2(\mathbb{T}^{dN})}^{2}
-\frac{\sigma}{2}
	\left\|
	\nabla_{X^N}f^\eps_N(t,\cdot) 
	\right\|_{L^2(\mathbb{T}^{dN})}^{2}
	\,,
\end{equation*}
for some constant $C>0$ independent of $\eps$. We multiply the latter relation by $e^{-2Ct}$ and integrate from $0$ to $t$, obtaining
\begin{equation*}
		\|f_{N}^\eps(t,\cdot)\|_{L^2(\mathbb{T}^{dN})}^{2}
		+
		\sigma 
		\int_0^t
		\|\nabla_{X^N} f^\eps_N(s,\cdot)\|_{L^2(\mathbb{T}^{dN})}^{2}\,\dD s
		\,\leq\,e^{2Ct}
		\|f_{N}^0\|_{L^2(\mathbb{T}^{dN})}^{2} \,,
\end{equation*}
for some large enough constant $C>0$ independent of $\eps$. Since $f^0_N$ belongs to $L^2$ and does not depend on $\eps>0$,  the previous estimate can be rewritten as 
\begin{equation}\label{estim:final:eps}
	\|f_{N}^\eps(t,\cdot)\|_{L^2(\mathbb{T}^{dN})}^{2}
	+
	\sigma
	\int_0^t
	\|\nabla_{X^N} f^\eps_N(s,\cdot)\|_{L^2(\mathbb{T}^{dN})}^{2}\,\dD s
	\,\leq\,C\,e^{Ct} \,,
\end{equation}
for all time $t\geq 0$, all $\eps>0$, and for some large enough constant $C>0$ independent of $\eps$. According to \eqref{estim:final:eps}, the sequence of regularized solutions $\left(f^\eps_N,\nabla_{X^N}f^\eps_N\right)_{\eps>0}$ is uniformly bounded in $L_{loc}^{\infty}\left(\R_+,L^2\left(\T^{dN}\right)\right)\times L_{loc}^{2}\left(\R_+\times\T^{dN}\right)$. Hence, it converges up to subsequence in $L_{loc}^{\infty}\left(\R_+,L^2\left(\T^{dN}\right)\right)\times L_{loc}^{2}\left(\R_+\times\T^{dN}\right)-weak*$ to a limit $\left(f_N,\nabla_{X^N}f_N\right)\in L_{loc}^{\infty}\left(\R_+,L^2\left(\T^{dN}\right)\right)\times L_{loc}^{2}\left(\R_+\times\T^{dN}\right)$ as $\eps\rightarrow 0$.

To justify that $ f_N $ is a weak solution to the Liouville equation \eqref{Liouville}, we fix an $ \eps > 0 $ and a compactly supported test function $ \varphi \in \mathcal{C}^{\infty}\left(\mathbb{R}^+ \times \mathbb{T}^{dN}\right) $. We then multiply equation \eqref{Liouville:eps} by $ \varphi $ and integrate over the region $ [0,t] \times \mathbb{T}^{dN} $, for some $ t > 0 $. After performing integration by parts with respect to both $ t $ and $ X^N $, we obtain
\begin{equation}\label{Liouville:weak:eps}
	\begin{split}
		\int_{\T^{dN}}
		f^\eps_N& 	\varphi\,
		\dD X^N=
		\int_{\T^{dN}}
		f^0_N 	\varphi^0
		\dD X^N\\
		+&\int_{0}^t\int_{\T^{dN}}
		f^\eps_N 			\partial_t\varphi
		+
		\frac{1}{N} \sum_{\substack{i,j=1\\ i\neq j}}^N  K^\eps\left(x_i-x_j\right) f^\eps_N\nabla_{x_i}\varphi+
		\sigma \sum_{i=1}^N f^\eps_N \Delta_{x_i}\varphi \,\dD X^N\dD s\,,
	\end{split}
\end{equation}
where $\varphi^0(\cdot)=\varphi(0,\cdot)$. Let us now check that the interaction term $\cA$ defined as 
\[
\cA\,=\,
\frac{1}{N} \sum_{\substack{i,j=1\\ i\neq j}}^N \int_{0}^t\int_{\T^{dN}}
 K^\eps\left(x_i-x_j\right) f^\eps_N\nabla_{x_i}\varphi \,\dD X^N\dD s\,,
\]
converges as $\eps\rightarrow 0$. We fix some $\phi \in L^2\left(\T^d\right)^{d\times d}$ such that $K=\udiv_{x} \phi$ and replace $K^\eps$ in $\cA$:
\[
\cA\,=\,
\frac{1}{N} \sum_{\substack{i,j=1\\ i\neq j}}^N\int_{0}^t\int_{\T^{dN}}
  \udiv_{x_i} \left(\rho^\eps\star\phi\right)\left(x_i-x_j\right) f^\eps_N\nabla_{x_i}\varphi \,\dD X^N\dD s\,.
\]
Then, we integrate by parts with respect to $x_i$ in the latter relation, which yields
\[
\cA\,=\,\cA_1+\cA_2
\,,
\]
where $\cA_1$ and $\cA_2$ are defined as
	\begin{equation*}
		\left\{
		\begin{array}{ll}
			&\ds\cA_1\,=\,-\,	\frac{1}{N} \sum_{\substack{i,j=1\\ i\neq j}}^N\int_{0}^t\int_{\T^{dN}}
		  (\rho^\eps\star\phi\left(x_i-x_j\right) \nabla_{x_i}f^\eps_N)\cdot \nabla_{x_i}\varphi  \,\dD X^N\dD s\,,  \\[1.em]
			&\ds\cA_2\,=\,-\,\frac{1}{N} \sum_{\substack{i,j=1\\ i\neq j}}^N \int_{0}^t\int_{\T^{dN}}
			\operatorname{tr}( \rho^\eps\star\phi\left(x_i-x_j\right) \nabla^2_{x_i}\varphi)\, f^\eps_N\, \dD X^N\dD s\,,
		\end{array}
		\right.
\end{equation*}
where $\operatorname{tr}(A)$ denotes the trace of the matrix $A$. \\
On the one hand, we have 
\begin{equation*}
	\left\{
	\begin{array}{ll}
		&\ds\rho^\eps\star\phi\left(x_i-x_j\right)\,\underset{\eps\rightarrow0}{\longrightarrow}\, \phi\left(x_i-x_j\right)\,,\quad \textrm{in}\quad L^{2}_{loc}\left(\Theta\right)\,,  \\[1.em]
		&\ds\nabla_{x_i}f^\eps_N\,\underset{\eps\rightarrow0}{\longrightarrow}\, \nabla_{x_i}f_N\,,\quad \textrm{in}\quad L^{2}_{loc}\left(\Theta\right)-weak\,,
	\end{array}
	\right.
	\end{equation*}
where $\Theta$ is the support of $\varphi$ in $\R_+\times\T^{dN}$.
This implies that
	\[
	\cA_1\underset{\eps\rightarrow0}{\longrightarrow}\, 
	-\,\frac{1}{N} \sum_{\substack{i,j=1\\ i\neq j}}^N\int_{0}^t\int_{\T^{dN}}	  
	(\phi\left(x_i-x_j\right) \nabla_{x_i}f_N) \cdot \nabla_{x_i}\varphi  \,\dD X^N\dD s\,.
	\]
On the other hand, we have 
	\begin{equation*}
		\left\{
		\begin{array}{ll}
			&\ds\rho^\eps\star\phi\left(x_i-x_j\right)\,\underset{\eps\rightarrow0}{\longrightarrow}\, \phi\left(x_i-x_j\right)\,,\quad \textrm{in}\quad L^{2}\left(\Theta\right)\,,  \\[1.em]
			&\ds f^\eps_N\,\underset{\eps\rightarrow0}{\longrightarrow}\, f_N\,,\quad \textrm{in}\quad L^{2}\left(\Theta\right)-weak\,.
		\end{array}
		\right.
\end{equation*}
These convergences ensure
	\[
	\cA_2\,\underset{\eps\rightarrow0}{\longrightarrow}\, 
	-\,\frac{1}{N} \sum_{\substack{i,j=1\\ i\neq j}}^N\int_{0}^t\int_{\T^{dN}}
	  \operatorname{tr}(\phi\left(x_i-x_j\right) \nabla^2_{x_i}\varphi)\, f_N\, \dD X^N\dD s\,.
	\]

Based on the previous computations, we can take the limit as $ \eps \to 0 $ in the weak formulation \eqref{Liouville:weak:eps}, and deduce that $ f_N $ is a weak solution to the Liouville equation \eqref{Liouville}, that is,
		\begin{equation}\label{Liouville:weak}
	\begin{split}
		&\int_{\T^{dN}}
		f_N \varphi
		\dD X^N=
		\int_{\T^{dN}}
		f^0_N 	\varphi^0\,
		\dD X^N +\int_{0}^t\int_{\T^{dN}}
		f_N 			\partial_t\varphi \,\dD X^N\dD s\\
		&-\frac{1}{N} \sum_{\substack{i,j=1\\ i\neq j}}^N  \int_{0}^t\int_{\T^{dN}}
		(\phi\left(x_i-x_j\right)\nabla_{x_i}f_N)\cdot\nabla_{x_i}\varphi   + \operatorname{tr}\Big(\phi\left(x_i-x_j\right)\nabla^2_{x_i}\varphi\Big)\, f_N\dD X^N\dD s\\
		&+\sigma\sum_{i=1}^N \int_{0}^t\int_{\T^{dN}} f_N \Delta_{x_i}\varphi \,\dD X^N\dD s\,,
	\end{split}
\end{equation}
for all time $t>0$, all compactly supported test function $\varphi\in \scC^{\infty}\left(\R^+\times \T^{dN}\right)$ and all $\phi \in L^2\left(\T^d\right)^{d\times d}$ such that $K=\udiv_{x} \phi$.

To conclude, we address the problem of uniqueness of weak solutions to the Liouville equation \eqref{Liouville}. We consider two solutions 
\[
\left(f_N,g_N\right)\in L^{\infty}\left(\R_+,L^2\left(\T^{dN}\right)\right)\cap L^{2}\left(\R_+\times H^1\left(\T^{dN}\right)\right)
\]
to \eqref{Liouville:weak} with the same initial data $f_N^0$ and a smooth compactly supported vector field $\varphi\in \scC^{\infty}\left(\R^+ \times \T^{dN}\right)^{dN}$. Since $f_N$ and $g_N$ have the same initial data, the difference between the weak formulations \eqref{Liouville:weak} on $f_N$ and $g_N$ applied to $\udiv_{X^N} \varphi$ simplifies into
	\begin{equation*}
		\begin{split}
		\int_{\T^{dN}}
			(f_N-&g_N)\udiv_{X^N} \varphi\,
			\dD X^N= \int_{0}^t\int_{\T^{dN}}
			(f_N-g_N) 			\partial_t\udiv_{X^N} \varphi\,\dD X^N\dD s\\
			&-\frac{1}{N} \sum_{\substack{i\neq j}}^N \int_{0}^t\int_{\T^{dN}}
			 (\phi\left(x_i-x_j\right)\nabla_{x_i}(f_N-g_N))\cdot \nabla_{x_i}\udiv_{X^N} \varphi  dX^Nds \\
			&-\frac{1}{N} \sum_{\substack{i\neq j}}^N \int_{0}^t\int_{\T^{dN}}
			\operatorname{tr}(\phi\left(x_i-x_j\right)\nabla^2_{x_i}\udiv_{X^N} \varphi\, )(f_N-g_N)dX^Nds\\
			&+\sigma \sum_{i=1}^N\int_{0}^t\int_{\T^{dN}}
			 (f_N-g_N)\Delta_{x_i}\udiv_{X^N} \varphi \,\dD X^N\dD s\,.
		\end{split}
\end{equation*}
We apply the Cauchy-Schwarz inequality to estimate the right-hand side, which yields
	\begin{equation*}
		\begin{split}
			\int_{\T^{dN}}
			(f_N&-g_N)\udiv_{X^N} \varphi\,
			\dD X^N \leq \\
			&C_N\left\|\varphi\right\|_{W^{3,\infty}}(1+\left\|\phi\right\|_{L^{2}})\int_{0}^t
			\left\|(f_N-g_N)(s,\cdot)\right\|_{L^{2}}+ \left\|\nabla_{X^N}\left(f_N-g_N\right)(s,\cdot)\right\|_{L^{2}}\dD s.
		\end{split}
\end{equation*}
Since the difference $f_N-g_N$ has zero mean over $\T^{dN}$ for all time $s>0$, the Poincar\'e inequality  gives
\[
\left\|(f_N-g_N)(s,\cdot)\right\|_{L^{2}}\leq \left\|\nabla_{X^N}\left(f_N-g_N\right)(s,\cdot)\right\|_{L^{2}}\,.
\]
Hence, we obtain 
	\begin{equation*}
		\begin{split}
			\int_{\T^{dN}}
			\left(f_N-g_N\right)\udiv_{X^N} \varphi\,
			\dD X^N \leq 
			C_{N,\phi}\left\|\varphi\right\|_{W^{3,\infty}}\int_{0}^t
			\left\|\nabla_{X^N}\left(f_N-g_N\right)(s,\cdot)\right\|_{L^{2}}\dD s\,.
		\end{split}
\end{equation*}
Then, we integrate the latter inequality with respect to time, obtaining
	\begin{equation}\label{Gronw:WP}
		\begin{split}
			\int_{0}^T
			\int_{\T^{dN}}
			\left(f_N-g_N\right)&\udiv_{X^N} \varphi\,
			\dD X^N \,\dD t\,\leq \\
			&C_{N,\phi}\left\|\varphi\right\|_{W^{3,\infty}}\int_{0}^T\int_{0}^t
			\left\|\nabla_{X^N}\left(f_N-g_N\right)(s,\cdot)\right\|_{L^{2}}\dD s\,\dD t\,,
		\end{split}
\end{equation}
for all $T>0$. We lower bound the left hand side in \eqref{Gronw:WP} using that $\nabla_{X^N}(f_N-g_N)$ belongs to $L^{2}\left(\R_+\times \T^{dN}\right)^{dN}$ and the density of $\scC^{\infty}\left(\R^+ \times \T^{dN}\right)^{dN}$ in $L^{2}\left(\R_+\times \T^{dN}\right)^{dN}$. Indeed, this implies that there exists $\varphi$ such that 
\[
\frac{1}{2}
\int_{0}^T
\left\|\nabla_{X^N}\left(f_N-g_N\right)(t,\cdot)\right\|_{L^{2}}^2 \,\dD t
\,\leq \,
\int_{0}^T
\int_{\T^{dN}}
\left(f_N-g_N\right)\udiv_{X^N} \varphi\,
\dD X^N \,\dD t\,.
\]
We fix such $\varphi$ in \eqref{Gronw:WP} and obtain
\begin{equation*}
		\int_{0}^T
		\left\|\nabla_{X^N}\left(f_N-g_N\right)(t,\cdot)\right\|_{L^{2}}^2 \,\dD t \leq
		C_{N,\varphi,\phi}\int_{0}^T\int_{0}^t
		\left\|\nabla_{X^N}\left(f_N-g_N\right)(s,\cdot)\right\|_{L^{2}}\dD s\,\dD t\,.
\end{equation*}
Then, we lower bound the latter left hand side thanks to Jensen inequality, which yields
	\begin{equation}\label{Grow:2:WP}
		\begin{split}
		T^{-1}
		\Big(\int_{0}^T
		\left\|\nabla_{X^N}\left(f_N-g_N\right)(t,\cdot)\right\|_{L^{2}}&\dD t\Big)^2\\
		\leq
		&C_{N,\varphi,\phi}\int_{0}^T\int_{0}^t
		\left\|\nabla_{X^N}\left(f_N-g_N\right)(s,\cdot)\right\|_{L^{2}}\dD s\,\dD t\,.
	\end{split}
\end{equation}
Next, we suppose by contradiction that $f_N\,\neq\,g_N$. This implies that there exists $T>0$ and $\eta>0$ such that 
	\begin{equation}\label{contradiction}
		\eta<\int_{0}^T
		\left\|\nabla_{X^N}\left(f_N-g_N\right)(t,\cdot)\right\|_{L^{2}}\dD t\,.
\end{equation}
We multiply \eqref{Grow:2:WP} by $T/\eta$ and use the latter relation to deduce
	\begin{equation*}\int_{0}^T
		\left\|\nabla_{X^N}\left(f_N-g_N\right)(t,\cdot)\right\|_{L^{2}}\dD t
		\leq
		C_{N,\varphi,\phi,T,\eta}\int_{0}^T\int_{0}^t
		\left\|\nabla_{X^N}\left(f_N-g_N\right)(s,\cdot)\right\|_{L^{2}}\dD s\,\dD t\,.
\end{equation*}
We apply Gr{\"o}nwall Lemma in the latter inequality, which ensures
	\begin{equation*}\int_{0}^T
		\left\|\nabla_{X^N}\left(f_N-g_N\right)(t,\cdot)\right\|_{L^{2}}\dD t
		\,=\,0\,,
\end{equation*}
and therefore contradicts \eqref{contradiction}.

\subsection{Rigorous justification of the computations in Sections \ref{sec:K:H-1} and \ref{sec:W:theta}}
In this section, we rigorously justify the formal computations presented in Sections \ref{sec:K:H-1} and \ref{sec:W:theta}, thereby validating the estimates established in Theorems \ref{th:1} and \ref{th:2}.

We begin by assuming that $ K $ and $ \left(f^0_N\right)_{N \geq 1} $ satisfy the hypotheses of Theorem \ref{th:2}, and we proceed to justify the estimate \eqref{SmallnessH-1}. To this end, we consider the sequence of solutions $ \left(f^\eps_N\right)_{\eps > 0} $ to the mollified Liouville equation \eqref{Liouville:eps}, with initial data $ \left(f^0_N\right)_{N \geq 1} $ and mollified kernel $ K^\eps $. We verify that the regularized kernels $ K_\eps $ satisfy the assumption \eqref{hyp:K:sing} for each $ \eps > 0 $. Additionally, the regularity of the solution $ f^\eps_N $ with respect to both $ X^N \in \mathbb{T}^{dN} $ and $ t > 0 $ ensures that all the computations presented in Section \ref{sec:K:H-1} are rigorously justified for $ f^\eps_N $, and in particular, the following bound is verified:
\begin{equation*}
	\frac{\dD}{\dD t} \sum_{k=1}^N\frac{\|f^\eps_k(t,\cdot)\|_{L^2(\mathbb{T}^{dk})}^2}{R^{2k}}
	\,\leq \,
	\sum_{k=1}^{N} \left(\frac{\|K^\eps\|^2_{H^{-1}(\mathbb{T}^{d})}}{\sigma R^{2(k-1)}}-\frac{\sigma}{R^{2k}}\right)\|\nabla_{X^k} f^\eps_k(t,\cdot)\|^2_{L^2(\mathbb{T}^{dk})}
	\,,
\end{equation*}
for all time $t>0$, where $R$ is given in \eqref{SmallnessH-1}. To estimate the latter right-hand side, we decompose the sum as follows
\begin{equation*}
	\begin{split}
		\frac{\dD}{\dD t} \sum_{k=1}^N\frac{\|f^\eps_k(t,\cdot)\|_{L^2(\mathbb{T}^{dk})}^2}{R^{2k}}
		\,\leq \,
		&\sum_{k=1}^{N}\left(\|K^\eps\|_{H^{-1}(\mathbb{T}^{d})}^2 - \|K\|_{H^{-1}(\mathbb{T}^{d})}^2\right)
		\frac{\|\nabla_{X^k} f^\eps_k(t,\cdot)\|^2_{L^2(\mathbb{T}^{dk})}}{\sigma R^{2(k-1)}}\\ 
		&+\sum_{k=1}^{N} \left(\frac{ \|K\|_{H^{-1}(\mathbb{T}^{d})}^2}{\sigma R^{2(k-1)}}-\frac{\sigma}{R^{2k}}\right)\|\nabla_{X^k} f^\eps_k(t,\cdot)\|^2_{L^2(\mathbb{T}^{dk})}
		\,.
\end{split}
\end{equation*}
Under the assumption of Theorem \ref{th:2}, we have that $\sigma\geq \|K\|_{H^{-1}}^2R^2/\sigma$, which ensures that  the last sum in the latter estimate is non positive. Hence we obtain
	\begin{equation*}
		\frac{\dD}{\dD t} \sum_{k=1}^N\frac{\|f^\eps_k(t,\cdot)\|_{L^2(\mathbb{T}^{dk})}^2}{R^{2k}}
		\,\leq \left(\|K^\eps\|_{H^{-1}(\mathbb{T}^{d})}^2 - \|K\|_{H^{-1}(\mathbb{T}^{d})}^2\right)\sum_{k=1}^{N}
		\frac{\|\nabla_{X^k} f^\eps_k(t,\cdot)\|^2_{L^2(\mathbb{T}^{dk})}}{\sigma R^{2(k-1)}}
		\,.
\end{equation*}
We integrate the latter relation from $0$ to $t$, obtaining
	\begin{equation*}
	\begin{split}
		\sum_{k=1}^N\frac{\|f^\eps_k(t,\cdot)\|_{L^2(\mathbb{T}^{dk})}^2}{R^{2k}}
		\,\leq \,
		&\left(\|K^\eps\|_{H^{-1}(\mathbb{T}^{d})}^2 - \|K\|_{H^{-1}(\mathbb{T}^{d})}^2\right)\sum_{k=1}^{N}
		\int_0^t
		\frac{\|\nabla_{X^k} f^\eps_k(s,\cdot)\|^2_{L^2(\mathbb{T}^{dk})}}{\sigma R^{2(k-1)}}
		\,\dD s
		\\ 
		&+\sum_{k=1}^N\frac{\|f_{k}^0\|_{L^2(\mathbb{T}^{dk})}^2}{R^{2k}} \,.
\end{split}
\end{equation*}
Next, we bound the last sum thanks to assumption \eqref{SmallnessH-1}, which yields
	\begin{equation*}
		\sum_{k=1}^N\frac{\|f^\eps_k(t,\cdot)\|_{L^2(\mathbb{T}^{dk})}^2}{R^{2k}}
		\,\leq \,
		\left(\|K^\eps\|_{H^{-1}(\mathbb{T}^{d})}^2 - \|K\|_{H^{-1}(\mathbb{T}^{d})}^2\right)\sum_{k=1}^{N}
		\int_0^t
		\frac{\|\nabla_{X^k} f^\eps_k(s,\cdot)\|^2_{L^2(\mathbb{T}^{dk})}}{\sigma R^{2(k-1)}}
		\dD s+\,C^2\,,
\end{equation*}
for all time $t\geq 0$ and $\eps>0$, and where $C$ and $R$ are given in \eqref{SmallnessH-1}. On the one hand, we estimate the sum on the latter right hand side using estimate \eqref{estim:final:eps}, which ensures that the sequence of regularized gradients $\left(\nabla_{X^N}f^\eps_N\right)_{\eps>0}$ is uniformly bounded in $ L^{2}_{loc}\left(\R_+\times\T^{dN}\right)$. Hence, there exists a positive constant $C_N>0$ depending on $N$ and $t$ but not on $\eps>0$ such that 
	\begin{equation}\label{estim:final:eps2}
		\sum_{k=1}^N\frac{\|f^\eps_k(t,\cdot)\|_{L^2(\mathbb{T}^{dk})}^2}{R^{2k}}
		\,\leq \,
		C_N
		\left(\|K^\eps\|_{H^{-1}(\mathbb{T}^{d})}^2 - \|K\|_{H^{-1}(\mathbb{T}^{d})}^2\right)
		+\,C^2\,.
\end{equation}
On the other hand, we rewrite the error between $K$ and $K^\eps$ thanks to the following relation
\[
\|K^\eps\|_{H^{-1}}^2 - \|K\|_{H^{-1}(\mathbb{T}^{d})}^2
\,=\,
\left(\|K^\eps\|_{H^{-1}} - \|K\|_{H^{-1}(\mathbb{T}^{d})}\right)
\left(\|K^\eps\|_{H^{-1}} + \|K\|_{H^{-1}(\mathbb{T}^{d})}\right),
\]
and then apply the triangular inequality on the right hand side, which yields
\[
\|K^\eps\|_{H^{-1}}^2 - \|K\|_{H^{-1}(\mathbb{T}^{d})}^2
\,\leq\,
\|K^\eps - K\|_{H^{-1}(\mathbb{T}^{d})}
\left(\|K^\eps\|_{H^{-1}} + \|K\|_{H^{-1}(\mathbb{T}^{d})}\right).
\]
According to the definition of the $H^{-1}$ norm below \eqref{hyp:K:sing}, the latter right hand side is bounded by
\[
\|K^\eps\|_{H^{-1}}^2 - \|K\|_{H^{-1}(\mathbb{T}^{d})}^2
\,\leq\,
\left\|\rho_\eps\star\phi-\phi\right\|_{L^{2}}
\left(\left\|\rho_\eps\star\phi\right\|_{L^{2}} + \|K\|_{H^{-1}}\right),
\]
for all $\phi \in L^2\left(\T^d\right)^{d\times d}$ such that $K=\udiv_{x} \phi$. We fix $\phi \in L^2\left(\T^d\right)^{d\times d}$ and point out that $\left\|\phi-\rho_\eps\star\phi\right\|_{L^{2}}\rightarrow 0$ as $\eps\rightarrow 0$, which yields
\[
\limsup_{\eps\rightarrow 0}
\left(
\|K^\eps\|_{H^{-1}}^2 - \|K\|_{H^{-1}(\mathbb{T}^{d})}^2\right)\leq 0\,,
\]
taking the $\liminf$ in $\eps$ in \eqref{estim:final:eps2} and using the latter result, we deduce
	\begin{equation*}
		\liminf_{\eps\rightarrow 0}\;\sum_{k=1}^N\frac{\|f^\eps_k(t,\cdot)\|_{L^2(\mathbb{T}^{dk})}^2}{R^{2k}}
		\,\leq \,
		C^2\,.
\end{equation*}
Based on the previous section, the sequence of regularized solutions $\left(f^\eps_N,\nabla_{X^N}f^\eps_N\right)_{\eps>0}$ is uniformly bounded in $L^{\infty}\left(\R_+,L^2\left(\T^{dN}\right)\right)\times L^{2}\left(\R_+\times\T^{dN}\right)$ and it converges up to subsequence in $L^{\infty}\left(\R_+,L^2\left(\T^{dN}\right)\right)\times L^{2}\left(\R_+\times\T^{dN}\right)-weak*$ to the unique weak solution $\left(f_N,\nabla_{X^N}f_N\right)\in L^{\infty}\left(\R_+,L^2\left(\T^{dN}\right)\right)\times L^{2}\left(\R_+\times\T^{dN}\right)$ of \eqref{Liouville} as $\eps\rightarrow 0$. Therefore, we may pass to the limit $\eps \rightarrow 0$ in the latter estimate, that is,
\begin{equation*}
	\sum_{k=1}^N\frac{\|f_{k}(t,\cdot)\|_{L^2}^2}{R^{2k}}
	\,\leq\,C^2\,,
\end{equation*}
for all time $t\geq 0$, where $C$ is given in \eqref{SmallnessH-1}.  As in the proof of Theorem \ref{th:2}, we fix some $k$ between $1$ and $N$, lower bound the latter left hand side by $\|f_{k}(t,\cdot)\|_{L^2}^2/R^{2k}$,  and take the square root on both sides of the inequality, which justifies the estimate in Theorem \ref{th:2}, that is,
\begin{equation*}
	\frac{\|f_{k}(t,\cdot)\|_{L^2}}{R^{k}}
	\,\leq\,C\,, \quad \forall t\in \R^+\,,\quad \forall k \in \left\{1,\dots, N\right\}\,.
\end{equation*}

The key point in the argument above is the uniform estimate in $ \eps $ given by \eqref{estim:final:eps}, along with the fact that the constants depending on $ K $ that appear in the computations from Section \ref{sec:K:H-1} are continuous with respect to the limit $ \eps \to 0 $ in our regularization procedure. Consequently, the same method can be applied to justify the estimate \eqref{MarginalsBound} in Theorem \ref{th:1}, provided that the constants depending on $ K $ in the computations of Section \ref{sec:W:theta} are also continuous with respect to $ \eps \to 0 $. For brevity, we omit the detailed computations.

\bibliographystyle{abbrv}
\bibliography{refer}

\end{document}